\documentclass[12pt]{report}

\usepackage{tikz}
	\usetikzlibrary{3d}
\usepackage{amsmath,amssymb,amsfonts,amscd, array,amsthm}
\usepackage{caption}
\usepackage[labelformat=simple,labelfont={}]{subcaption}
	
\usepackage{setspace}
\usepackage{fancyhdr}
\usepackage{fullpage}
\usepackage{url}
\usepackage{xcolor}
\usepackage{enumerate}
\usepackage{enumitem}

\newtheorem{theorem}{Theorem}[section] 
\newtheorem{lemma}[theorem]{Lemma}
\newtheorem{proposition}[theorem]{Proposition}
\newtheorem{corollary}[theorem]{Corollary}

\theoremstyle{definition}
\newtheorem{definition}[theorem]{Definition}
\newtheorem{example}[theorem]{Example}

\newtheorem{remark}[theorem]{Remark}

\renewcommand{\H}{\mathcal{H}}
\DeclareMathOperator{\TL}{TL}
\DeclareMathOperator{\PFTL}{\widehat{TL}}
\newcommand{\Z}{\mathbb{Z}}
\newcommand{\N}{\mathbb{N}}
\newcommand{\A}{\mathcal{A}}
\newcommand{\C}{\widetilde{C}}

\newcommand{\x}{\mathsf{x}}

\renewcommand{\a}{\mathbf{a}}

\DeclareMathOperator{\DTL}{\mathbb{D}TL}
\renewcommand{\P}{\mathcal{P}}

\DeclareMathOperator{\LFD}{\mathbb{D}\widehat{TL}}
\renewcommand{\b}{\hat{b}}
\renewcommand{\d}{\hat{d}}

\newcommand{\bcirc}{{\color{cyan}\bullet}}

\newcommand{\supp}{\mathrm{supp}}
\renewcommand{\L}{\mathcal{L}}
\newcommand{\R}{\mathcal{R}}
\renewcommand{\(}{\left(}
\renewcommand{\)}{\right)}
\newcommand{\w}{\overline{w}}
\renewcommand{\H}{\mathcal{H}}
\renewcommand{\A}{\mathcal{A}}
\DeclareMathOperator{\FC}{FC}

\newcommand{\WA}{W(A_{n})}
\newcommand{\WD}{W(D_{n})}
\renewcommand{\qed}{\hfill \mbox{$\Box$}}

\selectcolormodel{cmyk}
\definecolor{naugreen}{cmyk}{.43,0,.34,.38}
\definecolor{naublue}{cmyk}{1,.72,0,.32}
\definecolor{lightskyblue}{cmyk}{.4,.11,0,.2}
\definecolor{softplum}{cmyk}{.02,.45,0,.68}
\definecolor{lightorange}{cmyk}{0,.45,.59,.02}
\definecolor{brickorange}{cmyk}{0,.59,.59,.20}
\definecolor{deeppink}{cmyk}{0,.6,.3,.3}
\definecolor{orchid}{cmyk}{0,.49,.02,.05}

\newcommand\drawface{\draw[fill=cyan] (-1,-1) rectangle (1,1)}
\newcommand\colordrawface{\draw[fill=magenta] (-1,-1) rectangle (1,1)}

\newcommand\drawfaceblank{\draw[fill=white, fill opacity=0.1, densely dashed] (-1,-1) rectangle (1,1)}

\newcommand\colorfaceone{
        \begin{scope}[canvas is yz plane at x=-1]
            \colordrawface;
        \end{scope}
}

\newcommand\colorfacetwo{
        \begin{scope}[canvas is yx plane at z=-1]
           \colordrawface;
        \end{scope} 
}

\newcommand\colorfacethree{
        \begin{scope}[canvas is zx plane at y=-1]
           \colordrawface;
        \end{scope}
}

\newcommand\colorrightonebar{
        \begin{scope}[canvas is zx plane at y=1]
           \colordrawface;
           \node[rotate=-90] {$\overline{1}$};
        \end{scope}
}   

\newcommand\colortoponebar{
        \begin{scope}[canvas is yx plane at z=1]
          \colordrawface;
          \node[yscale=-1] {$\overline{1}$};
      \end{scope}
}

\newcommand\colorfrontonebar{
      \begin{scope}[canvas is yz plane at x=1]
        \colordrawface;
        \node  {$\overline{1}$};
      \end{scope}
}

\newcommand\colorrightone{
        \begin{scope}[canvas is zx plane at y=1]
           \colordrawface;
           \node[rotate=-90] {$1$};
        \end{scope}
}    

\newcommand\colortopone{  
        \begin{scope}[canvas is yx plane at z=1]
          \colordrawface;
          \node[yscale=-1] {$1$};
      \end{scope}
}

\newcommand\colorfrontone{
      \begin{scope}[canvas is yz plane at x=1]
        \colordrawface;
        \node  {$1$};
      \end{scope}
}

\newcommand\faceone{
        \begin{scope}[canvas is yz plane at x=-1]
            \drawface;
        \end{scope}
}

\newcommand\facetwo{
        \begin{scope}[canvas is yx plane at z=-1]
           \drawface;
        \end{scope} 
}

\newcommand\facethree{
        \begin{scope}[canvas is zx plane at y=-1]
           \drawface;
        \end{scope}
}

\newcommand\rightonebar{
        \begin{scope}[canvas is zx plane at y=1]
           \drawface;
           \node[rotate=-90] {$\overline{1}$};
        \end{scope}
}    

\newcommand\toponebar{
        \begin{scope}[canvas is yx plane at z=1]
          \drawface;
          \node[yscale=-1] {$\overline{1}$};
      \end{scope}
}

\newcommand\frontonebar{
      \begin{scope}[canvas is yz plane at x=1]
        \drawface;
        \node  {$\overline{1}$};
      \end{scope}
}

\newcommand\rightone{
        \begin{scope}[canvas is zx plane at y=1]
           \drawface;
           \node[rotate=-90] {$1$};
        \end{scope}
}    

\newcommand\topone{
        \begin{scope}[canvas is yx plane at z=1]
          \drawface;
          \node[yscale=-1] {$1$};
      \end{scope}
}

\newcommand\frontone{
      \begin{scope}[canvas is yz plane at x=1]
        \drawface;
        \node  {$1$};
      \end{scope}
}

\newcommand\righttwo{
        \begin{scope}[canvas is zx plane at y=1]
           \drawface;
           \node[rotate=-90] {$2$};
        \end{scope}
}
   
\newcommand\toptwo{
        \begin{scope}[canvas is yx plane at z=1]
          \drawface;
          \node[yscale=-1] {$2$};
      \end{scope}
}

\newcommand\fronttwo{
      \begin{scope}[canvas is yz plane at x=1]
        \drawface;
        \node  {$2$};
      \end{scope}
}

\newcommand\rightthree{
        \begin{scope}[canvas is zx plane at y=1]
           \drawface;
           \node[rotate=-90] {$3$};
        \end{scope}
}    

\newcommand\topthree{
        \begin{scope}[canvas is yx plane at z=1]
          \drawface;
          \node[yscale=-1] {$3$};
      \end{scope}
}
\newcommand\frontthree{
      \begin{scope}[canvas is yz plane at x=1]
        \drawface;
        \node  {$3$};
      \end{scope}
}

\newcommand\rightfour{
        \begin{scope}[canvas is zx plane at y=1]
           \drawface;
           \node[rotate=-90] {$4$};
        \end{scope}
}    

\newcommand\topfour{
        \begin{scope}[canvas is yx plane at z=1]
          \drawface;
          \node[yscale=-1] {$4$};
      \end{scope}
}

\newcommand\frontfour{
      \begin{scope}[canvas is yz plane at x=1]
        \drawface;
        \node  {$4$};
      \end{scope}
}

\newcommand\rightfive{
        \begin{scope}[canvas is zx plane at y=1]
           \drawface;
           \node[rotate=-90] {$5$};
        \end{scope}
}    

\newcommand\topfive{
        \begin{scope}[canvas is yx plane at z=1]
          \drawface;
          \node[yscale=-1] {$5$};
      \end{scope}
}

\newcommand\frontfive{
      \begin{scope}[canvas is yz plane at x=1]
        \drawface;
        \node  {$5$};
      \end{scope}
}

\newcommand\righti{
        \begin{scope}[canvas is zx plane at y=1]
           \drawface;
           \node[rotate=-90] {$i$};
        \end{scope}
}    

\newcommand\topi{
        \begin{scope}[canvas is yx plane at z=1]
          \drawface;
          \node[yscale=-1] {$i$};
      \end{scope}
}

\newcommand\fronti{
      \begin{scope}[canvas is yz plane at x=1]
        \drawface;
        \node  {$i$};
      \end{scope}
}

\newcommand\rightiplus{
        \begin{scope}[canvas is zx plane at y=1]
           \drawface;
           \node[rotate=-90] {$i+1$};
        \end{scope}
}  
  
\newcommand\topiplus{
        \begin{scope}[canvas is yx plane at z=1]
          \drawface;
          \node[yscale=-1] {$i+1$};
      \end{scope}
}

\newcommand\frontiplus{
      \begin{scope}[canvas is yz plane at x=1]
        \drawface;
        \node  {$i+1$};
      \end{scope}
}

\newcommand\faceoneblank{
        \begin{scope}[canvas is yz plane at x=-1]
            \drawfaceblank;
        \end{scope}
}

\newcommand\facetwoblank{
        \begin{scope}[canvas is yx plane at z=-1]
           \drawfaceblank;
        \end{scope} 
}

\newcommand\facethreeblank{
        \begin{scope}[canvas is zx plane at y=-1]
           \drawfaceblank;
        \end{scope}
}

\newcommand\rightblank{
        \begin{scope}[canvas is zx plane at y=1]
           \drawfaceblank;
           \node[rotate=-90] {};
        \end{scope}
}    

\newcommand\topblank{
        \begin{scope}[canvas is yx plane at z=1]
          \drawfaceblank;
          \node[yscale=-1] {};
      \end{scope}
}

\newcommand\frontblank{
      \begin{scope}[canvas is yz plane at x=1]
        \drawfaceblank;
        \node  {};
      \end{scope}
}

\newcommand\onebarcube[2]{
\begin{scope}[xshift=#1, yshift=#2]
\faceone \facetwo \facethree \rightonebar \toponebar \frontonebar
\end{scope}
}
\newcommand\onecube[2]{
\begin{scope}[xshift=#1, yshift=#2]
\faceone \facetwo \facethree \rightone \topone \frontone
\end{scope}
}
\newcommand\coloronebarcube[2]{
\begin{scope}[xshift=#1, yshift=#2]
\colorfaceone \colorfacetwo \colorfacethree \colorrightonebar \colortoponebar \colorfrontonebar
\end{scope}
}
\newcommand\coloronecube[2]{
\begin{scope}[xshift=#1, yshift=#2]
\colorfaceone \colorfacetwo \colorfacethree \colorrightone \colortopone \colorfrontone
\end{scope}
}
\newcommand\twocube[2]{
\begin{scope}[xshift=#1, yshift=#2]
\faceone \facetwo \facethree \righttwo \toptwo \fronttwo
\end{scope}
}
\newcommand\threecube[2]{
\begin{scope}[xshift=#1, yshift=#2]
\faceone \facetwo \facethree \rightthree \topthree \frontthree
\end{scope}
}
\newcommand\fourcube[2]{
\begin{scope}[xshift=#1, yshift=#2]
\faceone \facetwo \facethree \rightfour \topfour \frontfour
\end{scope}
}
\newcommand\fivecube[2]{
\begin{scope}[xshift=#1, yshift=#2]
\faceone \facetwo \facethree \rightfive \topfive \frontfive
\end{scope}
}

\newcommand\icube[2]{
\begin{scope}[xshift=#1, yshift=#2]
\faceone \facetwo \facethree \righti \topi \fronti
\end{scope}
}
\newcommand\ipluscube[2]{
\begin{scope}[xshift=#1, yshift=#2]
\faceone \facetwo \facethree \rightiplus\topiplus \frontiplus
\end{scope}
}
\newcommand\blankcube[2]{
\begin{scope}[xshift=#1, yshift=#2]
\faceoneblank \facetwoblank \facethreeblank \rightblank \topblank \frontblank
\end{scope}
}

\newcommand\xxaxis{0}
\newcommand\yyaxis{90}

\newcommand\sq[2]{\fill[fill=cyan, draw=black,shift={(\xxaxis:#1)},shift={(\yyaxis:#2)}] (0,0) -- (1,0) -- (1,-1) -- (0,-1) --(0,0);}

\newcommand\csq[2]{\fill[fill=magenta, draw=black,shift={(\xxaxis:#1)},shift={(\yyaxis:#2)}] (0,0) -- (1,0) -- (1,-1) -- (0,-1) --(0,0);}

\newcommand\bsq[2]{\fill[fill=white, fill opacity=0.5, densely dashed, draw=black,shift={(\xxaxis:#1)},shift={(\yyaxis:#2)}] (0,0) -- (1,0) -- (1,-1) -- (0,-1) --(0,0);}

\newcommand\kbox[1]{
  	\fill[fill=white, draw=black, shift={(\yyaxis:#1)}] (0,0) -- (3.5,0);
  	\fill[fill=white, draw=black, shift={(\yyaxis:#1)}, dashed] (3.5,0) -- (4.5,0);
 	\fill[fill=white, draw=black, shift={(\yyaxis:#1)}] (4.5,0) -- (6,0) -- (6,-2) -- (4.5,-2);
 	\fill[fill=white, draw=black, shift={(\yyaxis:#1)}, dashed] (3.5,-2) -- (4.5,-2);
 	\fill[fill=white, draw=black, shift={(\yyaxis:#1)}] (3.5,-2) -- (0,-2) --(0,0);
	\draw[fill=black] \foreach \x in {1,2,3,5} {(\x,#1) circle (1.5pt)}; 
	\draw[fill=black] \foreach \x in {1,2,3,5} {(\x,#1-2) circle (1.5pt)}; 
}

\newcommand\fivebox[1]{
  	\fill[fill=white, draw=black, shift={(\yyaxis:#1)}] (0,0) -- (6,0) -- (6,-2) -- (0,-2) -- (0,0);
	\draw[fill=black] \foreach \x in {1,2,3,4,5} {(\x,#1) circle (1.5pt)}; 
	\draw[fill=black] \foreach \x in {1,2,3,4,5} {(\x,#1-2) circle (1.5pt)}; 
}

\newcommand\dprimebox[1]{
  	\fill[fill=white, draw=black, shift={(\yyaxis:#1)}] (0,-1) -- (0,0) -- (7,0) -- (7,-1);
	\draw[fill=black] \foreach \x in {1,2,3,4,5,6} {(\x,#1) circle (1.5pt)}; 
}

\newcommand\topbox[1]{
  	\fill[fill=white, draw=black, shift={(\yyaxis:#1)}] (0,-2) -- (0,0) -- (6,0) -- (6,-2);
	\draw[fill=black] \foreach \x in {1,2,3,4,5} {(\x,#1) circle (1.5pt)}; 
}

\newcommand\middlebox[1]{
  	\fill[fill=white, draw=black, shift={(\yyaxis:#1)}] (0,-2) -- (0,0); 
 	\fill[fill=white, draw=black, shift={(\yyaxis:#1)}](6,0) -- (6,-2);
}

\newcommand\bottombox[1]{
  	\fill[fill=white, draw=black, shift={(\yyaxis:#1)}] (6,0) -- (6,-2) -- (0,-2) -- (0,0);
	\draw[fill=black] \foreach \x in {1,2,3,4,5} {(\x,#1-2) circle (1.5pt)}; 
}

\newcommand\tbox[1]{
  	\fill[fill=white, draw=black, shift={(\yyaxis:#1)}] (0,-2) -- (0,0) ;
	\fill[fill=white, draw=black, shift={(\yyaxis:#1)}, dashed] (0,0) -- (3,0) -- (3,-2);
}

\newcommand\botbox[1]{
  	\fill[fill=white, draw=black, shift={(\yyaxis:#1)}, dashed] (3,0) -- (3,-2) -- (0,-2);
  	\fill[fill=white, draw=black, shift={(\yyaxis:#1)}] (0,-2) -- (0,0);
}

\newcommand\tttbox[1]{
  	\fill[fill=white, draw=black, shift={(\yyaxis:#1)},dashed] (0,-2) -- (0,0) -- (1,0);
 	\fill[fill=white, draw=black, shift={(\yyaxis:#1)},dashed] (1,0) -- (3,0);
	\fill[fill=white, draw=black, shift={(\yyaxis:#1)},dashed] (3,0) -- (4,0) -- (4,-2);
	\node[above,scale=0.7] at (1,#1){$i$};
	\node[above,scale=0.7] at (2,#1){$i+1$};
	\node[above,scale=0.7] at (3,#1){$i+2$};
}

\newcommand\mmmidbox[1]{
  	\fill[fill=white, draw=black, shift={(\yyaxis:#1)},dashed] (0,-2) -- (0,0); 
 	\fill[fill=white, draw=black, shift={(\yyaxis:#1)},dashed](4,0) -- (4,-2); 
}

\newcommand\bbbotbox[1]{
  	\fill[fill=white, draw=black, shift={(\yyaxis:#1)},dashed] (4,0) -- (4,-2) -- (3,-2);
  	\fill[fill=white, draw=black, shift={(\yyaxis:#1)},dashed](3,-2) -- (1,-2);
  	\fill[fill=white, draw=black, shift={(\yyaxis:#1)}, dashed](1,-2) -- (0,-2) -- (0,0);
}

\newcommand\sixbox[1]{
  	\fill[fill=white, draw=black, shift={(\yyaxis:#1)}] (0,0) -- (7,0) -- (7,-2) -- (0,-2) -- (0,0);
	\draw[fill=black] \foreach \x in {1,2,3,4,5,6} {(\x,#1) circle (1.5pt)}; 
	\draw[fill=black] \foreach \x in {1,2,3,4,5,6} {(\x,#1-2) circle (1.5pt)}; 
}

\newcommand\longbox[1]{
  	\fill[fill=white, draw=black, shift={(\yyaxis:#1)}] (0,0) -- (11,0) -- (11,-2) -- (0,-2) -- (0,0);
	\draw[fill=black] \foreach \x in {1,2,3,4,5,6,7,8,9,10} {(\x,#1) circle (1.5pt)}; 
	\draw[fill=black] \foreach \x in {1,2,3,4,5,6,7,8,9,10} {(\x,#1-2) circle (1.5pt)}; 
}

\newcommand\eightbox[1]{
  	\fill[fill=white, draw=black, shift={(\yyaxis:#1)}] (0,0) -- (8,0) -- (8,-2) -- (0,-2) -- (0,0);
	\draw[fill=black] \foreach \x in {1,2,3,4,5,6,7} {(\x,#1) circle (1.5pt)}; 
	\draw[fill=black] \foreach \x in {1,2,3,4,5,6,7} {(\x,#1-2) circle (1.5pt)}; 
}

\newcommand\ninebox[1]{
  	\fill[fill=white, draw=black, shift={(\yyaxis:#1)}] (0,0) -- (9,0) -- (9,-2) -- (0,-2) -- (0,0);
	\draw[fill=black] \foreach \x in {1,2,3,4,5,6,7,8} {(\x,#1) circle (1.5pt)}; 
	\draw[fill=black] \foreach \x in {1,2,3,4,5,6,7,8} {(\x,#1-2) circle (1.5pt)}; 
}

\newcommand\lp[1]{
  	\fill[fill=white, draw=black, xshift=-1, shift={(\yyaxis:#1)}] (0.75,#1) ellipse (16pt and 8pt);
}

\newcommand\dlp[2]{
  	\fill[fill=white, draw=black] (#1,#2) ellipse (16pt and 8pt);
\draw[fill=cyan, draw=white, xshift=-16pt] (#1,#2) circle (2.8pt);
}

\allowdisplaybreaks
 
\begin{document}

\title{Conjugacy classes of cyclically fully commutative elements in Coxeter groups of type~$A$}
\author{Brooke Fox}


\begin{titlepage}
\ 

\vfill

\begin{center}
{\Large\textbf{A cellular quotient of the Temperley--Lieb\\
algebra of type $D$}}

\vspace{.5cm}

MS Thesis, Northern Arizona University, 2014
\end{center}

\vspace{1cm}

\noindent Kirsten Davis\\
Northern Arizona University\\
Department of Mathematics and Statistics\\
Northern Arizona University\\
Flagstaff, AZ 86011\\
\url{knd27@nau.edu}\\
\\
Advisor: Dana C.~Ernst, PhD\\
Second Reader: Michael Falk, PhD\\
Third Reader: Janet McShane, PhD

\vfill

\end{titlepage}

\pagenumbering{roman}
\pagestyle{plain}


\chapter*{Abstract}

The Temperley--Lieb algebra, invented by Temperley and Lieb in 1971, is a finite dimensional associative algebra that arose in the context of statistical mechanics. Later in 1971, Penrose showed that this algebra can be realized in terms of certain diagrams. Then in 1987, Jones showed that the Temperley--Lieb algebra occurs naturally as a quotient of the Hecke algebra arising from a Coxeter group of type $A$. This realization of the Temperley--Lieb algebra as a Hecke algebra quotient was later generalized to the case of an arbitrary Coxeter group by Graham.

Cellular algebras were introduced by Graham and Lehrer, and are a class of finite dimensional associative algebras defined in terms of a ``cell datum" and three axioms. The axioms allow one to define a set of modules for the algebra known as cell modules, and one of the main strengths of the theory is that it is relatively straightforward to construct and to classify the irreducible modules for a cellular algebra in terms of quotients of the cell modules. In this thesis, we present an associative diagram algebra that is a faithful representation of a particular quotient of the Temperley--Lieb algebra of type $D$, which has a basis indexed by the so-called type II fully commutative elements of the Coxeter group of type $D$. By explicitly constructing a cell datum for the corresponding diagram algebra, we show that the quotient is cellular.


\chapter*{Acknowledgements}

First and foremost, I would like to express my appreciation and thanks to my advisor Dana C.~Ernst for his dedicated involvement in every step throughout this process. Without his assistance and support, this thesis would have never been accomplished. I would also like to show gratitude to Michael Falk and Janet McShane for taking the time to read my thesis and provide me with useful comments. Most importantly, none of this could have happened without my family. This thesis stands as a testament to your unconditional love and encouragement.


\tableofcontents
\listoffigures


\chapter{Preliminaries}

\pagenumbering{arabic}


\section{Introduction}


This thesis is organized as follows. After necessary background material on Coxeter systems is presented in Section~\ref{sec:coxsyst}, we introduce the class of fully commutative elements in Section~\ref{sec:FC}. Then, in Section~\ref{Aheap} and Section~\ref{sec:Dheaps}, we discuss a visual representation for elements of Coxeter systems, called heaps. We then recall requisite terminology and facts about Hecke algebras in Section~\ref{hecke} and Temperley--Lieb algebras in Section~\ref{sec:TL}. In Chapter~\ref{ch:diag}, we establish our notation and introduce all of the terminology required to define an associative diagram algebra, $\DTL(D_n)$, that is a faithful diagrammatic representation of the Temperley--Lieb algebra of type $D$, $\TL(D_n)$. Next, in Section~\ref{sec:loopfree}, we construct a faithful diagrammatic representation, $\LFD(D_n)$, of a particular quotient of the Temperley--Lieb algebra, $\PFTL(D_n)$, that we introduce in Section~\ref{sec:pairfree}. After defining cellular algebras in Section~\ref{sec:cellular}, we explicitly construct a cell datum for $\LFD(D_n)$ that is used to prove the main result (Theorem~\ref{loopfreecellular} and Corollary~\ref{pairfreecellular}), which says that $\LFD(D_n)$ and therefore $\PFTL(D_n)$ are cellular. 


\section{Coxeter systems}\label{sec:coxsyst}


A \emph{Coxeter system} is a pair $(W,S)$ consisting of a finite set $S$ of generating involutions and a group $W$, called a \emph{Coxeter group}, with presentation
\[
W = \langle S :(st)^{m(s, t)} = e \text{ for } m(s, t) < \infty \rangle,
\] 
where $e$ is the identity, $m(s, t) = 1$ if and only if $s=t$, and $m(s, t) = m(t, s)$.  It turns out that the elements of $S$ are distinct as group elements and that $m(s, t)$ is the order of $st$. 

Since $s$ and $t$ are elements of order 2, the relation $(st)^{m(s,t)}=e$ can be rewritten as 
\begin{equation}
\label{relation}
{\underbrace{sts \cdots }_{m(s,t)} } = {\underbrace{tst \cdots}_{m(s,t)}}
\end{equation}
with $m(s,t)\ge 2$ factors. If $m(s,t)=2$, then $st=ts$ is called a \emph{short braid relation} and $s$ and $t$ commute. Otherwise, if $m(s,t)\ge 3$, then relation~\ref{relation} is called a \emph{long braid relation}. Replacing $\underbrace{sts \cdots }_{m(s,t)}$ with $\underbrace{tst \cdots}_{m(s,t)}$ will be referred to as a \emph{braid move}.

We can represent $(W,S)$ with a unique \emph{Coxeter graph}, $\Gamma$, having: 
\begin{enumerate}
\item vertex set $S$;
\item edges $\{s,t\}$ labeled with $m(s,t)$ for all $m(s,t)\ge3$.
\end{enumerate}
Since $m(s,t)\ge 3$ occurs most frequently, it is customary to leave the corresponding edge unlabeled.

If $(W,S)$ is a Coxeter system with Coxeter graph $\Gamma$, then for emphasis, we may denote the group as $W(\Gamma)$. There is a one-to-one correspondence between Coxeter systems and Coxeter graphs. Given the Coxeter graph $\Gamma$, we can uniquely reconstruct the corresponding Coxeter system. Note that generators $s$ and $t$ are not connected in the Coxeter graph if and only if $s$ and $t$ commute. In addition, the Coxeter group $W$ is said to be \emph{irreducible} if and only if the corresponding Coxeter graph is connected. If the Coxeter graph is disconnected, the connected components correspond to factors in a direct product of irreducible Coxeter groups~\cite[Section 2.2]{Humphreys.J:A}.
 
\begin{example}
~
\begin{enumerate}[label=(\alph*)]
\item We define $A_{n}$ to be the Coxeter graph in Figure~\ref{fig:Agraph}. Given $A_{n}$, we can construct $\(\WA, S\)$ with the generators $\left\{ s_{1},s_{2},\ldots,s_{n}\right\} $ and defining relations
\begin{align*}
s_{i}^2&=  e & \text{for all }i,\\
s_{i}s_{j}&=  s_{j}s_{i} & \text{when }\left|i-j\right|>1,\\
s_{i}s_{j}s_{i}&=  s_{j}s_{i}s_{j} & \text{when }\left|i-j\right|=1.
\end{align*}
The Coxeter group $\WA$ is isomorphic to the symmetric group $S_{n+1}$ under $s_{i}\mapsto \(i,i+1\)$. 

\item We define $D_{n}$ to be the Coxeter graph in Figure~\ref{fig:Dgraph}. Given $D_{n}$, we can construct $\(\WD, S\)$ with the generators $\left\{ s_{\overline{1}},s_{1},s_{2},\ldots,s_{n-1}\right\} $ and defining relations
\begin{align*}
s_{i}^2&=  e & \text{for all }i,\\
s_{i}s_{j}&=  s_{j}s_{i} & \text{when }\left|i-j\right|>1 \text{ for }i, j \in \left\{1,2,\ldots,n\right\},\\
s_{i}s_{j}s_{i}&=  s_{j}s_{i}s_{j} & \text{when }\left|i-j\right|=1 \text{ for }i, j \in \left\{1,2,\ldots,n\right\}, \\
s_{\overline{1}}s_{i}&=  s_{i}s_{\overline{1}} & \text{for all }i\in \left\{1,3,\ldots,n\right\},\\
s_{\overline{1}}s_{2}s_{\overline{1}}&=  s_{2}s_{\overline{1}}s_{2}.
\end{align*}
The Coxeter group $\WD$ is isomorphic to $S_{n}^{D}$, where $S_{n}^{D}$ is the subgroup of the group of signed permutations having an even number of sign changes, called the \emph{group of even signed permutations}. The main focus of this thesis will be the Coxeter system of type $D_n$.
\end{enumerate}
\end{example}

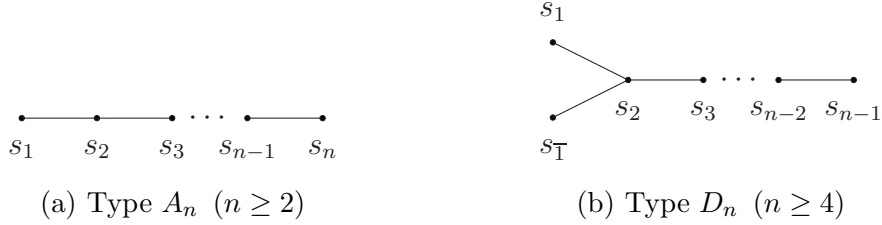
\begin{figure}[h]
\centering
\begin{subfigure}[b]{0.4\textwidth}
\centering
\begin{tikzpicture} \draw[fill=black] \foreach \x in {1,2,3,4,5} {(\x,10) circle (1pt)}; \draw \foreach \x in {1,2,3} {(\x,10) node[label=below:$s_{\x}$]{}};  \draw {(4,10) node[label=below:$s_{n-1}$]{}}; \draw {(5,10) node[label=below:$s_{n}$]{}}; \draw {(3.5,10) node[]{$\cdots$}}; \draw[-] (1,10) -- (2,10); \draw[-] (2,10) -- (3,10); \draw[-] (4,10) -- (5,10); \end{tikzpicture}
\caption{Type $A_{n}~\(n\ge 2\)$}
\label{fig:Agraph}
\end{subfigure}
\quad
\begin{subfigure}[b]{0.4\textwidth}
\centering
\begin{tikzpicture} \draw[fill=black] \foreach \x in {2,3,4,5} {(\x,10) circle (1pt)}; \draw \foreach \x in {2,3} {(\x,10) node[label=below:$s_{\x}$]{}}; \draw[fill=black] \foreach \x in {9.5,10.5} {(1,\x) circle (1pt)}; \draw {(1,9.5) node[label=below:$s_{\overline{1}}$]{}}; \draw {(1,10.5) node[label=above:$s_{1}$]{}}; \draw {(4,10) node[label=below:$s_{n-2}$]{}}; \draw {(5,10) node[label=below:$s_{n-1}$]{}}; \draw {(3.5,10) node[]{$\cdots$}}; \draw[-] (1,9.5) -- (2,10); \draw[-] (1,10.5) -- (2,10); \draw[-] (2,10) -- (3,10); \draw[-] (4,10) -- (5,10); \end{tikzpicture} 
\caption{Type $D_{n}~\(n\ge 4\)$}
\label{fig:Dgraph}
\end{subfigure}
\caption{Coxeter graphs corresponding to Coxeter systems of type $A_{n}$ and $D_{n}$.}
\label{fig:graphs}
\end{figure}

Given a Coxeter system $(W,S)$, a word $s_{x_1}s_{x_2}\cdots s_{x_m}$ in the free monoid on $S$ is called an \emph
{expression} for $w\in W$ if it is equal to $w$ when considered as a group element. If $m$ is minimal among all expressions for $w$, the corresponding word is called a \emph{reduced expression} for $w$. In this case, we define the \emph{length} of $w$ to be $\ell(w):=m$. Given $w \in W$, if we wish to emphasize a fixed, possibly reduced, expression for $w$, we represent it as $\w=s_{x_1}s_{x_2}\cdots s_{x_m}.$

\begin{theorem}[Matsumoto's Theorem~\cite{Geck2000}]
If $w \in W$, then every reduced expression for $w$ can be obtained from any other by applying a sequence of braid moves of the form 
\[
{\underbrace{sts \cdots }_{m(s,t)} } \mapsto {\underbrace{tst \cdots}_{m(s,t)}}
\]
where $s,t \in S$ and $m(s,t)\ge 2$.\qed
\end{theorem}

It follows from Matsumoto's theorem that any two reduced expressions for $w \in W$ have the same number of generators in the expression. The \emph{support} of an element $w \in W$, denoted $\supp(w)$, is the set of all generators appearing in any reduced expression for $w$, which is well-defined by Matsumoto's Theorem. We will use $\w$ to represent a fixed expression, possibly reduced, for $w \in W$. We define a \emph{subexpression} of $\w$ to be any subsequence of $\w$.  If $x \in W(\Gamma)$ has an expression that is equal to a subexpression of $\w$, then we write $x \leq w$.  This is a well-defined partial order~\cite[Chapter 5]{Humphreys.J:A} on $W(\Gamma)$ and is called the \emph{(strong) Bruhat order}. We will refer to a consecutive subexpression of $\w$ as a \emph{subword}. 

Let $w \in W$.  We define
\[
\L(w):=\{s \in S: \ell(sw) < \ell(w)\}
\]
and
\[
\R(w):=\{s\in S:\ell(ws)<\ell(w)\}.
\]
The set $\L(w)$ is called the \emph{left descent set} of $w$ and $\R(w)$ is called the \emph{right descent set}. It turns out that $s \in \L(w)$ if and only if $w$ has a reduced expression beginning with $s$ and $s \in \R(w)$ if and only if $w$ has a reduced expression ending with $s$.

\begin{example}
Let $w\in W(A_{4})$ and let $\w =s_{1}s_{2}s_{1}s_{4}s_{2}$ be an expression for $w$. Then
\begin{align*}
s_{1}s_{2}s_{1}s_{4}s_{2}&=  s_{2}s_{1}s_{2}s_{4}s_{2}\\
&=  s_{2}s_{1}s_{2}s_{2}s_{4} \\
&=  s_{2}s_{1}s_{4}. 
\end{align*}
This shows that $\w$ is not reduced. However, it turns out that $s_{2}s_{1}s_{4}$ is a reduced expression for $w$ and hence $l(w)=3$. Note that $\L(w)=\{s_2,s_4\},\R(w)=\{s_1,s_4\}$, and $s_{1}s_{4}$ is a subword of $\w$.
\end{example}


\section{Fully commutative}\label{sec:FC}


Let $(W,S)$ be a Coxeter system of type $\Gamma$ and let $w \in W$. Following Stembridge~\cite{Stembridge1996}, we define a relation $\sim$ on the set of reduced expressions for $w$. Let $\w_1$ and $\w_2$ be two reduced expressions for $w$. We define $\w_1 \sim \w_2$ if we can obtain $\w_1$ from $\w_1$ by applying a single commutation move of the form $st\mapsto ts$, where $m(s,t)=2$. Now, define the equivalence relation $\approx$ by taking the reflexive transitive closure of $\sim$. Each equivalence class under $\approx$ is called a \emph{commutation class}. If $w$ has a single commutation class, then we say that $w$ is \emph{fully commutative}. That is, $w$ is fully commutative if any two reduced expressions for $w$ can be transformed into each other by a sequence of short braid relations. We denote the set of all fully commutative elements of $W$ by $\FC(\Gamma)$ where $\Gamma$ is the corresponding Coxeter graph to $(W,S)$. The following theorem shows that we never have an opportunity to apply a long braid relation when $w$ is fully commutative.

\begin{theorem}[Stembridge~\cite{Stembridge1996}]\label{stem}
An element $w\in W$ is fully commutative if and only if no reduced expression for $w$ contains ${\underbrace{sts \cdots }_{m(s,t)} }$ as a subword for $m(s,t)\ge3$. \qed
\end{theorem}

\begin{remark}\label{subwords}
The elements of $FC(D_n)$ are precisely those whose reduced expressions avoid consecutive subwords of the type $s_{i}s_{j}s_{i}$ when $s_{i}$ and $s_{j}$ are connected in the Coxeter graph of type $D_n$
\end{remark}

\begin{example}
Figure~\ref{fig:commclasses} depicts all possible reduced expressions of $w$ and the relationships among them via commutations and long braid moves. It is clear by inspection that there are two commutation classes for $w$ represented by the two boxes, hence $w$ is not fully commutative.
\end{example}

\begin{figure}[h]
\centering
\begin{tikzpicture}
\draw (0,0) -- (4,0) -- (4,-5) -- (0,-5) -- (0,0);
\node at (2,-1) {$s_3s_1s_2s_3s_4$};
\tikzset{>=latex}
\draw [green, <->,>=stealth] (2,-1.5) to (2,-3.5) node[scale=0.85, left] at (2,-2.5) {$s_1s_3=s_3s_1$};
\node at (2,-4) {$s_1s_3s_2s_3s_4$};
\draw [red, <->,>=stealth] (4.5,-4) to (6.5,-4) node[scale=0.85, below] at (5.5,-4) {$s_3s_2s_3=s_2s_3s_2$};
\draw (7,0) -- (11,0) -- (11,-5) -- (7,-5) -- (7,0);
\node at (9,-1) {$s_1s_2s_3s_4s_2$};
\tikzset{>=latex}
\draw [cyan, <->,>=stealth] (9,-1.5) to (9,-3.5) node[scale=0.85, right] at (9,-2.5) {$s_2s_4=s_4s_2$};
\node at (9,-4) {$s_1s_2s_3s_2s_4$};
\end{tikzpicture}
\caption{Commutation classes for a non-fully commutative element.}
\label{fig:commclasses}
\end{figure}
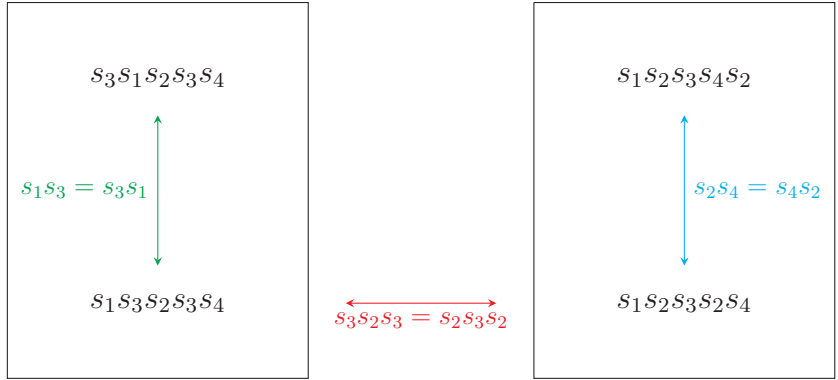

\begin{example}
Figure~\ref{fig:commclass} depicts all possible reduced expressions of $w$ and the relationships among them via only commutations. In this case, we see that every reduced expression for $w$ can be obtained from another via a sequence of commutation moves, and hence there is only one commutation class for $w$. Therefore $w$ is fully commutative. It is also clear that there is never an opportunity to apply a long braid relation, which agrees with Theorem~\ref{stem}.
\end{example}

\begin{figure}[h]
\centering
\begin{tikzpicture}
\draw (-1,0) -- (9,0) -- (9,-6) -- (-1,-6) -- (-1,0);
\node at (3,-1) {$s_3s_1s_2s_4s_3$};
\tikzset{>=latex}
\draw [green, <->,>=stealth] (2.75,-1.5) to (1.25,-2.5) node[scale=0.85, above left] at (2,-2) {$s_1s_3=s_3s_1$};
\draw [cyan, <->,>=stealth] (3.25,-1.5) to (4.75,-2.5) node[scale=0.85, above right] at (4,-2) {$s_2s_4=s_4s_2$};
\node at (1,-3) {$s_1s_3s_2s_4s_3$};
\node at (5,-3) {$s_3s_1s_4s_2s_3$};
\draw [green, <->,>=stealth] (4.75,-3.5) to (3.25,-4.5) node[scale=0.85, below right] at (4,-4) {$s_1s_3=s_3s_1$};
\draw [cyan, <->,>=stealth] (1.25,-3.5) to (2.75,-4.5) node[scale=0.85, below left] at (2,-4) {$s_2s_4=s_4s_2$};
\node at (3,-5) {$s_1s_3s_4s_2s_3$};
\draw [magenta, <->,>=stealth] (5.25,-3.5) to (6.75,-4.5) node[scale=0.85, above right] at (6,-4) {$s_1s_3=s_3s_1$};
\node at (7,-5) {$s_3s_4s_1s_2s_3$};
\end{tikzpicture}
\caption{Commutation class for a fully commutative element.}
\label{fig:commclass}
\end{figure}
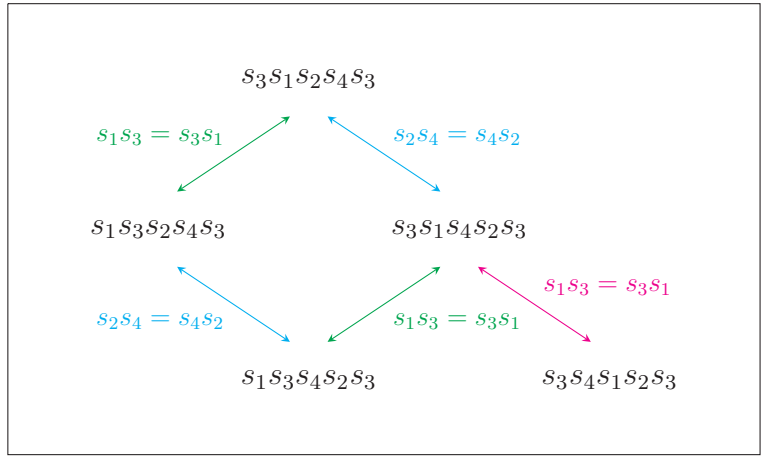

Stembridge classified the irreducible Coxeter groups that contain a finite number of fully commutative elements, the so-called \emph{FC-finite Coxeter groups}.  While $W(D_{n})$ is a finite Coxeter group, and therefore contains a finite number of fully commutative elements, there are examples of infinite Coxeter groups that contain a finite number of fully commutative elements.  For example, Coxeter groups of type $E_n$ for $n\geq 9$ (see Figure~\ref{fig:cfcfg}) are infinite, but contain only finitely many fully commutative elements.

\begin{theorem}[Stembridge~\cite{Stembridge1996}]
The irreducible FC-finite Coxeter groups are the Coxeter groups of type $A_{n}\left(n\ge 1\right),$ $B_{n}\left(n\ge 2\right),$ $D_{n}\left(n\ge 4\right),$ $E_{n}\left(n\ge 6\right),$ $F_{n}\left(n\ge 4\right),$ $H_{n}\left(n\ge 3\right)$ and $I_{2}\left(m\right)\left(3\le m<\infty \right)$ with corresponding Coxeter graphs shown in Figure~\ref{fig:cfcfg}. \qed
\end{theorem}

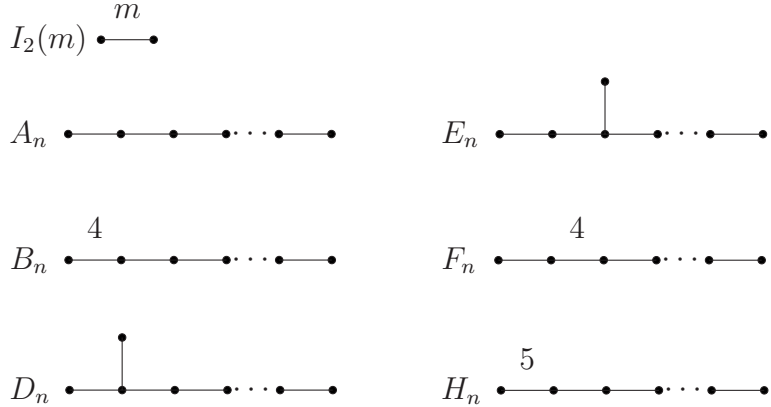
\begin{figure}[h]
\centering
\begin{tabular}{ll}
\begin{tikzpicture}[scale=.7]
\draw[fill=black] \foreach \x in {1,2} {(\x,0) circle (2pt)};
\draw {(0,0) node{$I_{2}(m)$}
(1.5,0) node[label=above:$m$]{}
[-] (1,0) -- (2,0)};
\end{tikzpicture}
& \\
\begin{tikzpicture}[scale=.7]
\draw[fill=black] \foreach \x in {1,2,...,6} {(\x,10) circle (2pt)};
\draw {(.25,10) node{$A_{n}$}
(4.5,10) node{$\cdots$}
[-] (1,10) -- (4,10)
[-] (5,10) -- (6,10)}; 
\end{tikzpicture}
&
\quad  \quad \begin{tikzpicture}[scale=.7]
\draw[fill=black] \foreach \x in {1,2,...,6} {(\x,4.5) circle (2pt)};
\draw[fill=black] (3,5.5) circle (2pt);
\draw {(.25,4.5) node{$E_{n}$}
(4.5,4.5) node{$\cdots$}
[-] (1,4.5) -- (4,4.5)
[-] (5,4.5) -- (6,4.5)
[-] (3,4.5) -- (3,5.5)};
\end{tikzpicture}\\
\\
\begin{tikzpicture}[scale=.7]
\draw [fill=black] \foreach \x in {1,2,...,6} {(\x,8.5) circle (2pt)};
\draw {(.25,8.5) node{$B_{n}$}
(1.5,8.5) node[label=above:$4$]{}
(4.5,8.5) node{$\cdots$}
[-] (1,8.5) -- (4,8.5)
[-] (5,8.5) -- (6,8.5)}; 
\end{tikzpicture}
&
\quad  \quad \begin{tikzpicture}[scale=.7]
\draw[fill=black] \foreach \x in {1,2,...,6} {(\x,3) circle (2pt)};
\draw {(.25,3) node{$F_{n}$}
(2.5,3) node[label=above:$4$]{}
(4.5,3) node{$\cdots$}
[-] (1,3) -- (4,3)
[-] (5,3) -- (6,3)};
\end{tikzpicture}\\
\\
\begin{tikzpicture}[scale=.7]
\draw[fill=black] \foreach \x in {1,2,...,6} {(\x,6.5) circle (2pt)};
\draw[fill=black] (2,7.5) circle (2pt);
\draw {(.25,6.5) node{$D_{n}$}
(4.5,6.5) node{$\cdots$}
[-] (1,6.5) -- (4,6.5)
[-] (5,6.5) -- (6,6.5)
[-] (2,6.5) -- (2,7.5)};
\end{tikzpicture}
&
\quad  \quad \begin{tikzpicture}[scale=.7]
\draw[fill=black] \foreach \x in {1,2,...,6} {(\x,1.5) circle (2pt)};
\draw {(.25,1.5) node{$H_{n}$}
(1.5,1.5) node[label=above:$5$]{}
(4.5,1.5) node{$\cdots$}
[-] (1,1.5) -- (4,1.5)
[-] (5,1.5) -- (6,1.5)}; 
\end{tikzpicture}
\end{tabular}
\caption{Coxeter graphs corresponding to the irreducible FC-finite Coxeter groups.}
\label{fig:cfcfg}
\end{figure}


\section{Heaps in type $A_{n}$}\label{Aheap}


Every reduced expression can be associated with a labeled partially ordered set (poset) called a heap. Heaps provide a visual representation of a reduced expression while preserving the relations among the generators. For simplicity, we will first define heaps corresponding to Coxeter groups of type $A_{n}$ and mimic the development found in~\cite{Billey2007},~\cite{Ernst2010a}, and~\cite{Stembridge1996}.  

Let $(W,S)$ be a Coxeter system.  Suppose $\w = s_{x_1} \cdots s_{x_k}$ is a fixed reduced expression for $w \in \WA$.  We define a partial ordering on the indices $\{1, \dots, k\}$ by the transitive closure of the relation $j \lessdot i$ if $i < j$ and $s_{x_i}$ and $s_{x_j}$ do not commute.  In particular, since $\w$ is reduced, $j \lessdot i$ if $i < j$ and $s_{x_i} = s_{x_j}$ by transitivity.  This partial order is referred to as the \emph{heap} of $\w$, where $i$ is labeled by $s_{x_i}$. Each heap corresponds to a commutation class and it follows from~\cite{Stembridge1996} that $w$ is fully commutative if and only if there is a unique heap associated to $w$.

\begin{example}\label{ex:Hasse}
Let $\w=s_{2}s_{1}s_{3}s_{2}s_{4}s_{5}$ be a reduced expression for $w \in \FC(A_{5})$. We see that $\w$ is indexed by $\{1,2,3,4,5,6\}$. For example, $6 \lessdot 5$ since $5<6$ and $s_{4}$ and $s_{5}$ do not commute. The labeled Hasse diagram for the heap poset of $w$ is shown in Figure~\ref{fig:hasse}. 
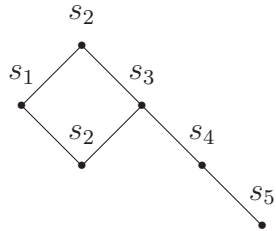
\begin{figure}[h]
\centering
\begin{tikzpicture}[scale=0.8]
\node[scale=0.6, label=above:$s_{2}$] at (1,5.5) {$\bullet$};
\node[scale=0.6, label=above:$s_{1}$] at (0,4.5) {$\bullet$};
\node[scale=0.6, label=above:$s_{3}$] at (2,4.5) {$\bullet$};
\node[scale=0.6, label=above:$s_{2}$] at(1,3.5) {$\bullet$};
\node[scale=0.6, label=above:$s_{4}$] at(3,3.5) {$\bullet$};
\node[scale=0.6, label=above:$s_{5}$] at (4,2.5) {$\bullet$};
\draw (1,5.5)--(0,4.5)--(1,3.5)--(2,4.5)--(3,3.5)--(4,2.5);
\draw (1,5.5)--(2,4.5);
\end{tikzpicture}
\caption{Labeled Hasse diagram for the heap of a fully commutative element.}
\label{fig:hasse}
\end{figure}
\end{example}

Let $\w$ be a fixed reduced expression for $w \in W(A_{n})$.  As in~\cite{Billey2007} and~\cite{Ernst2010a}, we will represent a heap for $\w$ as a set of lattice points embedded in $\{1,2,\ldots,n\} \times \mathbb{N}$.  To do so, we assign coordinates (not unique) $(x,y) \in \{1,2,\ldots, n+1\} \times \mathbb{N}$ to each entry of the labeled Hasse diagram for the heap of $\w$ in such a way that:
\begin{enumerate}
\item An entry with coordinates $(x,y)$ is labeled $s_i$ (or $i$) in the heap if and only if $x = i$; 

\item If an entry with coordinates $(x,y)$ is greater than an entry with coordinates $(x',y')$ in the heap then $y > y'$.
\end{enumerate}

In the case of type $A_{n}$ and other straight line Coxeter graphs, it follows from the definition that $(x,y)$ covers $(x',y')$ in the heap if and only if $x = x' \pm 1$, $y > y'$, and there are no entries $(x'', y'')$ such that $x'' \in \{x, x'\}$ and $y'< y'' < y$.  This implies that we can completely reconstruct the edges of the Hasse diagram and the corresponding heap poset from a lattice point representation. A lattice point representation of a heap allows us to visualize potentially cumbersome arguments.  Note that entries on top of a heap correspond to generators occurring to the left  in the corresponding reduced expression.  

Let $\w$ be a reduced expression for $w \in W(A_{n})$.  We let $H(\w)$ denote a lattice representation of the heap poset in $\{1,2,\ldots,n+1\} \times \N$.  If $w$ is fully commutative, then the choice of reduced expression for $w$ is irrelevant, in which case, we will often write $H(w)$ and we will refer to $H(w)$ as the heap of $w$.  Note that if $w\in \FC(A_n)$, then entries on the top of a heap correspond to $s_i\in \L(w)$.

Given a heap, every generator will have a fixed $x$-co\-or\-di\-nate, yet the $y$-co\-or\-di\-nates may differ.  In particular, two entries labeled by the same generator will possess the same $x$-co\-or\-di\-nate but may differ by the amount of vertical space between them.

Let $\w=s_{x_1}\cdots s_{x_k}$ be a reduced expression for $w \in W(A_{n})$.  If $s_{x_i}$ and $s_{x_j}$ are adjacent generators in the Coxeter graph with $i<j$, then we must place the point labeled by $s_{x_i}$ at a level that is above the level of the point labeled by $s_{x_j}$.  Because generators that are not adjacent in the Coxeter graph do commute, points whose $x$-coordinates differ by more than one can slide past each other or land at the same level.  To emphasize the covering relations of the lattice representation we will enclose each entry of the heap in a $2\times 2$ square in such a way that if one entry covers another, the squares overlap halfway. We will also label the squares with $i$ to represent the generator $s_{i}$.

\begin{example}
Let $\w=s_1s_2s_4$ be a reduced expression for $w\in \FC(A_4)$. Figure~\ref{fig:multirep}  shows two representations for the heap of $w$. 
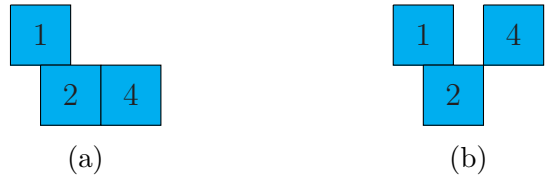
\begin{figure}[h]
\centering
\begin{subfigure}[b]{0.3\textwidth}
\centering
\begin{tikzpicture}[scale=0.8]
\sq{0}{7};
\node at (0.5,6.5) {$1$};
\sq{1.5}{6};
\node at (2,5.5) {$4$};
\sq{0.5}{6};
\node at (1,5.5) {$2$};
\end{tikzpicture}
\caption{}
\end{subfigure}
\begin{subfigure}[b]{0.3\textwidth}
\centering
\begin{tikzpicture}[scale=0.8]
\sq{0}{7};
\node at (0.5,6.5) {$1$};
\sq{1.5}{7};
\node at (2,6.5) {$4$};
\sq{0.5}{6};
\node at (1,5.5) {$2$};
\end{tikzpicture}
\caption{}
\label{canonical}
\end{subfigure}
\caption{Two possible representations for the heap of a fully commutative element.}
\label{fig:multirep}
\end{figure}
\end{example}

\begin{example}\label{multipleheaps}
Let $\w_1=s_1s_3s_2s_1$ be a reduced expression for $w\in W(A_3)$. By applying the short braid relation, $s_{1}s_{3}=s_{3}s_{1}$, we can obtain another reduced expression, $\w_2=s_3s_1s_2s_1$ which is in the same commutation class as $\w_1$, and hence has the same heap. But, by applying the long braid relation, $s_1s_2s_1=s_2s_1s_2$, we can obtain a third reduced expression $\w_3=s_3s_2s_1s_2$, which is in a different commutation class. The representations of $H(\w_1)$, $H(\w_2)$, and $H(\w_3)$ are given in Figure~\ref{fig:multiheap} where we have color-coded the blocks of the heap that correspond to the long braid relation, $s_1s_2s_1=s_2s_1s_2$.
\end{example}

\begin{figure}[h]
\centering
\begin{subfigure}[b]{0.3\textwidth}
\centering
\begin{tikzpicture}[scale=0.8]
\csq{0}{7};
\node at (0.5,6.5) {$1$};
\sq{1}{7};
\node at (1.5,6.5) {$3$};
\csq{0.5}{6};
\node at (1,5.5) {$2$};
\csq{0}{5};
\node at (0.5,4.5) {$1$};
\end{tikzpicture}
\caption{$H(\w_1)$ and $H(\w_2)$}
\label{commclass}
\end{subfigure}
\begin{subfigure}[b]{0.3\textwidth}
\centering
\begin{tikzpicture}[scale=0.8]
\sq{1}{7};
\node at (1.5,6.5) {$3$};
\csq{0.5}{6};
\node at (1,5.5) {$2$};
\csq{0}{5};
\node at (0.5,4.5) {$1$};
\csq{0.5}{4};
\node at (1,3.5) {$2$};
\end{tikzpicture}
\caption{$H(\w_3)$}
\label{commclasstwo}
\end{subfigure}
\caption{Two heaps for a non-fully commutative element.}
\label{fig:multiheap}
\end{figure}
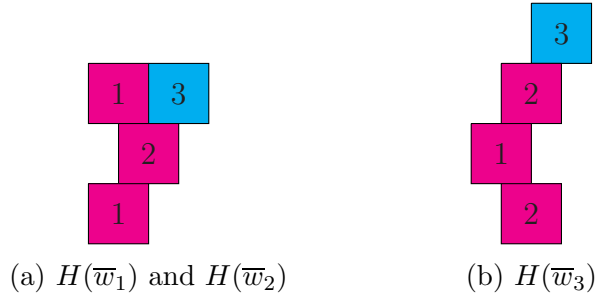

When $w$ is fully commutative, we wish to make a canonical representation of $H(w)$ by giving all entries corresponding to elements in $\L(w)$ the same vertical position and all other entries in the heap should have the highest vertical position possible (as in Figure~\ref{canonical}).

Let $w \in \FC(A_n)$ have reduced expression $\w=s_{x_1}\cdots s_{x_k}$ and suppose $s_{x_i}$ and $s_{x_j}$ equal the same generator $s_m$, so that the corresponding entries have $x$-coordinate $m$ in $H(w)$.  We say that $s_{x_i}$ and $s_{x_j}$ are \emph{consecutive} if there is no other occurrence of $s_{m}$ occurring between them in $\w$.  In this case, $s_{x_i}$ and $s_{x_j}$ are consecutive in $H(w)$, as well.

Let $\w=s_{x_{1}} \cdots s_{x_{k}}$ be a reduced expression for $w \in W(A_{n})$.  We define a heap $H'$ to be a \emph{subheap} of $H(\w)$ if $H'=H(\w')$, where $\w'=s_{y_1}s_{y_2} \cdots s_{y_m}$ is a subexpression (not necessarily a subword) of $\w$.  

A subposet $Q$ of a poset $P$ is called \emph{convex} if $y \in Q$ whenever $x < y < z$ in $P$ and $x, z \in Q$.  We will refer to a subheap as a \emph{convex subheap} if the underlying subposet is convex.  

\begin{example}
Let $\w= s_{2}s_{1} s_{3}s_{2} s_{4}s_{5}$, as in Example~\ref{ex:Hasse}.  Now, let $\w'=s_{2}s_{3}s_{2}$ be the subexpression of $\w$ that results from deleting all but the first, third, and fourth generators of $\w$.  Then $H(\w')$ equals the heap given in Figure~\ref{fig:subheap} and is a subheap of $H(\w)$, but is not convex since there is an entry in $H(\w)$ labeled by $s_{1}$ occurring between the two consecutive occurrences of $s_{2}$ that does not occur in $H(\w')$.  However, if we do include the entry labeled by $s_{1}$, then we obtain the heap in Figure~\ref{fig:heapheap}, which is a convex subheap of $H(\w)$. Note that if we delete the occurrence of the block labeled by $1$ in the original heap, then the heap in Figure~\ref{fig:subheap} is a convex subheap.
\end{example}

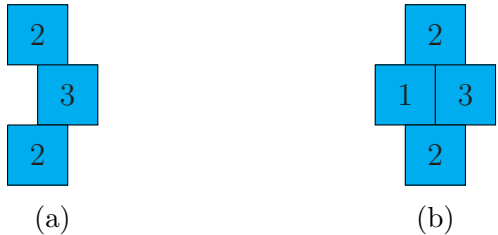
\begin{figure}[!h]
\centering
\begin{subfigure}[b]{0.3\textwidth}
\centering
\begin{tikzpicture}[scale=0.8]
\sq{0.5}{6};
\node at (1,5.5) {$2$};
\sq{1}{5};
\node at (1.5,4.5) {$3$};
\sq{0.5}{4};
\node at(1,3.5) {$2$};
\end{tikzpicture}
\caption{}
\label{fig:subheap}
\end{subfigure}
\begin{subfigure}[b]{0.3\textwidth}
\centering
\begin{tikzpicture}[scale=0.8]
\sq{0.5}{6};
\node at (1,5.5) {$2$};
\sq{0}{5};
\node at (0.5,4.5) {$1$};
\sq{1}{5};
\node at (1.5,4.5) {$3$};
\sq{0.5}{4};
\node at(1,3.5) {$2$};
\end{tikzpicture}
\caption{}
\label{fig:heapheap}
\end{subfigure}
\caption{Subheaps of the heap from Example~\ref{ex:Hasse}.}
\label{fig:144}
\end{figure}

The following fact is implicit in the literature (in particular, see the proof of~\cite[Proposition 3.3]{Stembridge1996}) and follows easily from the definitions.

\begin{proposition}
Let $w \in \FC(W)$.  Then $H'$ is a convex subheap of $H(w)$ if and only if $H'$ is the heap for some subword of some reduced expression for $w$. \qed
\end{proposition}

The following lemma follows from~\cite[Lemma 2.4.5]{Ernst2010a} and will enable us to recognize when a heap corresponds to a fully commutative element in $W(A_{n})$.  

\begin{lemma}\label{lem:notFCheaps}
Let $w \in \FC(A_{n})$.  Then $H(w)$ cannot contain any of the convex subheaps in Figure~\ref{fig:notFCA}, where $i\in \{1,\ldots ,n-1\}$ and we use \begin{tikzpicture}[scale=0.3] \bsq{0}{0}; \end{tikzpicture} to emphasize that no element of the heap occupies the corresponding position. \qed
\end{lemma}

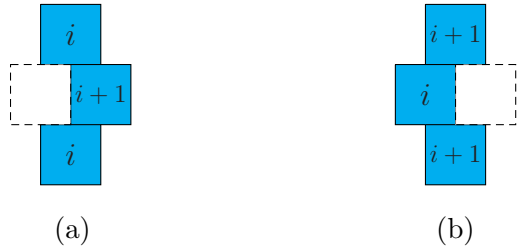
\begin{figure}[!h]
\centering
\begin{subfigure}[b]{0.3\textwidth}
\centering
\begin{tabular}[c]{c}
\begin{tikzpicture}[scale=0.8]
\sq{0.5}{6};
\node at (1,5.5) {$i$};
\sq{1}{5};
\node [scale=0.8] at (1.5,4.5) {$i+1$};
\sq{0.5}{4};
\node at(1,3.5) {$i$};
\bsq{0}{5};
\end{tikzpicture}
\end{tabular}
\caption{}
\label{fig:ii+1i}
\end{subfigure}
\begin{subfigure}[b]{0.3\textwidth}
\centering
\begin{tabular}[c]{c}
\begin{tikzpicture}[scale=0.8]
\sq{1}{7};
\node[scale=0.8] at (1.5,6.5) {$i+1$};
\sq{0.5}{6};
\node at (1,5.5) {$i$};
\sq{1}{5};
\node[scale=0.8] at (1.5,4.5) {$i+1$};
\bsq{1.5}{6};
\end{tikzpicture}
\end{tabular}
\caption{}
\label{fig:i+1ii+1}
\end{subfigure}
\caption{Impermissible convex subheaps for elements in $\FC(A_n)$.}
\label{fig:notFCA}
\end{figure}


\section{Heaps in type $D_{n}$}\label{sec:Dheaps}


For type $D_{n}$, we will represent heaps in a similar fashion, but need to make one modification.  Recall that in type $D_{n}$ (see Figure~\ref{fig:Dgraph}) the generators are $\{s_{\overline{1}},s_{1},\ldots ,s_{n-1}\}$, where $\{s_{1},\ldots ,s_{n-1}\}$ generates a Coxeter group of type $A_{n-1}$, so we just need to modify our heaps to include $s_{\overline{1}}$. Let $\w$ be a fixed reduced expression for $w \in W(D_{n})$. We will represent a heap for $\w$ as a set of lattice points embedded in $\{1,2,\ldots,n\} \times \mathbb{N}\times \{-1,0,1\}$ and assign coordinates (not unique) $(x,y,z) \in \{1,2,\ldots, n+1\} \times \mathbb{N}\times \{-1,0,1\}$ to each entry of the labeled Hasse diagram for the heap of $\w$ in such a way that:

\begin{enumerate}
\item An entry with coordinates $(x,y,0)$ is labeled $s_i$ (or $i$) in the heap if and only if $x=i$ and $i\in \{2,3,\ldots,n-1\}$;
\item  An entry with coordinates $(1,y,-1)$ is labeled $s_{\overline{1}}$ (or $\overline{1}$) in the heap;
\item  An entry with coordinates $(1,y,1)$ is labeled $s_{1}$ (or $1$) in the heap;
\item If an entry with coordinates $(x,y,z)$ is greater than an entry with coordinates $(x',y',z')$ in the heap then $y > y'$.
\end{enumerate}
To emphasize the covering relations of the lattice we will enclose each entry just as we did in Section~\ref{Aheap} but with  $2\times 2\times 2$ cubes instead of $2\times 2$ squares in such a way that if one entry covers another, the cubes overlap.
\begin{example}
Let $\w = s_{\overline{1}}s_{3}s_{2}s_{4}s_{3}s_{5}$ be a reduced expression for $w\in \FC(D_{6})$. Then Figure~\ref{fig:Dheap} is a 3-dimensional representation for $H(w)$.
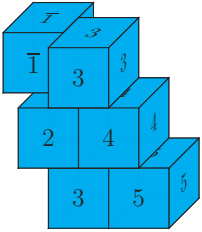
\begin{figure}[h]
\centering
\begin{tikzpicture}
    [x={(-0.5cm,-0.5cm)}, y={(1cm,0cm)}, z={(0cm,1cm)}, 
    scale=0.4,fill opacity=0.97, every node/.append style={thick, transform shape, scale=2}]
\threecube{28.5}{-57};
\fivecube{85.5}{-57};
\twocube{0}{0};
\fourcube{57}{0};
\onebarcube{-14.25}{71.25};
\threecube{28.5}{57};
\end{tikzpicture}
\caption{A heap for a fully commutative element in $W(D_{6})$.}
\label{fig:Dheap}
\end{figure}
\end{example}

The following lemma is analogous to Lemma~\ref{lem:notFCheaps} and will be helpful in recognizing when a heap represents a fully commutative element in $W(D_{n})$. Note that all heaps corresponding to a fully commutative element in $W(D_{n})$ will be referred to as \emph{fully commutative heaps} throughout the rest of this thesis. The following Lemma follows from Remark~\ref{subwords}.

\begin{lemma}\label{lem:notFC3Dheaps}
Let $w \in \FC(D_{n})$.  Then $H(w)$ cannot contain any of the convex subheaps in Figure~\ref{fig:notFCD}, where $i\in \{2,\ldots ,n-2\}$ and we use~
\begin{tikzpicture}
    [x={(-0.5cm,-0.5cm)}, y={(1cm,0cm)}, z={(0cm,1cm)}, 
    scale=0.15,fill opacity=0.97, every node/.append style={thick, transform shape, scale=1.8}]
\blankcube{85.5}{-57};
\end{tikzpicture}
~to emphasize that no element of the heap occupies the corresponding position. \qed
\end{lemma}

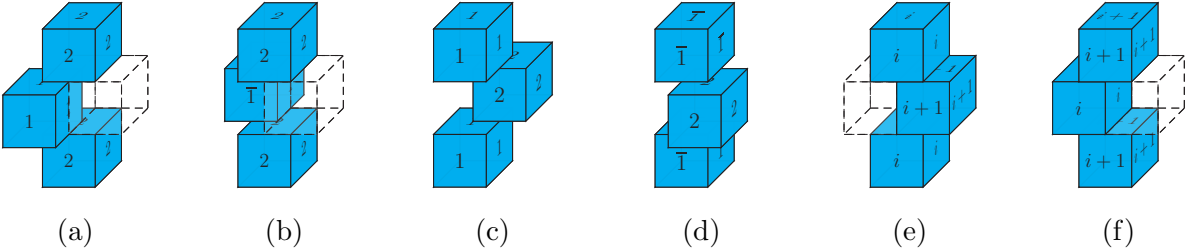
\begin{figure}[!h]
\centering
\begin{subfigure}[b]{0.16\textwidth}
\centering
\begin{tabular}[c]{c}
\begin{tikzpicture}
    [x={(-0.5cm,-0.5cm)}, y={(1cm,0cm)}, z={(0cm,1cm)}, 
    scale=0.35,fill opacity=0.97, every node/.append style={thick, transform shape, scale=1.8}]
\twocube{57}{-114};
\onecube{14.25}{-71.25};
\blankcube{85.5}{-57};
\twocube{57}{0};
\end{tikzpicture}
\end{tabular}
\caption{}
\end{subfigure}
\begin{subfigure}[b]{0.16\textwidth}
\centering
\begin{tabular}[c]{c}
\begin{tikzpicture}
    [x={(-0.5cm,-0.5cm)}, y={(1cm,0cm)}, z={(0cm,1cm)}, 
    scale=0.35,fill opacity=0.97, every node/.append style={thick, transform shape, scale=1.8}]
\twocube{57}{-114};
\onebarcube{42.75}{-42.75};
\blankcube{85.5}{-57};
\twocube{57}{0};
\end{tikzpicture}
\end{tabular}
\caption{}
\end{subfigure}
\begin{subfigure}[b]{0.16\textwidth}
\centering
\begin{tabular}[c]{c}
\begin{tikzpicture}   [x={(-0.5cm,-0.5cm)}, y={(1cm,0cm)}, z={(0cm,1cm)}, 
    scale=0.35,fill opacity=0.97, every node/.append style={thick, transform shape, scale=2}]
\onecube{-42.75}{-71.25};
\twocube{0}{0};
\onecube{-42.75}{42.75};
\end{tikzpicture}
\end{tabular}
\caption{}
\end{subfigure}
\begin{subfigure}[b]{0.16\textwidth}
\centering
\begin{tabular}[c]{c}
\begin{tikzpicture}   [x={(-0.5cm,-0.5cm)}, y={(1cm,0cm)}, z={(0cm,1cm)}, 
    scale=0.35,fill opacity=0.97, every node/.append style={thick, transform shape, scale=2}]
\onebarcube{-14.25}{-42.75};
\twocube{0}{0};
\onebarcube{-14.25}{71.25};
\end{tikzpicture}
\end{tabular}
\caption{}
\end{subfigure}
\begin{subfigure}[b]{0.16\textwidth}
\centering
\begin{tabular}[c]{c}
\begin{tikzpicture}
    [x={(-0.5cm,-0.5cm)}, y={(1cm,0cm)}, z={(0cm,1cm)}, 
    scale=0.35,fill opacity=0.97, every node/.append style={thick, transform shape, scale=1.8}]
\icube{57}{-114};
\blankcube{28.5}{-57};
\ipluscube{85.5}{-57};
\icube{57}{0};
\end{tikzpicture}
\end{tabular}
\caption{}
\end{subfigure}
\begin{subfigure}[b]{0.16\textwidth}
\centering
\begin{tabular}[c]{c}
\begin{tikzpicture}
    [x={(-0.5cm,-0.5cm)}, y={(1cm,0cm)}, z={(0cm,1cm)}, 
    scale=0.35,fill opacity=0.97, every node/.append style={thick, transform shape, scale=1.8}]
\ipluscube{57}{-114};
\blankcube{85.5}{-57};
\icube{28.5}{-57};
\ipluscube{57}{0};
\end{tikzpicture}
\end{tabular}
\caption{}
\end{subfigure}
\caption{Impermissible convex subheaps for elements in $\FC(D_n)$.}
\label{fig:notFCD}
\end{figure}

We say that a fully commutative heap of type $D_{n}$  is of \emph{type I} if the corresponding $w\in \FC(D_{n})$ has $s_{\overline{1}}s_{1}$ as a subword of some reduced expression for $w$. Otherwise, the fully commutative heap is of \emph{type II}.

\begin{example} 
Figure~\ref{fig:typeone} depicts all of the type I heaps for the Coxeter group of type $D_{4}$ and Figure~\ref{fig:typetwo} shows all of the type II heaps for the Coxeter group of type $D_{4}$.
\end{example}

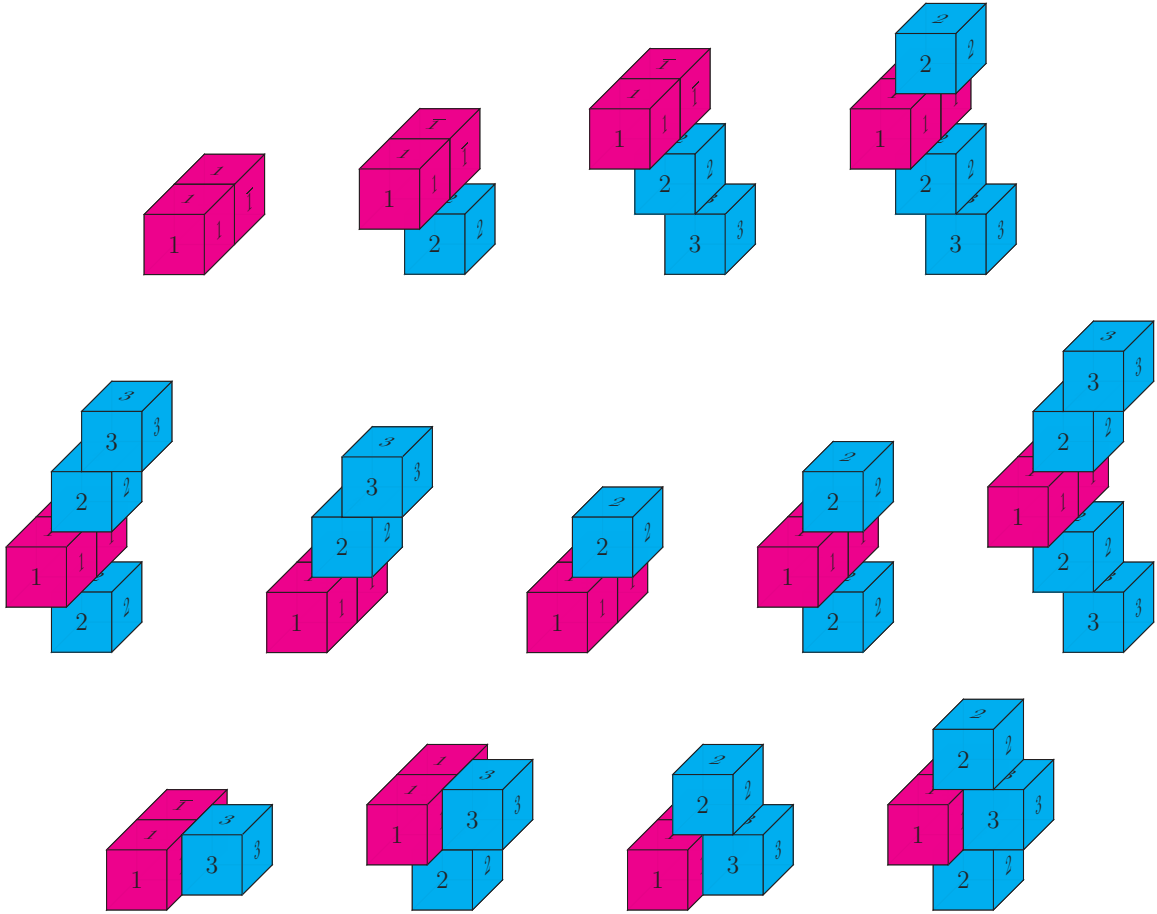
\begin{figure}[ht]
\centering
\begin{tabular}{llll}
\begin{tikzpicture}
    [x={(-0.5cm,-0.5cm)}, y={(1cm,0cm)}, z={(0cm,1cm)}, 
    scale=0.4,fill opacity=0.97, every node/.append style={thick, transform shape, scale=2}]
\coloronebarcube{-14.25}{71.25};
\coloronecube{-42.75}{42.75};
\end{tikzpicture}
&
\quad  \quad 
\begin{tikzpicture}
    [x={(-0.5cm,-0.5cm)}, y={(1cm,0cm)}, z={(0cm,1cm)}, 
    scale=0.4,fill opacity=0.97, every node/.append style={thick, transform shape, scale=2}]
\twocube{0}{0};
\coloronebarcube{-14.25}{71.25};
\coloronecube{-42.75}{42.75};
\end{tikzpicture}
&
\quad  \quad 
\begin{tikzpicture}
    [x={(-0.5cm,-0.5cm)}, y={(1cm,0cm)}, z={(0cm,1cm)}, 
    scale=0.4,fill opacity=0.97, every node/.append style={thick, transform shape, scale=2}]
\threecube{28.5}{-57};
\twocube{0}{0};
\coloronebarcube{-14.25}{71.25};
\coloronecube{-42.75}{42.75};
\end{tikzpicture}
&
\quad  \quad 
\begin{tikzpicture}
    [x={(-0.5cm,-0.5cm)}, y={(1cm,0cm)}, z={(0cm,1cm)}, 
    scale=0.4,fill opacity=0.97, every node/.append style={thick, transform shape, scale=2}]
\threecube{28.5}{-57};
\twocube{0}{0};
\coloronebarcube{-14.25}{71.25};
\coloronecube{-42.75}{42.75};
\twocube{0}{114};
\end{tikzpicture}
\end{tabular}

\bigskip

\begin{tabular}{lllll}
\begin{tikzpicture}
    [x={(-0.5cm,-0.5cm)}, y={(1cm,0cm)}, z={(0cm,1cm)}, 
    scale=0.4,fill opacity=0.97, every node/.append style={thick, transform shape, scale=2}]
\twocube{0}{0};
\coloronebarcube{-14.25}{71.25};
\coloronecube{-42.75}{42.75};
\twocube{0}{114};
\threecube{28.5}{171};
\end{tikzpicture}
&
\quad  \quad 
\begin{tikzpicture}
    [x={(-0.5cm,-0.5cm)}, y={(1cm,0cm)}, z={(0cm,1cm)}, 
    scale=0.4,fill opacity=0.97, every node/.append style={thick, transform shape, scale=2}]
\coloronebarcube{-14.25}{71.25};
\coloronecube{-42.75}{42.75};
\twocube{0}{114};
\threecube{28.5}{171};
\end{tikzpicture}
&
\quad  \quad 
\begin{tikzpicture}
    [x={(-0.5cm,-0.5cm)}, y={(1cm,0cm)}, z={(0cm,1cm)}, 
    scale=0.4,fill opacity=0.97, every node/.append style={thick, transform shape, scale=2}]
\coloronebarcube{-14.25}{71.25};
\coloronecube{-42.75}{42.75};
\twocube{0}{114};
\end{tikzpicture}
&
\quad  \quad 
\begin{tikzpicture}
    [x={(-0.5cm,-0.5cm)}, y={(1cm,0cm)}, z={(0cm,1cm)}, 
    scale=0.4,fill opacity=0.97, every node/.append style={thick, transform shape, scale=2}]
\twocube{0}{0};
\coloronebarcube{-14.25}{71.25};
\coloronecube{-42.75}{42.75};
\twocube{0}{114};
\end{tikzpicture}
&
\quad  \quad 
\begin{tikzpicture}
    [x={(-0.5cm,-0.5cm)}, y={(1cm,0cm)}, z={(0cm,1cm)}, 
    scale=0.4,fill opacity=0.97, every node/.append style={thick, transform shape, scale=2}]
\threecube{28.5}{-57};
\twocube{0}{0};
\coloronebarcube{-14.25}{71.25};
\coloronecube{-42.75}{42.75};
\twocube{0}{114};
\threecube{28.5}{171};
\end{tikzpicture}
\end{tabular}

\bigskip

\begin{tabular}{llll}
\begin{tikzpicture}
    [x={(-0.5cm,-0.5cm)}, y={(1cm,0cm)}, z={(0cm,1cm)}, 
    scale=0.4,fill opacity=0.97, every node/.append style={thick, transform shape, scale=2}]
\coloronebarcube{-14.25}{71.25};
\coloronecube{-42.75}{42.75};
\threecube{28.5}{57};
\end{tikzpicture}
&
\quad  \quad 
\begin{tikzpicture}
    [x={(-0.5cm,-0.5cm)}, y={(1cm,0cm)}, z={(0cm,1cm)}, 
    scale=0.4,fill opacity=0.97, every node/.append style={thick, transform shape, scale=2}]
\twocube{0}{0};
\coloronebarcube{-14.25}{71.25};
\coloronecube{-42.75}{42.75};
\threecube{28.5}{57};
\end{tikzpicture}
&
\quad  \quad 
\begin{tikzpicture}
    [x={(-0.5cm,-0.5cm)}, y={(1cm,0cm)}, z={(0cm,1cm)}, 
    scale=0.4,fill opacity=0.97, every node/.append style={thick, transform shape, scale=2}]
\coloronebarcube{-14.25}{71.25};
\coloronecube{-42.75}{42.75};
\threecube{28.5}{57};
\twocube{0}{114};
\end{tikzpicture}
&
\quad  \quad 
\begin{tikzpicture}
    [x={(-0.5cm,-0.5cm)}, y={(1cm,0cm)}, z={(0cm,1cm)}, 
    scale=0.4,fill opacity=0.97, every node/.append style={thick, transform shape, scale=2}]
\twocube{0}{0};
\coloronebarcube{-14.25}{71.25};
\coloronecube{-42.75}{42.75};
\threecube{28.5}{57};
\twocube{0}{114};
\end{tikzpicture}
\end{tabular}
\caption{Type I heaps for the Coxeter group of type $D_{4}$.}
\label{fig:typeone}
\end{figure}

\begin{figure}[!ht]
\centering
\begin{tabular}{lllll}
\begin{tikzpicture}
    [x={(-0.5cm,-0.5cm)}, y={(1cm,0cm)}, z={(0cm,1cm)}, 
    scale=0.3,fill opacity=0.97, every node/.append style={thick, transform shape, scale=2}]
\blankcube{0}{0};
\end{tikzpicture}
&
\quad  \quad 
\begin{tikzpicture}
    [x={(-0.5cm,-0.5cm)}, y={(1cm,0cm)}, z={(0cm,1cm)}, 
    scale=0.3,fill opacity=0.97, every node/.append style={thick, transform shape, scale=2}]
\onecube{-42.75}{-71.75};
\end{tikzpicture}
&
\quad  \quad 
\begin{tikzpicture}
    [x={(-0.5cm,-0.5cm)}, y={(1cm,0cm)}, z={(0cm,1cm)}, 
    scale=0.3,fill opacity=0.97, every node/.append style={thick, transform shape, scale=2}]
\onebarcube{-14.25}{-42.75};
\end{tikzpicture}
&
\quad  \quad 
\begin{tikzpicture}
    [x={(-0.5cm,-0.5cm)}, y={(1cm,0cm)}, z={(0cm,1cm)}, 
    scale=0.3,fill opacity=0.97, every node/.append style={thick, transform shape, scale=2}]
\twocube{0}{0};
\end{tikzpicture}
&
\quad  \quad 
\begin{tikzpicture}
    [x={(-0.5cm,-0.5cm)}, y={(1cm,0cm)}, z={(0cm,1cm)}, 
    scale=0.3,fill opacity=0.97, every node/.append style={thick, transform shape, scale=2}]
\threecube{28.5}{57};
\end{tikzpicture}
\\
\\

\begin{tikzpicture}
    [x={(-0.5cm,-0.5cm)}, y={(1cm,0cm)}, z={(0cm,1cm)}, 
    scale=0.3,fill opacity=0.97, every node/.append style={thick, transform shape, scale=2}]
\twocube{0}{0};
\onecube{-42.75}{42.75};
\end{tikzpicture}
&
\quad  \quad 
\begin{tikzpicture}
    [x={(-0.5cm,-0.5cm)}, y={(1cm,0cm)}, z={(0cm,1cm)}, 
    scale=0.3,fill opacity=0.97, every node/.append style={thick, transform shape, scale=2}]
\twocube{0}{0};
\onebarcube{-14.25}{71.25};
\end{tikzpicture}
&
\quad  \quad
\begin{tikzpicture}
    [x={(-0.5cm,-0.5cm)}, y={(1cm,0cm)}, z={(0cm,1cm)}, 
    scale=0.3,fill opacity=0.97, every node/.append style={thick, transform shape, scale=2}]
\onecube{-42.75}{42.75};
\twocube{0}{114};
\end{tikzpicture}
&
\quad  \quad 
\begin{tikzpicture}
    [x={(-0.5cm,-0.5cm)}, y={(1cm,0cm)}, z={(0cm,1cm)}, 
    scale=0.3,fill opacity=0.97, every node/.append style={thick, transform shape, scale=2}]
\onebarcube{-14.25}{71.25};
\twocube{0}{114};
\end{tikzpicture}
&
\quad  \quad
\begin{tikzpicture}
    [x={(-0.5cm,-0.5cm)}, y={(1cm,0cm)}, z={(0cm,1cm)}, 
    scale=0.3,fill opacity=0.97, every node/.append style={thick, transform shape, scale=2}]
\onecube{-42.75}{42.75};
\threecube{28.5}{57};
\end{tikzpicture}
\\
\\

\begin{tikzpicture}
    [x={(-0.5cm,-0.5cm)}, y={(1cm,0cm)}, z={(0cm,1cm)}, 
    scale=0.3,fill opacity=0.97, every node/.append style={thick, transform shape, scale=2}]
\onebarcube{-14.25}{71.25};
\threecube{28.5}{57};
\end{tikzpicture}

&
\quad  \quad 
\begin{tikzpicture}
    [x={(-0.5cm,-0.5cm)}, y={(1cm,0cm)}, z={(0cm,1cm)}, 
    scale=0.3,fill opacity=0.97, every node/.append style={thick, transform shape, scale=2}]
\threecube{28.5}{57};
\twocube{0}{114};
\end{tikzpicture}

&
\quad  \quad 
\begin{tikzpicture}
    [x={(-0.5cm,-0.5cm)}, y={(1cm,0cm)}, z={(0cm,1cm)}, 
    scale=0.3,fill opacity=0.97, every node/.append style={thick, transform shape, scale=2}]
\twocube{0}{0};
\threecube{28.5}{57};
\end{tikzpicture}
&
\quad  \quad 
\begin{tikzpicture}
    [x={(-0.5cm,-0.5cm)}, y={(1cm,0cm)}, z={(0cm,1cm)}, 
    scale=0.3,fill opacity=0.97, every node/.append style={thick, transform shape, scale=2}]
\onecube{-42.75}{42.75};
\twocube{0}{114};
\onebarcube{-14.25}{185.25};
\end{tikzpicture}
&
\quad  \quad 
\begin{tikzpicture}
    [x={(-0.5cm,-0.5cm)}, y={(1cm,0cm)}, z={(0cm,1cm)}, 
    scale=0.3,fill opacity=0.97, every node/.append style={thick, transform shape, scale=2}]
\onebarcube{-14.25}{71.25};
\twocube{0}{114};
\onecube{-42.75}{156.75};
\end{tikzpicture}
\\
\\

\begin{tikzpicture}
    [x={(-0.5cm,-0.5cm)}, y={(1cm,0cm)}, z={(0cm,1cm)}, 
    scale=0.3,fill opacity=0.97, every node/.append style={thick, transform shape, scale=2}]
\threecube{28.5}{57};
\twocube{0}{114};
\onecube{-42.75}{156.75};
\end{tikzpicture}
&
\quad  \quad
\begin{tikzpicture}
    [x={(-0.5cm,-0.5cm)}, y={(1cm,0cm)}, z={(0cm,1cm)}, 
    scale=0.3,fill opacity=0.97, every node/.append style={thick, transform shape, scale=2}]
\threecube{28.5}{57};
\twocube{0}{114};
\onebarcube{-14.25}{185.25};
\end{tikzpicture}
&
\quad  \quad
\begin{tikzpicture}
    [x={(-0.5cm,-0.5cm)}, y={(1cm,0cm)}, z={(0cm,1cm)}, 
    scale=0.3,fill opacity=0.97, every node/.append style={thick, transform shape, scale=2}]
\onecube{-42.75}{-71.75};
\twocube{0}{0};
\threecube{28.5}{57};
\end{tikzpicture}
&
\quad  \quad
\begin{tikzpicture}
    [x={(-0.5cm,-0.5cm)}, y={(1cm,0cm)}, z={(0cm,1cm)}, 
    scale=0.3,fill opacity=0.97, every node/.append style={thick, transform shape, scale=2}]
\onebarcube{-14.25}{-42.75};
\twocube{0}{0};
\threecube{28.5}{57};
\end{tikzpicture}
&
\quad  \quad
\begin{tikzpicture}
    [x={(-0.5cm,-0.5cm)}, y={(1cm,0cm)}, z={(0cm,1cm)}, 
    scale=0.3,fill opacity=0.97, every node/.append style={thick, transform shape, scale=2}]
\twocube{0}{0};
\onecube{-42.75}{42.75};
\threecube{28.5}{57};
\end{tikzpicture}
\\
\\

\begin{tikzpicture}
    [x={(-0.5cm,-0.5cm)}, y={(1cm,0cm)}, z={(0cm,1cm)}, 
    scale=0.3,fill opacity=0.97, every node/.append style={thick, transform shape, scale=2}]
\onecube{-42.75}{42.75};
\threecube{28.5}{57};
\twocube{0}{114};
\end{tikzpicture}
&
\quad  \quad
\begin{tikzpicture}
    [x={(-0.5cm,-0.5cm)}, y={(1cm,0cm)}, z={(0cm,1cm)}, 
    scale=0.3,fill opacity=0.97, every node/.append style={thick, transform shape, scale=2}]
\twocube{0}{0};
\onebarcube{-14.25}{71.25};
\threecube{28.5}{57};
\end{tikzpicture}
&
\quad  \quad
\begin{tikzpicture}
    [x={(-0.5cm,-0.5cm)}, y={(1cm,0cm)}, z={(0cm,1cm)}, 
    scale=0.3,fill opacity=0.97, every node/.append style={thick, transform shape, scale=2}]
\onebarcube{-14.25}{71.25};
\threecube{28.5}{57};
\twocube{0}{114};
\end{tikzpicture}
&
\quad  \quad
\begin{tikzpicture}
    [x={(-0.5cm,-0.5cm)}, y={(1cm,0cm)}, z={(0cm,1cm)}, 
    scale=0.3,fill opacity=0.97, every node/.append style={thick, transform shape, scale=2}]
\twocube{0}{0};
\onecube{-42.75}{42.75};
\threecube{28.5}{57};
\twocube{0}{114};
\end{tikzpicture}
&
\quad  \quad
\begin{tikzpicture}
    [x={(-0.5cm,-0.5cm)}, y={(1cm,0cm)}, z={(0cm,1cm)}, 
    scale=0.3,fill opacity=0.97, every node/.append style={thick, transform shape, scale=2}]
\twocube{0}{0};
\onebarcube{-14.25}{71.25};
\threecube{28.5}{57};
\twocube{0}{114};
\end{tikzpicture}
\\
\\

\begin{tikzpicture}
    [x={(-0.5cm,-0.5cm)}, y={(1cm,0cm)}, z={(0cm,1cm)}, 
    scale=0.3,fill opacity=0.97, every node/.append style={thick, transform shape, scale=2}]
\onebarcube{-14.25}{-42.75};
\twocube{0}{0};
\onecube{-42.75}{42.75};
\threecube{28.5}{57};
\end{tikzpicture}
&
\quad  \quad
\begin{tikzpicture}
    [x={(-0.5cm,-0.5cm)}, y={(1cm,0cm)}, z={(0cm,1cm)}, 
    scale=0.3,fill opacity=0.97, every node/.append style={thick, transform shape, scale=2}]
\onecube{-42.75}{-71.75};
\twocube{0}{0};
\onebarcube{-14.25}{71.25};
\threecube{28.5}{57};
\end{tikzpicture}
&
\quad  \quad
\begin{tikzpicture}
    [x={(-0.5cm,-0.5cm)}, y={(1cm,0cm)}, z={(0cm,1cm)}, 
    scale=0.3,fill opacity=0.97, every node/.append style={thick, transform shape, scale=2}]
\onecube{-42.75}{42.75};
\threecube{28.5}{57};
\twocube{0}{114};
\onebarcube{-14.25}{185.25};
\end{tikzpicture}
&
\quad  \quad
\begin{tikzpicture}
    [x={(-0.5cm,-0.5cm)}, y={(1cm,0cm)}, z={(0cm,1cm)}, 
    scale=0.3,fill opacity=0.97, every node/.append style={thick, transform shape, scale=2}]
\onebarcube{-14.25}{71.25};
\threecube{28.5}{57};
\twocube{0}{114};
\onecube{-42.75}{156.75};
\end{tikzpicture}
&
\quad  \quad
\begin{tikzpicture}
    [x={(-0.5cm,-0.5cm)}, y={(1cm,0cm)}, z={(0cm,1cm)}, 
    scale=0.3,fill opacity=0.97, every node/.append style={thick, transform shape, scale=2}]
\twocube{0}{0};
\onecube{-42.75}{42.75};
\threecube{28.5}{57};
\twocube{0}{114};
\onebarcube{-14.25}{185.25};
\end{tikzpicture}
\\
\\

\begin{tikzpicture}
    [x={(-0.5cm,-0.5cm)}, y={(1cm,0cm)}, z={(0cm,1cm)}, 
    scale=0.3,fill opacity=0.97, every node/.append style={thick, transform shape, scale=2}]
\twocube{0}{0};
\onebarcube{-14.25}{71.25};
\threecube{28.5}{57};
\twocube{0}{114};
\onecube{-42.75}{156.75};
\end{tikzpicture}
&
\quad  \quad
\begin{tikzpicture}
    [x={(-0.5cm,-0.5cm)}, y={(1cm,0cm)}, z={(0cm,1cm)}, 
    scale=0.3,fill opacity=0.97, every node/.append style={thick, transform shape, scale=2}]
\onebarcube{-14.25}{-42.75};
\twocube{0}{0};
\onecube{-42.75}{42.75};
\threecube{28.5}{57};
\twocube{0}{114};
\end{tikzpicture}
&
\quad  \quad
\begin{tikzpicture}
    [x={(-0.5cm,-0.5cm)}, y={(1cm,0cm)}, z={(0cm,1cm)}, 
    scale=0.3,fill opacity=0.97, every node/.append style={thick, transform shape, scale=2}]
\onecube{-42.75}{-71.75};
\twocube{0}{0};
\onebarcube{-14.25}{71.25};
\threecube{28.5}{57};
\twocube{0}{114};
\end{tikzpicture}
&
\quad  \quad
\begin{tikzpicture}
    [x={(-0.5cm,-0.5cm)}, y={(1cm,0cm)}, z={(0cm,1cm)}, 
    scale=0.3,fill opacity=0.97, every node/.append style={thick, transform shape, scale=2}]
\onecube{-42.75}{-71.75};
\twocube{0}{0};
\onebarcube{-14.25}{71.25};
\threecube{28.5}{57};
\twocube{0}{114};
\onecube{-42.75}{156.75};
\end{tikzpicture}
&
\quad  \quad
\begin{tikzpicture}
    [x={(-0.5cm,-0.5cm)}, y={(1cm,0cm)}, z={(0cm,1cm)}, 
    scale=0.3,fill opacity=0.97, every node/.append style={thick, transform shape, scale=2}]
\onebarcube{-14.25}{-42.75};
\twocube{0}{0};
\onecube{-42.75}{42.75};
\threecube{28.5}{57};
\twocube{0}{114};
\onebarcube{-14.25}{185.25};
\end{tikzpicture}
\end{tabular}
\caption{Type II heaps for the Coxeter group of type $D_{4}$.}
\label{fig:typetwo}
\end{figure}
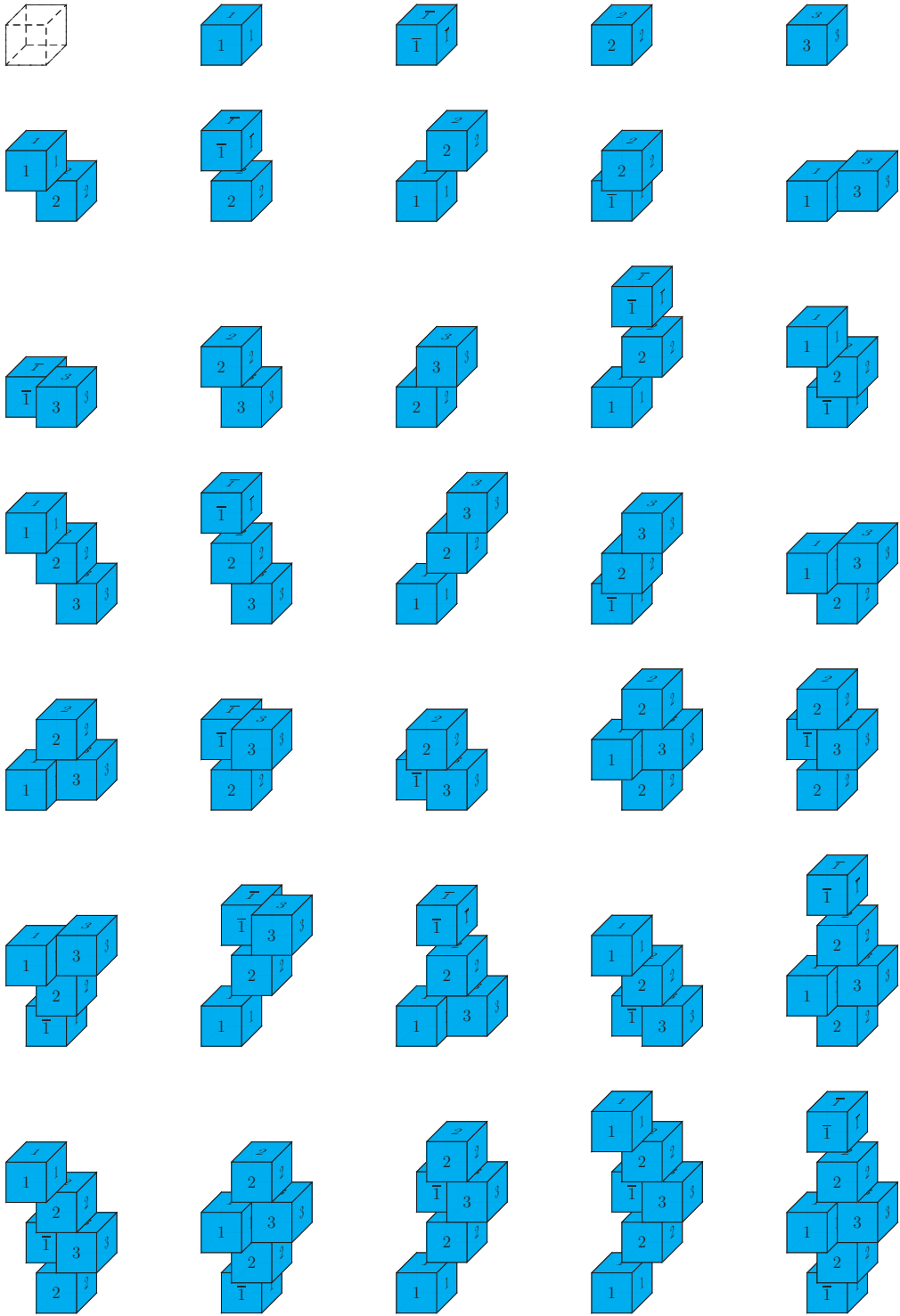


\section{Hecke algebras}\label{hecke}


Let $(W,S)$ be an arbitrary Coxeter system associated to the Coxeter graph $\Gamma$. We define the $\Z[q,q^{-1}]$-algebra, $\H_q(\Gamma)$, with basis consisting of elements $T_w$, for all $w\in W$, satisfying
\[
T_{s}T_{w}:=\begin{cases}
T_{sw} & \text{if }\ell\left(sw\right)>\ell\left(w\right),\\
qT_{sw}+\left(q-1\right)T_{w} & \text{otherwise}
\end{cases}
\]
where $s\in S$ and $w\in W$. If $\w=s_{x_1}s_{x_2}\cdots s_{x_k}$ is a reduced expression for $w\in W$, then $T_{w}=T_{s_{x_1}}T_{s_{x_2}}\cdots T_{s_{x_k}}$. We will abbreviate $T_{w}$ with $T_{x_1x_2\cdots x_k}$. In particular, $T_{s_i}$ will be written as $T_i$.

It is convenient to extend the scalars of $\H_q(\Gamma)$ to produce an $\A$-algebra, $\H(\Gamma)=\A \otimes_{\Z[q,q^{-1}]}\H_q(\Gamma)$, where $\A$ is the ring of Laurent polynomials, $\mathbb{Z}\left[v,v^{-1}\right],$ and $v=q^{\frac{1}{2}}$ to obtain the \emph{Hecke algebra} of type $\Gamma$ denoted by $\H(\Gamma)$.

\begin{example}
Let $\w_{1} =s_{1}s_{2}s_{1}s_{3}$ and $\w_{2} = s_{3}s_{2}s_{3}$ be reduced expressions for $w_1,w_2\in W(A_3)$. We wish to calculate $T_{w_2}T_{w_1}$. Observe that each of $s_3s_1s_2s_1s_3$ and $s_2s_3s_1s_2s_1s_3$ are reduced. However, we see that 
\begin{align*}
s_3s_2s_3s_1s_2s_1s_3&=s_2s_3s_2s_1s_2s_1s_3\\
&=s_2s_3s_1s_2s_1s_1s_3\\
&=s_2s_3s_1s_2s_3,
\end{align*} where the last expression is reduced. This implies that
\begin{align*}
T_{w_{2}}T_{w_{1}}&=  T_{{3}{2}{3}}T_{{1}{2}{1}{3}}\\
&=  T_{{3}}T_{{2}}T_{{3}}T_{{1}{2}{1}{3}}\\
&=  T_{{3}}T_{{2}}T_{{3}{1}{2}{1}{3}} \\
&=  T_{{3}}T_{{2}{3}{1}{2}{1}{3}}\\
&=  qT_{{2}{3}{1}{2}{3}}+\left(q-1\right)T_{{2}{3}{1}{2}{1}{3}}.
\end{align*}
\end{example}

The set $\{ T_{w}:w\in W\}$ is is the natural basis for $\H(\Gamma)$ but there is another remarkable basis $\{C'_w:w\in W\}$, where 
\[
C'_s=v^{-1}T_s+v^{-1}T_{e}
\]
for each $s\in S$. The following theorem defines the basis element $C'_w\in \H(\Gamma)$. 

\begin{theorem}[Kazhdan, Lusztig~\cite{Kazhdan1979}]
There is a unique element $C'_{w}\in \H$ such that 
\[
C'_{w}=\sum_{x\le w}v^{-l(w)}P_{x,w}T_{x},
\]
where $\le$ is the Bruhat ordering on the Coxeter group $W, P_{x,w}\in\Z \left[v^{-1}\right]$ if $x<w$, and $P_{w,w}=1.$ \qed
\end{theorem}

The polynomials, $P_{x,w}$, are known as the \emph{Kazhdan--Lusztig polynomials} and the set $\{C'_{w}:w\in W\}$ is known as the \emph{Kazhdan--Lusztig basis} and has multiplication determined by
\[
C^{\prime}_{s}C^{\prime}_{w}:=
\begin{cases}
\left(v+v^{-1}\right)C^{\prime}_{sw}, & \text{if }\ell\left(sw\right)>\ell\left(w\right)\\
C^{\prime}_{sw}+\sum \mu \left(s,w\right) C^{\prime}_{s}, & \text{otherwise }
\end{cases}
\]
where $\mu \left(s,w\right)$ is the leading coefficient of $P_{s,w}$.


\section{Temperley--Lieb algebras}\label{sec:TL}


Let $(W,S)$ be a Coxeter system with graph $\Gamma$. Next, we define a quotient of $\H(\Gamma)$, called the Temperley--Lieb algebra of type $\Gamma$.  

Define $J(\Gamma)$ to be
the two-sided ideal of $\H(\Gamma)$ generated by 
\[
\sum_{w\in\langle s,s'\rangle}T_{w},
\]
where $(s,s')$ runs over all pairs of elements of $S$ with $3\leq m(s,s')<\infty$,
and $\langle s,s'\rangle$ is the (parabolic) subgroup generated by
$s$ and $s'$.

\begin{example}
In type $A_{3}$, 
\[
\langle s_{1},s_{2}\rangle=\{e,s_{1},s_{2},s_{1}s_{2},s_{2}s_{1},s_{1}s_{2}s_{1}\} \text{ and } \langle s_{2},s_{3}\rangle=\{e,s_{2},s_{3},s_{2}s_{3},s_{3}s_{2},s_{2}s_{3}s_{2}\},
\]
so $J(A_3)$ is generated by
\[
\{T_{e}+T_{1}+T_{2}+T_{12}+T_{21}+T_{121}, T_{e}+T_{2}+T_{3}+T_{23}+T_{32}+T_{232}\}.
\]
\end{example}

\begin{definition}
The Temperley--Lieb algebra of type $\Gamma$, $TL(\Gamma)$, is defined to be the quotient algebra $\H(\Gamma)/J(\Gamma)$. 
\end{definition}

\begin{theorem}[Graham~\cite{Graham1995}]\label{t-basis}
Let $t_{w}$ denote the image of $T_{w}$ in the quotient.
Then $\{t_{w}:w\in FC(W)\}$ is a basis for $TL(\Gamma)$, called the t-basis. \qed
\end{theorem}

\begin{theorem}[Green~\cite{Green1999}]\label{c-basis}
Let $c_{w}$ denote the image of $C'_{w}$ in the quotient.
Then $\{c_{w}:w\in FC(W)\}$ is a basis for $TL(\Gamma)$, called the canonical basis. \qed
\end{theorem}

\begin{definition}
For each $s_{i}\in S$, define $b_{i}=v^{-1}t_{i}+v^{-1}$. If
$w\in FC(\Gamma)$ has reduced expression $\overline{w}=s_{x_{1}}\cdots s_{x_{m}}$,
define 
\[
b_{w}=b_{x_{1}}\cdots b_{x_{m}}.
\]
The \emph{monomial basis} is then defined as the set  $\left\{ b_{w}:w\in FC(\Gamma)\right\}$.
\end{definition}

\begin{theorem}[Graham~\cite{Graham1995}]
The monomial basis forms a basis for $TL(\Gamma)$. 
\end{theorem}

\begin{remark}
In~\cite{Green1999}, it is shown that in type $D_n$, the canonical basis is equal to the monomial basis.
\end{remark}

Now we will present the \emph{Temperley--Lieb algebra of type $D_{n}$} in terms of generators and relations.

\begin{theorem}[Green~\cite{Green2006a}]\label{def:TL(D)}
The algebra $\TL(D_{n})$ where $n\ge4$  is the unital $\Z[\delta]$-algebra generated by $b_{\overline{1}},b_{1},b_{2},\ldots,b_{n-1}$ with defining relations
\begin{enumerate}
\item $b_{i}^{2}=\delta b_{i}$ for all $i$, where $\delta=v+v^{-1}$;
\item $b_{i}b_{j} = b_{j}b_{i}$ if $s_i$ and $s_j$ are not connected in the graph;
\item $b_{i}b_{j}b_{i} = b_{i}$ if $s_i$ and $s_j$ are connected in the graph.
\end{enumerate}
\end{theorem}

\begin{proof}
We will check that the relations hold in type $D_n$, but for the full proof we refer the reader to~\cite[Proposition~2.6]{Green2006a}. Remember that $b_{i}=v^{-1}t_{i}+v^{-1}$.
\begin{enumerate}
\item We see that
\begin{align*}
b_{i}^{2} &= \left(v^{-1}t_{i}+v^{-1}\right)\left(v^{-1}t_{i}+v^{-1}\right)\\
&= v^{-2}\left(t_{i}^{2}+2t_{i}+1\right)\\
&= v^{-2}\left(v^{2}+\left(v^{2}-1\right)t_{i}+2t_{i}+1\right)\\
&= v^{-2}\left(v^{2}+v^{2}t_{i}+t_{i}+1\right)\\
&= v^{-2}t_{i}+v^{-2}+t_{i}+1\\
&= \left(v^{-1}+v\right)\left(v^{-1}t_{i}+v^{-1}\right)\\
&=\delta b_{i},
\end{align*}
since $\delta=v+v^{-1}$.
\item Assume that $m(s_i,s_j)=2$. Note that since $s_is_j=s_js_i, t_{ij}=t_{ji}.$ We see that
\begin{align*}
b_{i}b_{j} &= \left(v^{-1}t_{i}+v^{-1}\right)\left(v^{-1}t_{j}+v^{-1}\right)\\
&= v^{-2}\left(t_{ij}+t_{i}+t_{j}+1\right)\\
&= v^{-2}\left(t_{ji}+t_{i}+t_{j}+1\right)\\
&= \left(v^{-1}t_{j}+v^{-1}\right)\left(v^{-1}t_{i}+v^{-1}\right)\\
&= b_{j}b_{i}.
\end{align*}
\item Assume $m(s_i,s_j)=3$. Then note that $T_{iji}+T_{ij}+T_{ji}+T_i+T_j+1\in J(D_n)$. This implies that $t_{iji}+t_{ji}+t_{ij}+t_{i}+t_{j}+1=0$. We see that
 \begin{align*}
b_{i}b_{j}b_{i} &= \left(v^{-1}t_{i}+v^{-1}\right)\left(v^{-1}t_{j}+v^{-1}\right)\left(v^{-1}t_{i}+v^{-1}\right)\\
&= \left(v^{-2}t_{i}t_{j}+v^{-2}t_{i}+v^{-2}t_{j}+v^{-2}\right)\left(v^{-1}t_{i}+v^{-1}\right)\\
&= v^{-3}\left(t_{i}t_{j}t_{i}+t_{i}^{2}+t_{j}t_{i}+t_{i}+t_{i}t_{j}+t_{i}+t_{j}+1\right)\\
&= v^{-3}\left(v^{2}+\left(v^{2}-1\right)t_{i}+t_{i}+\left(t_{iji}+t_{ji}+t_{ij}+t_{i}+t_{j}+1\right)\right)\\
&= v^{-3}\left(v^{2}+\left(v^{2}-1\right)t_{i}+t_{i}\right)+0\\
&= v^{-3}\left(v^{2}+v^{2}t_{i}\right)\\
&= v^{-1}+v^{-1}t_{i}\\
&= b_{i}.
\end{align*}
\end{enumerate}
\end{proof}

\begin{theorem}[Green~\cite{Green2006a}]
The algebra $\TL(A_{n-1})$ is generated as a unital $\Z[\delta]$-algebra by $\{b_{1}, b_{2}, \dots, b_{n-1}\}$ with the same relations as Theorem~\ref{def:TL(D)}. \qed
\end{theorem}

It is known that we can consider $\TL(A_{n-1})$ as a subalgebra of $\TL(D_{n})$ in the obvious way.

\begin{theorem}[Fan~\cite{Fan1997}]
The dimension of $\TL(D_n)$ is 
\[
\left(\frac{n+3}{2}\right)C(n)-1,
\]
where $C(n)$ is the Catalan number defined by
\[
C(n):=\frac{1}{n+1}{2n\choose n}.
\]\qed
\end{theorem}


\chapter{Diagram algebras}\label{ch:diag}


This chapter provides necessary background on diagram algebras and is modeled after~\cite{Ernst2012}.


\begin{section}{Undecorated diagrams}\label{undec}


First, we discuss undecorated diagrams and their corresponding diagram algebras.

\begin{definition}\label{def:k-box}
Let $k$ be a nonnegative integer.  The \emph{standard $k$-box} is a rectangle with $2k$ points, called \emph{nodes},  labeled as in Figure~\ref{Fig073}.  We will refer to the top of the rectangle as the \emph{north face} and the bottom as the \emph{south face}.
\end{definition}

\begin{figure}[!ht]
\centering
\begin{tikzpicture}
\kbox{0};
\draw \foreach \x in {1,2,3} {(\x,0) node[label=above:${\x}$]{}}; 
\draw \foreach \x in {1,2,3} {(\x,-2) node[label=below:${\x}^{\prime}$]{}};
\draw {(5,0) node[label=above:$k$]{}};
\draw {(5,-2) node[label=below:$k^{\prime}$]{}};
\end{tikzpicture}
\caption{The standard $k$-box.}\label{Fig073}
\end{figure}
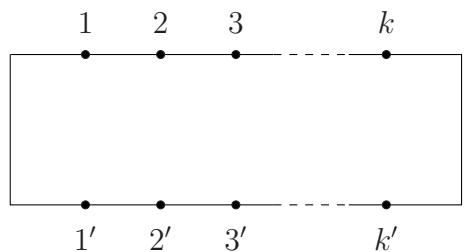

The next definition summarizes the construction of the ordinary Temperley--Lieb pseudo diagrams.  

\begin{definition}\label{def:T_k(emptyset)}
A \emph{concrete pseudo $k$-diagram} consists of a finite number of disjoint curves (planar), called \emph{edges}, embedded in the standard $k$-box with the following restrictions.  
\begin{enumerate}
\item Every node of the box is the endpoint of exactly one edge, which meets the box transversely. 
\item All other edges must be closed (isotopic to circles) and disjoint from the box. 
\end{enumerate}
\end{definition}

\begin{example}\label{ex:pseudo diagram}
The diagram in Figure~\ref{psuedo} is an example of a concrete pseudo 5-diagram, whereas the diagram in Figure~\ref{fig:notp} does not represent a concrete pseudo 5-diagram since the diagram contains edges that are not disjoint (i.e., they intersect), node $4$ is the endpoint for more than one edge, and node $5$ is not an endpoint for any edge.
\end{example}

\begin{figure}[!h]
\centering
\begin{subfigure}[b]{0.4\textwidth}
\centering
\begin{tikzpicture}
\fivebox{0};
\draw (1,0)  arc (-180:0:1.5 and 0.8) ;
\draw (2,0)  arc (-180:0:0.5 and 0.4) ;
\draw (1,-2)  arc (180:0:0.5 and 0.4) ;
\draw (4,-2)  arc (180:0:0.5 and 0.4) ;
\draw (5,0) .. controls (5,-1) and (3,-1) ..  (3,-2);
\lp{-0.5};
\end{tikzpicture}
\caption{A concrete pseudo 5-diagram}
\label{psuedo}
\end{subfigure}
\quad
\begin{subfigure}[b]{0.4\textwidth}
\centering
\begin{tikzpicture}
\fivebox{0};
\draw (2,0)  arc (-180:0:0.5 and 0.4) ;
\draw (2,-2)  arc (180:0:1 and 0.6) ;
\draw (1,0) .. controls (2,-1) and (0,-1) .. (1,-2);
\draw (4,0) .. controls (5,-1) and (4,-1) .. (5,-2);
\draw (4,0) .. controls (5,-1) and (3,-1) .. (3,-2);
\end{tikzpicture}
\caption{Not a concrete pseudo 5-diagram}
\label{fig:notp}
\end{subfigure}
\caption{Examples of diagrams.}
\end{figure}
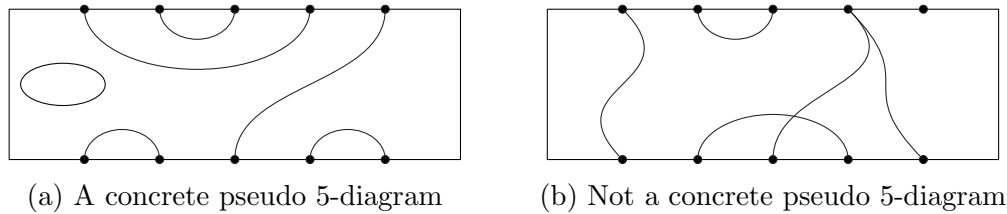

We now define an equivalence relation on the set of concrete pseudo $k$-diagrams. Two concrete pseudo $k$-diagrams are \emph{(isotopically) equivalent} if one concrete diagram can be obtained from the other by isotopically deforming the edges such that any intermediate diagram is also a concrete pseudo $k$-diagram. Note that an isotopy of the k-box is a 1-parameter family of homeomorhisms of the k-box to itself that are stationary on the boundary.

\begin{definition}
A \emph{pseudo $k$-diagram} (or an \emph{ordinary Temperley--Lieb pseudo-diagram}) is defined to be an equivalence class of equivalent concrete pseudo $k$-diagrams.  We denote the set of pseudo $k$-diagrams by $T_{k}(\emptyset)$.
\end{definition}

\begin{remark}\label{vertequiv}
When representing a pseudo $k$-diagram with a drawing, we pick an arbitrary concrete representative among a continuum of equivalent choices. When no confusion can arise, we will not make a distinction between a concrete pseudo $k$-diagram and the equivalence class that it represents. We say that two concrete pseudo $k$-diagrams are \emph{vertically equivalent} if they are equivalent in the above sense by an isotopy that preserves setwise each vertical cross-section of the $k$-box.
\end{remark}

\begin{example}
The concrete pseudo $5$-diagram in Figure~\ref{psuedo} and the concrete pseudo $5$-diagram in Figure~\ref{equivpseudo} are equivalent concrete pseudo $5$-diagrams since the diagram in Figure~\ref{psuedo} can be obtained by isotopically deforming the edges in Figure~\ref{equivpseudo}.
\end{example}

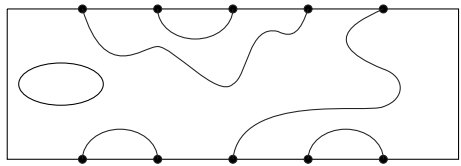
\begin{figure}[!h]
\centering
\begin{tikzpicture}
\fivebox{0};
\draw (1,0)  .. controls (1.3,-1) and (1.8,-0.5)  .. (2,-0.5) .. controls (2.2,-0.5) and (2.8,-1.2) .. (3,-1) .. controls (3.2,-0.9) and (3.2,-0.2) .. (3.6,-0.3) .. controls (3.8,-0.4) and (3.9,-0.3) .. (4,0);
\draw (2,0)  arc (-180:0:0.5 and 0.4) ;
\draw (1,-2)  arc (180:0:0.5 and 0.4) ;
\draw (4,-2)  arc (180:0:0.5 and 0.4) ;
\draw (5,0) .. controls (4.8,-0.1) and (4,-0.4) .. (5,-0.8) .. controls (5.3,-0.9) and (5.3,-1.2) .. (5,-1.3) .. controls (4.8,-1.4) and (3.1,-1.1) .. (3,-2);
\lp{-0.5};
\end{tikzpicture}
\caption{An isotopically equivalent diagram of Figure~\ref{psuedo}.}
\label{equivpseudo}
\end{figure}

Let $d$ be a diagram and let $e$ be an edge of $d$. If $e$ is a closed curve occurring in $d$, then we call $e$ a \emph{loop}.  For example, the diagram in Figure~\ref{psuedo} has a single loop.  If $e$ joins node $i$ in the north face to node $j'$ in the south face, then $e$ is called a \emph{propagating edge from $i$ to $j'$}. If $e$ is not propagating, loop or otherwise, it will be called \emph{non-propagating}.    

Note that we used the word ``pseudo'' in our definition to emphasize that we allow loops to appear in our diagrams. In Section~\ref{decorated}, we will add decorations to our diagrams. The presence of $\emptyset$ in the definition above is to emphasize that the edges of the diagrams are undecorated.

Note that the number of non-propagating edges in the north face of a diagram must be equal to the number of non-propagating edges in the south face.  We define the function $\a: T_{k}(\emptyset) \to \Z^{+}\cup \{0\}$ via
\[
\a(d)=\text{ number of non-propagating edges in the north face of } d
\]
and the function $\mathbf{p}: T_{k}(\emptyset) \to \Z^{+}\cup \{0\}$ via
\[
\mathbf{p}(d)=\text{ number of propagating edges in the north face of } d
\]
where $2\a(d)+\mathbf{p}(d)=k$.
For example, Figure~\ref{psuedo} has two non-propagating edges and one propagating edge in the north face and therefore, $\a(d)=2$ and $\mathbf{p}(d)=1$. There is only one diagram with $\a$-value $0$ having no loops; namely the diagram $d_{e}$ that appears in Figure~\ref{Fig076}.  The maximum value that $\a(d)$ can take is $\lfloor k/2 \rfloor$.  

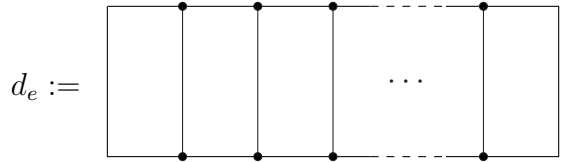
\begin{figure}[!ht]
\centering
$d_{e}:=$
\begin{tabular}[c]{l}
\begin{tikzpicture}
\kbox{0}
\draw (1,0) -- (1,-2);
\draw (2,0) -- (2,-2);
\draw (3,0) -- (3,-2);
\node at (4,-1) {$\cdots$};
\draw (5,0) -- (5,-2);
\end{tikzpicture}
\end{tabular}
\caption{The only diagram having $\a$-value 0 and no loops.}\label{Fig076}
\end{figure}

We wish to define an associative algebra that has the pseudo $k$-diagrams as a basis.

\begin{definition}\label{def:P_k(emptyset)}
Let $R$ be a commutative ring with $1$.  The associative algebra $\P_{k}(\emptyset)$ over $R$ is the free $R$-module having $T_{k}(\emptyset)$ as a basis, with multiplication (referred to as diagram concatenation) defined as follows. We define multiplication in $\P_{k}(\emptyset)$ by defining multiplication in the case where $d$ and $d'$ are basis elements, and then extend bilinearly. If $d, d' \in T_{k}(\emptyset)$, the product $d'd$ is the element of $T_{k}(\emptyset)$ obtained by placing $d'$ on top of $d$, so that node $i'$ of $d'$ coincides with node $i$ of $d$.
\end{definition}

\begin{example}\label{looploop}
Figure~\ref{Fig:looploop} depicts the product of three basis diagrams from $\P_{5}(\emptyset)$.
\end{example}

\begin{figure}[!ht]
\centering
$\begin{tabular}[c]{l}
\begin{tikzpicture}
\fivebox{0};
\draw (1,0)  arc (-180:0:1.5 and 0.8) ;
\draw (2,0)  arc (-180:0:0.5 and 0.4) ;
\draw (1,-2)  arc (180:0:0.5 and 0.4) ;
\draw (4,-2)  arc (180:0:0.5 and 0.4) ;
\draw (5,0) -- (3,-2);
\fivebox{2};
\draw (1,0)  arc (180:0:1.5 and 0.8) ;
\draw (2,0)  arc (180:0:0.5 and 0.4) ;
\draw (1,2)  arc (-180:0:0.5 and 0.4) ;
\draw (4,2)  arc (-180:0:0.5 and 0.4) ;
\draw (5,0) -- (3,2);
\fivebox{4};
\draw (1,4)  arc (-180:0:0.5 and 0.4) ;
\draw (1,2)  arc (180:0:0.5 and 0.4) ;
\draw (3,4) -- (3,2);
\draw (4,4) -- (4,2);
\draw (5,4) -- (5,2);
\end{tikzpicture}
\end{tabular}
=  
\ \begin{tabular}[c]{l}
\begin{tikzpicture}
\fivebox{0};
\draw (0.85,-0.8) arc(-180:0:0.3 and 0.2);
\draw (0.85,-0.8) arc(180:0:0.3 and 0.2);
\draw (1.85,-1.2) arc(-180:0:0.3 and 0.2);
\draw (1.85,-1.2) arc(180:0:0.3 and 0.2);
\draw (1.5,-1.2) arc(-180:0:0.65 and 0.35);
\draw (1.5,-1.2) arc(180:0:0.65 and 0.35);
\draw (1,0)  arc (-180:0:0.5 and 0.4) ;
\draw (1,-2)  arc (180:0:0.5 and 0.4) ;
\draw (4,0)  arc (-180:0:0.5 and 0.4) ;
\draw (4,-2)  arc (180:0:0.5 and 0.4) ;
\draw (3,0) -- (3,-2);
\end{tikzpicture}
\end{tabular}$
\caption{An example of multiplication in $\P_{5}(\emptyset)$.}\label{Fig:looploop}
\end{figure}
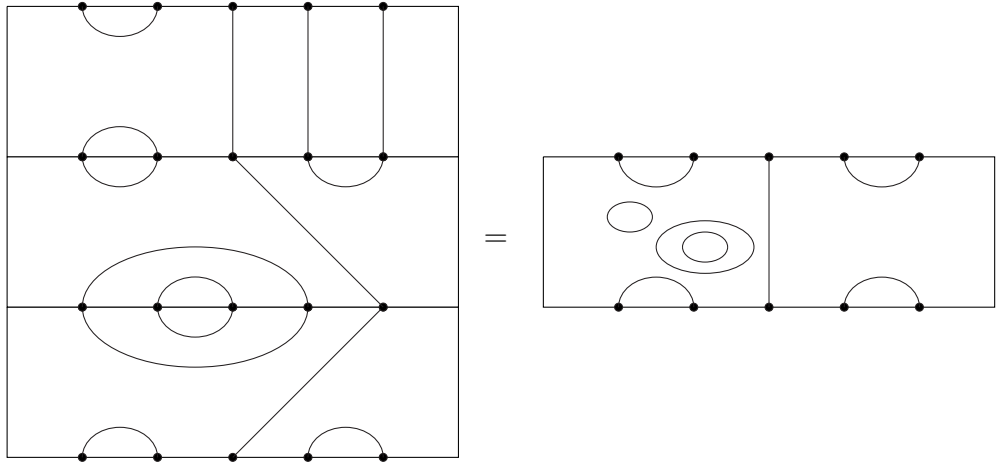

We now restrict our attention to a particular base ring, namely, let $R=\Z[\delta]$, the ring of polynomials in $\delta$ with integer coefficients.
\begin{definition}\label{def:DTL(A_n)}
Let $\DTL(A_{n})$ be the associative $\Z[\delta]$-algebra equal to the quotient of $\P_{n+1}(\emptyset)$ determined by the relation depicted in Figure~\ref{Fig077}.
\end{definition}

\begin{figure}[!ht]
\centering
\begin{tabular}[c]{@{} c@{}}
\begin{tikzpicture}
\lp{0};
\end{tikzpicture}
\end{tabular}
$= \delta$
\caption{The defining relation of $\DTL(A_{n})$.}\label{Fig077}
\end{figure}

It is well-known that $\DTL(A_{n})$ is the free $\Z[\delta]$-module with basis given by the elements of $T_{n+1}(\emptyset)$ having no loops. The multiplication is inherited from the multiplication on $\P_{n+1}(\emptyset)$ except we multiply by a factor of $\delta$ for each resulting loop and then discard the loop.  We will refer to $\DTL(A_{n})$ as the \emph{ordinary Temperley--Lieb diagram algebra}.

\begin{example}
Figure~\ref{Fig078--Fig079} depicts the product of three basis diagrams from $\DTL(A_{4})$. Note that this is the same product of diagrams as in Example~\ref{looploop}, however, in this case the three loops are replaced with the coefficient $\delta^3$.
\end{example}

\begin{figure}[!ht]
\centering
$\begin{tabular}[c]{l}
\begin{tikzpicture}
\fivebox{0};
\draw (1,0)  arc (-180:0:1.5 and 0.8) ;
\draw (2,0)  arc (-180:0:0.5 and 0.4) ;
\draw (1,-2)  arc (180:0:0.5 and 0.4) ;
\draw (4,-2)  arc (180:0:0.5 and 0.4) ;
\draw (5,0) -- (3,-2);
\fivebox{2};
\draw (1,0)  arc (180:0:1.5 and 0.8) ;
\draw (2,0)  arc (180:0:0.5 and 0.4) ;
\draw (1,2)  arc (-180:0:0.5 and 0.4) ;
\draw (4,2)  arc (-180:0:0.5 and 0.4) ;
\draw (5,0) -- (3,2);
\fivebox{4};
\draw (1,4)  arc (-180:0:0.5 and 0.4) ;
\draw (1,2)  arc (180:0:0.5 and 0.4) ;
\draw (3,4) -- (3,2);
\draw (4,4) -- (4,2);
\draw (5,4) -- (5,2);
\end{tikzpicture}
\end{tabular}
= \  \delta^{3} \ \begin{tabular}[c]{l}
\begin{tikzpicture}
\fivebox{0};
\draw (1,0)  arc (-180:0:0.5 and 0.4) ;
\draw (1,-2)  arc (180:0:0.5 and 0.4) ;
\draw (4,0)  arc (-180:0:0.5 and 0.4) ;
\draw (4,-2)  arc (180:0:0.5 and 0.4) ;
\draw (3,0) -- (3,-2);
\end{tikzpicture}
\end{tabular}$
\caption{An example of multiplication in $\DTL(A_{4})$.}\label{Fig078--Fig079}
\end{figure}
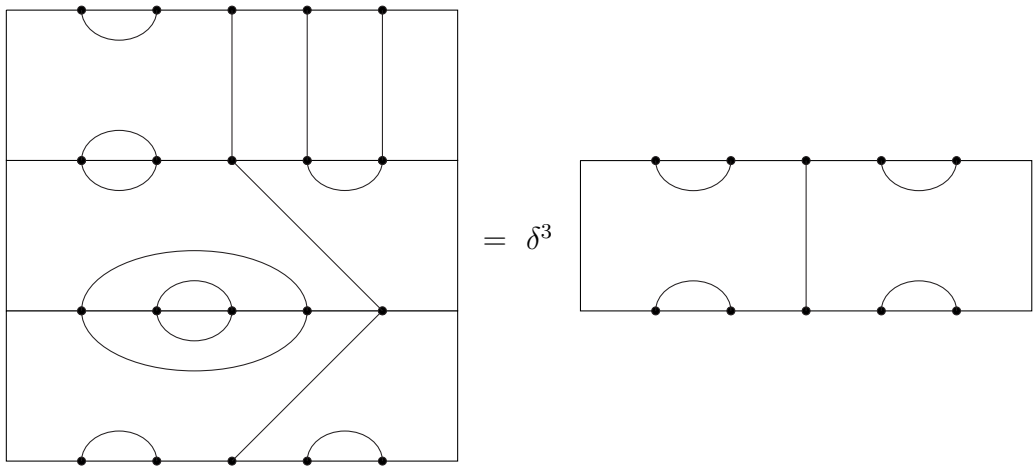

The next theorem describes the connection between $\TL(A_n)$ and $\DTL(A_n)$ shown in~\cite{Kauffman1987} and~\cite{Penrose1971}.

\begin{theorem}[Kaufmann~\cite{Kauffman1987}]\label{kauff}
As $\Z[\delta]$-algebras, the Temperley--Lieb algebra $\TL(A_{n})$ is isomorphic to $\DTL(A_{n})$.  Moreover, each loop-free diagram from $T_{n+1}(\emptyset)$ corresponds to a unique monomial basis element of $\TL(A_{n})$.  \qed
\end{theorem}

\end{section}


\begin{section}{Decorated diagrams}\label{decorated}


We now describe the construction of diagrams whose edges carry decorations. We will use the symbol $\bcirc$, which we refer to as a \emph{decoration}, to adorn the edges of a diagram. Let $\mathbf{b}=x_{1}x_{2}\cdots x_{r}$ be a finite sequence of decorations, where each $x_i=\bcirc$. We say that $\mathbf{b}$ is a \emph{block} of decorations of \emph{width} $r$.  Note that a block of width $1$ is just a single decoration.  The string $\bcirc~\bcirc~\bcirc~\bcirc~\bcirc~\bcirc~\bcirc$ is an example of a block of width 7.

Ultimately, we have three restrictions (D0, D1, D2) for how we allow the edges of a diagram to be decorated by blocks, which we will now outline. Note that we are maintaining consistency with cases involving multiple decoration types as in~\cite{Ernst2012}, and that $\a(d)$ is defined the same way as in type $A_n$. Let $d$ be a fixed concrete pseudo $k$-diagram and let $e$ be an edge of $d$.

\begin{enumerate}[leftmargin=0.5in]
\item[(D0)] If $\a(d)=0$, then $e$ is undecorated.
\end{enumerate}
In particular, the unique diagram $d_{e}$ with $\a$-value 0 and no loops is undecorated.
Subject to some restrictions, if $\a(d)>0$, we may adorn $e$ with a finite sequence of blocks of decorations $\mathbf{b}_{1}, \dots, \mathbf{b}_{m}$ such that adjacency of blocks and decorations of each block is preserved as we travel along $e$.  

If $e$ is a non-loop edge, the convention we adopt is that the decorations of the block are placed so that we can read off the sequence of decorations as we traverse $e$ from $i$ to $j'$ if $e$ is propagating, or from $i$ to $j$ (respectively, $i'$ to $j'$) with $i < j$ (respectively, $i' < j'$) if $e$ is non-propagating.

If $e$ is a loop, reading the corresponding sequence of decorations depends on an arbitrary choice of starting point and direction round the loop.

If $\a(d)\neq 0$, then we also require the following.

\begin{enumerate}[leftmargin=0.5in]
\item[(D1)] We allow adjacent blocks on $e$ to be conjoined to form larger blocks.
\end{enumerate}

\begin{definition}
A \emph{concrete decorated pseudo $k$-diagram} is any concrete pseudo $k$-diagram with decorations satisfying (D0) and (D1).
\end{definition}

\begin{definition}
We define two concrete pseudo decorated $k$-diagrams to be equivalent if we can isotopically deform one diagram into the other such that any intermediate diagram is also a concrete decorated pseudo $k$-diagram. 
\end{definition}

\begin{definition}\label{def:T_{k}(Omega)}
A \emph{decorated pseudo $k$-diagram} is defined to be an equivalence class of equivalent concrete decorated pseudo $k$-diagrams.  We denote the set of decorated diagrams by $T_{k}(\bcirc)$.
Then define $\P_{k}(\bcirc)$ to be the free $\Z[\delta]$-module having the decorated pseudo $k$-diagrams $T_{k}(\bcirc)$ as a basis.  
\end{definition}

We define multiplication in $\P_{k}(\bcirc)$ by concatenating diagrams, conjoining blocks and extending bilinearly (as in Definition~\ref{def:P_k(emptyset)}). It follows from Section 3 of~\cite{Ernst2012} that the multiplication just defined turns $\P_{k}(\bcirc)$ into a well-defined associative $\Z[\delta]$-algebra.
 
\begin{example}\label{ex:decorated diagrams}
Here are a few examples.
\begin{enumerate}[label=(\alph*)]
\item  The diagram in Figure~\ref{Fig080} is an example of a decorated pseudo $5$-diagram. The decorations on the unique propagating edge can be conjoined to form a maximal block of width 4.

\item  The diagram in Figure~\ref{Fig081} is another example of a decorated pseudo $5$-diagram, but with $\a$-value 1.  Note that the decorations can be conjoined to form a block of width 3.

\item Figure~\ref{fig:product} depicts the product of the diagram in Figure~\ref{Fig080} and the diagram in Figure~\ref{Fig081}.
\end{enumerate}
\end{example}

\begin{figure}[!ht]
\centering
\subcaptionbox{\label{Fig081}}{
\begin{tikzpicture}
\fivebox{0};
\draw (1,0) -- (1,-2)
	node[fill=cyan, pos=0.25, shape=circle, inner sep=1.8pt, minimum size=2pt]{}
	node[fill=cyan,  pos=0.5, shape=circle, inner sep=1.8pt, minimum size=2pt]{}
	node[fill=cyan,  pos=0.75, shape=circle, inner sep=1.8pt, minimum size=2pt]{};
\draw (2,0)  arc (-180:0:0.5 and 0.4) ;
\draw (2,-2)  arc (180:0:0.5 and 0.4) ;
\draw (4,0) -- (4,-2);
\draw (5,0) -- (5,-2);
\end{tikzpicture}}
\qquad
\subcaptionbox{\label{Fig080}}{
\begin{tikzpicture}
\fivebox{0};
\draw (1,0)  arc (-180:0:1.5 and 0.8) ;
\draw (2,0)  arc (-180:0:0.5 and 0.4) ;
\draw (1,-2)  arc (180:0:0.5 and 0.4) ;
\draw (4,-2)  arc (180:0:0.5 and 0.4) ;
\draw (5,0) -- (3,-2)
	node[fill=cyan,  pos=0.2, shape=circle, inner sep=1.8pt, minimum size=2pt]{}
	node[fill=cyan,  pos=0.4, shape=circle, inner sep=1.8pt, minimum size=2pt]{}
	node[fill=cyan,  pos=0.6, shape=circle, inner sep=1.8pt, minimum size=2pt]{}
	node[fill=cyan,  pos=0.8, shape=circle, inner sep=1.8pt, minimum size=2pt]{}; 
\draw[fill=cyan, draw=white]{(1.5,-1.6) circle (2.8pt)};
\draw[fill=cyan, draw=white]{(2.5,-0.4) circle (2.8pt)};
\draw[fill=cyan, draw=white]{(4.5,-1.6) circle (2.8pt)};
\end{tikzpicture}}
\caption{Examples of decorated pseudo diagrams.}
\end{figure}
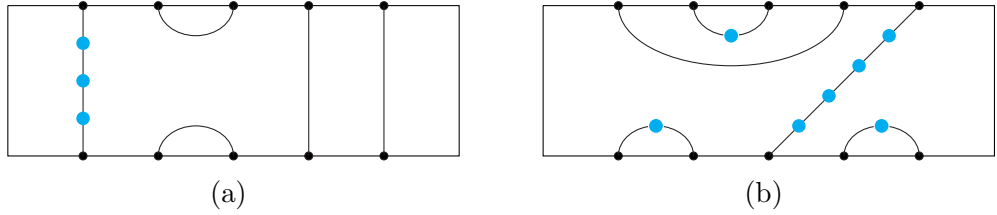

\begin{figure}[!ht]
\centering
\begin{tabular}[c]{l}
\begin{tikzpicture}
\fivebox{0};
\draw (1,0)  arc (-180:0:1.5 and 0.8) ;
\draw (2,0)  arc (-180:0:0.5 and 0.4) ;
\draw (1,-2)  arc (180:0:0.5 and 0.4) ;
\draw (4,-2)  arc (180:0:0.5 and 0.4) ;
\draw (5,0) -- (3,-2)
	node[fill=cyan,  pos=0.2, shape=circle, inner sep=1.8pt, minimum size=2pt]{}
	node[fill=cyan,  pos=0.4, shape=circle, inner sep=1.8pt, minimum size=2pt]{}
	node[fill=cyan,  pos=0.6, shape=circle, inner sep=1.8pt, minimum size=2pt]{}
	node[fill=cyan,  pos=0.8, shape=circle, inner sep=1.8pt, minimum size=2pt]{};
\draw[fill=cyan, draw=white]{(1.5,-1.6) circle (2.8pt)};
\draw[fill=cyan, draw=white]{(2.5,-0.4) circle (2.8pt)};
\draw[fill=cyan, draw=white]{(4.5,-1.6) circle (2.8pt)};
\fivebox{-2};
\draw (1,-2) -- (1,-4)
	node[fill=cyan,  pos=0.25, shape=circle, inner sep=1.8pt, minimum size=2pt]{}
	node[fill=cyan,  pos=0.5, shape=circle, inner sep=1.8pt, minimum size=2pt]{}
	node[fill=cyan,  pos=0.75, shape=circle, inner sep=1.8pt, minimum size=2pt]{};
\draw (2,-2)  arc (-180:0:0.5 and 0.4) ;
\draw (2,-4)  arc (180:0:0.5 and 0.4) ;
\draw (4,-2) -- (4,-4);
\draw (5,-2) -- (5,-4);
\end{tikzpicture}
\end{tabular}
~$=$~
\begin{tabular}[c]{l}
\begin{tikzpicture}
\fivebox{0};
\draw (1,0)  arc (-180:0:1.5 and 0.8) ;
\draw (2,0)  arc (-180:0:0.5 and 0.4) ;
\draw (5,0) -- (1,-2)
	node[fill=cyan,  pos=0.15, shape=circle, inner sep=1.8pt, minimum size=2pt]{}
	node[fill=cyan,  pos=0.25, shape=circle, inner sep=1.8pt, minimum size=2pt]{}
	node[fill=cyan,  pos=0.35, shape=circle, inner sep=1.8pt, minimum size=2pt]{}
	node[fill=cyan,  pos=0.45, shape=circle, inner sep=1.8pt, minimum size=2pt]{}
	node[fill=cyan,  pos=0.55, shape=circle, inner sep=1.8pt, minimum size=2pt]{}
	node[fill=cyan,  pos=0.65, shape=circle, inner sep=1.8pt, minimum size=2pt]{}
	node[fill=cyan,  pos=0.75, shape=circle, inner sep=1.8pt, minimum size=2pt]{}
	node[fill=cyan,  pos=0.85, shape=circle, inner sep=1.8pt, minimum size=2pt]{}; 
\draw (2,-2)  arc (180:0:0.5 and 0.4) ;
\draw (4,-2)  arc (180:0:0.5 and 0.4) ;
\draw[fill=cyan, draw=white]{(2.5,-0.4) circle (2.8pt)};
\draw[fill=cyan, draw=white]{(4.5,-1.6) circle (2.8pt)};
\end{tikzpicture}
\end{tabular}
\caption{The product of two decorated pseudo diagrams}
\label{fig:product}
\end{figure}
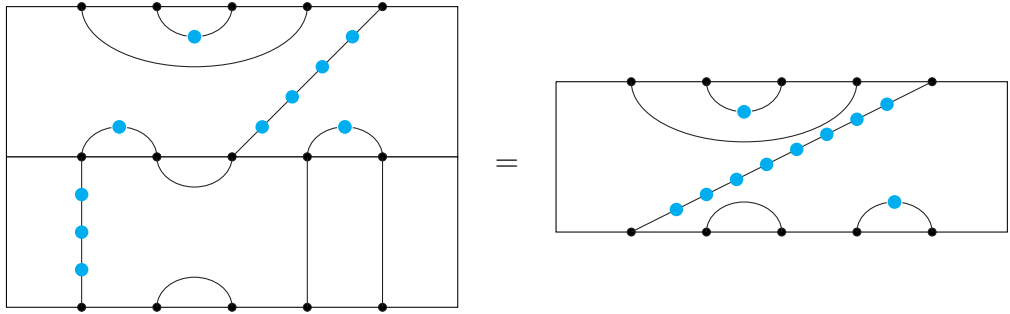

In type $D_n$, we also require the decorations to be ``left exposed" (requirement (D2)), a concept that appears in the context of the Temperley--Lieb algebra of type $B$~\cite{Green1998a}.

\begin{enumerate}[leftmargin=0.5in]
\item[(D2)]  All decorated edges can be simultaneously deformed so as to take decorations to the left wall of the diagram without crossing any other edges.
\end{enumerate}

\begin{remark}
In type $D_n$ we only need the decorations to be left-exposed, but some types also require the decorations to be right-exposed---for example, the diagrammatic representation of type $\C_n$ in~\cite{Ernst2012}.
\end{remark}

\begin{definition}
A \emph{concrete L-decorated pseudo $k$-diagram} is any concrete decorated pseudo $k$-diagram that also satisfies condition (D2).
\end{definition}

\begin{definition}\label{def:T_{k}^{LR}(Omega)}
An \emph{L-decorated pseudo $k$-diagram} is defined to be an equivalence class of equivalent concrete L-decorated pseudo $k$-diagrams.  We denote the set of equivalence classes from $T_{k}(\bcirc)$ where representatives are concrete L-decorated pseudo $k$-diagrams by $T_{k}^L(\bcirc)$. Then define $\P_{k}^L(\bcirc)$ to be the subalgrebra of $\P_{k}(\bcirc)$ with $T_{k}^L(\bcirc)$ as a basis.  
\end{definition}

Note that ``L" stands for ``left" in the definitions above.

\begin{example}
The diagram in Figure~\ref{Fig081} is an L-decorated diagram, however, the diagram in Figure~\ref{Fig080} is not an L-decorated diagram since there are two edges with decorations that cannot be deformed so as to take the decoration to the left wall of the diagram without crossing another edge.
\end{example}

\begin{remark}
We observe that the product of two L-decorated pseudo $k$-diagrams is a L-decorated pseudo $k$-diagram.
\end{remark}

\end{section}


\begin{section}{Diagrammatic relations}


Our immediate goal is to define a quotient of $\P_{k}^{L}(\bcirc)$ having relations that are determined by applying local combinatorial rules to the diagrams. 

\begin{definition}\label{def:big diagram alg}
Let $\widehat{\P}_{k}^{L}(\bcirc)$ be the associative $\Z[\delta]$-algebra equal to the quotient of $\P_{k}^{L}(\bcirc)$ by the relations depicted in Figure~\ref{fig:relations}, where the decorations on the edges represent adjacent decorations of the same block.
\end{definition}

\begin{figure}[!ht]
\centering
$\begin{tabular}[c]{l}
\begin{tikzpicture}
\lp{0};
\end{tikzpicture}
\end{tabular}
~$=$~
$\begin{tabular}[c]{l}
\begin{tikzpicture}
\node at (-1.2,0) {};
\node at (-1,0) {$\delta$};
\node at (-0.5,0) {};
\end{tikzpicture}
\end{tabular}
\quad \quad \quad
\begin{tabular}[c]{l}
\begin{tikzpicture}
\draw (1,0) -- (1,-2)
	node[fill=cyan,  pos=0.35, shape=circle, inner sep=1.8pt, minimum size=2pt]{}
	node[fill=cyan,  pos=0.65, shape=circle, inner sep=1.8pt, minimum size=2pt]{};
\end{tikzpicture}
\end{tabular}
~$=$~
\begin{tabular}[c]{l}
\begin{tikzpicture}
\draw (1,0) -- (1,-2);
\end{tikzpicture}
\end{tabular}
\quad \quad \quad
$\begin{tabular}[c]{l}
\begin{tikzpicture}
\dlp{0}{-1};
\draw (1,0) -- (1,-2)
	node[fill=cyan,  pos=0.5, shape=circle, inner sep=1.8pt, minimum size=2pt]{};
\end{tikzpicture}
\end{tabular}
~$=$~
$\begin{tabular}[c]{l}
\begin{tikzpicture}
\dlp{0}{-1};
\draw (1,0) -- (1,-2);
\end{tikzpicture}
\end{tabular}
\caption{Defining relations of $\widehat{\P}_{k}^{L}(\bcirc)$.}
\label{fig:relations}
\end{figure}

The third relation in Figure~\ref{fig:relations} means that any edge loses its decoration in the presence of a decorated loop. Using the first and third relation, if there is more than one decorated loop, then all loops are replaced with the coefficient $\delta$ except for one decorated loop. The second relation ensures that no edge may carry more than one decoration. Note that all of the relations are local in the sense that a single reduction involves edges bounding the same region of the diagram. 

\begin{remark}\label{diagbasis}
The local diagrammatic relations for $\P^{L}_{k}(\bcirc)$ make it clear that L-decorated diagrams with no undecorated loops having either exactly one decorated loop and no other decorations or no decorated loops with edges having at most one decoration form a basis for $\P^{L}_{k}(\bcirc)$.
\end{remark}

\begin{example}
Figure~\ref{fig:apply} depicts the relations from Figure~\ref{fig:relations} applied to the diagram from $\P^{L}_{5}(\bcirc)$ in Figure~\ref{Fig081}.
\end{example}

\begin{figure}[!h]
\centering
\begin{tikzpicture}[scale=1]
\fivebox{0};
\draw (1,0) -- (1,-2)
	node[fill=cyan,  pos=0.5, shape=circle, inner sep=1.8pt, minimum size=2pt]{};
\draw (2,0)  arc (-180:0:0.5 and 0.4) ;
\draw (2,-2)  arc (180:0:0.5 and 0.4) ;
\draw (4,0) -- (4,-2);
\draw (5,0) -- (5,-2);
\end{tikzpicture}
\caption{Example of a diagram from $\widehat{\P}_{5}^{L}(\bcirc)$.}
\label{fig:apply}
\end{figure}

\begin{example}
Figure~\ref{Fig101--Fig102} depicts multiplication of three diagrams from $\widehat{\P}_{5}^{L}(\bcirc)$ and Figure~\ref{Fig103--Fig104} shows an example where a decorated loop is present. 
\end{example}

\begin{figure}[!ht]
\centering
\begin{tabular}[c]{l}
\begin{tikzpicture}[scale=1]
\fivebox{2};
\draw (1,2) arc (-180:0:0.5 and 0.4);
\draw[fill=cyan, draw=white]{(1.5,1.6) circle (2.8pt)};
\draw (3,2) -- (1,0)
	node[fill=cyan,  pos=0.5, shape=circle, inner sep=1.8pt, minimum size=2pt]{};
\draw (2,0) arc (180:0:0.5 and 0.4);
\draw (4,0) arc (180:0:0.5 and 0.4);
\draw (4,2) arc (-180:0:0.5 and 0.4);
\fivebox{0};
\draw (1,0)  arc (-180:0:0.5 and 0.4) ;
\draw (3,0)  arc (-180:0:0.5 and 0.4) ;
\draw (1,-2)  arc (180:0:0.5 and 0.4) ;
\draw (4,-2)  arc (180:0:0.5 and 0.4) ;
\draw (5,0) -- (3,-2)
	node[fill=cyan,  pos=0.5, shape=circle, inner sep=1.8pt, minimum size=2pt]{};
\draw[fill=cyan, draw=white]{(1.5,-1.6) circle (2.8pt)};
\draw[fill=cyan, draw=white]{(1.5,-0.4) circle (2.8pt)};
\fivebox{-2};
\draw (1,-2) -- (1,-4)
	node[fill=cyan,  pos=0.5, shape=circle, inner sep=1.8pt, minimum size=2pt]{};
\draw (2,-2) -- (2,-4);
\draw (3,-2) arc (-180:0:0.5 and 0.4);
\draw (3,-4) arc (180:0:0.5 and 0.4);
\draw (5,-2) -- (5,-4);
\end{tikzpicture}
\end{tabular}
~$=$~
\begin{tabular}[c]{l}
\begin{tikzpicture}[scale=1]
\fivebox{0};
\draw (1,0) arc (-180:0:0.5 and 0.4);
\draw (4,0) arc (-180:0:0.5 and 0.4);
\draw (3,0) -- (5,-2)
	node[fill=cyan,  pos=0.5, shape=circle, inner sep=1.8pt, minimum size=2pt]{};
\draw (1,-2)  arc (180:0:0.5 and 0.4) ;
\draw (3,-2)  arc (180:0:0.5 and 0.4) ;
\draw[fill=cyan, draw=white]{(1.5,-0.4) circle (2.8pt)};
\end{tikzpicture}
\end{tabular}
\caption{Example of multiplication in $\widehat{\P}_{5}^{L}(\bcirc)$.}\label{Fig101--Fig102}
\end{figure} 

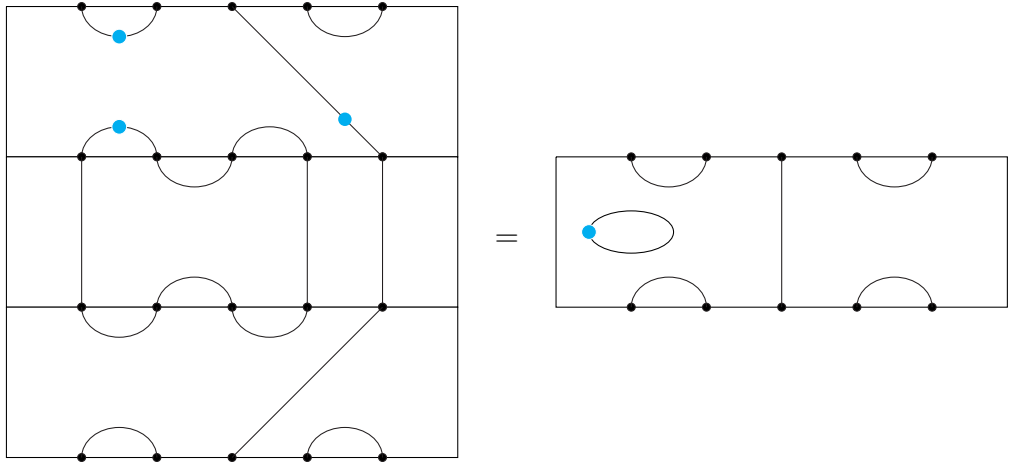
\begin{figure}[!ht]
\centering
\begin{tabular}[c]{l}
\begin{tikzpicture}[scale=1]
\fivebox{2};
\draw (1,0)  arc (180:0:0.5 and 0.4) ;
\draw[fill=cyan, draw=white]{(1.5,0.4) circle (2.8pt)};
\draw (1,2)  arc (-180:0:0.5 and 0.4) ;
\draw[fill=cyan, draw=white]{(1.5,1.6) circle (2.8pt)};
\draw (3,0)  arc (180:0:0.5 and 0.4) ;
\draw (4,2) arc (-180:0:0.5 and 0.4);
\draw (5,0) -- (3,2)
	node[fill=cyan,  pos=0.25, shape=circle, inner sep=1.8pt, minimum size=2pt]{};
\fivebox{0};
\draw (1,-2) -- (1,0);
\draw (2,0)  arc (-180:0:0.5 and 0.4) ;
\draw (2,-2) arc (180:0:0.5 and 0.4);
\draw (4,0) -- (4,-2);
\draw (5,0) -- (5,-2);
\fivebox{-2};
\draw (1,-2) arc (-180:0:0.5 and 0.4);
\draw (3,-2) arc (-180:0:0.5 and 0.4);
\draw (5,-2) -- (3,-4);
\draw (1,-4) arc (180:0:0.5 and 0.4);
\draw (4,-4) arc (180:0:0.5 and 0.4);
\end{tikzpicture}
\end{tabular}
~$=$~
\begin{tabular}[c]{l}
\begin{tikzpicture}[scale=1]
\fivebox{0};
\draw (1,0) arc (-180:0:0.5 and 0.4);
\draw (4,0) arc (-180:0:0.5 and 0.4);
\draw (3,0) -- (3,-2);
\draw (1,-2)  arc (180:0:0.5 and 0.4) ;
\draw (4,-2)  arc (180:0:0.5 and 0.4) ;
\dlp{1}{-1};
\end{tikzpicture}
\end{tabular}
\caption{Example of multiplication in $\widehat{\P}_{5}^{L}(\bcirc)$ resulting in a decorated loop.}
\label{Fig103--Fig104}
\end{figure}

\end{section}


\begin{section}{Simple diagrams}\label{sec:simple}


In this section, we define the diagram algebra $\DTL(D_n)$ as a certain subalgebra of $\widehat{\P}_{n+1}^{L}(\bcirc)$ that turns out to be a diagrammatic representation of $\TL(D_n)$. 

Define the \emph{simple diagrams} $d_{\overline{1}}, d_{1}, d_{2}, \dots, d_{n-1}$ as in Figure~\ref{Fig107--Fig109}.  Note that the simple diagrams are elements of the basis for $\widehat{\P}_{n+1}^{L}(\bcirc)$ (see Definition~\ref{def:big diagram alg} and Remark~\ref{diagbasis}).

\begin{figure}[h]
\begin{align*}
d_{\overline{1}} &=
\begin{tabular}[c]{l}
\begin{tikzpicture}
\sixbox{0};
\draw (1,0) arc (-180:0:0.5 and 0.4);
\draw (1,-2) arc (180:0:0.5 and 0.4);
\draw (3,0) -- (3,-2);
\draw (4,0) -- (4,-2);
\draw (5,0) -- (5,-2);
\draw (6,0) -- (6,-2);
\draw[fill=cyan, draw=white]{(1.5,-0.4) circle (2.8pt)};
\draw[fill=cyan, draw=white]{(1.5,-1.6) circle (2.8pt)};
\node at (5.5,-1) {$\cdots$};
\end{tikzpicture}
\end{tabular}
\\
d_{1} &=
\begin{tabular}[c]{l}
\begin{tikzpicture}
\sixbox{0};
\draw (1,0) arc (-180:0:0.5 and 0.4);
\draw (1,-2) arc (180:0:0.5 and 0.4);
\draw (3,0) -- (3,-2);
\draw (4,0) -- (4,-2);
\draw (5,0) -- (5,-2);
\draw (6,0) -- (6,-2);
\node at (5.5,-1) {$\cdots$};
\end{tikzpicture}
\end{tabular}
\\
d_{2} &= 
\begin{tabular}[c]{l}
\begin{tikzpicture}
\sixbox{0};
\draw (2,0) arc (-180:0:0.5 and 0.4);
\draw (2,-2) arc (180:0:0.5 and 0.4);
\draw (1,0) -- (1,-2);
\draw (4,0) -- (4,-2);
\draw (5,0) -- (5,-2);
\draw (6,0) -- (6,-2);
\node at (5.5,-1) {$\cdots$};
\end{tikzpicture}
\end{tabular}
\\
\quad &~ \vdots \\
d_{n-1} &=
\begin{tabular}[c]{l}
\begin{tikzpicture}
\sixbox{0};
\draw (5,0) arc (-180:0:0.5 and 0.4);
\draw (5,-2) arc (180:0:0.5 and 0.4);
\draw (1,0) -- (1,-2);
\draw (2,0) -- (2,-2);
\draw (3,0) -- (3,-2);
\draw (4,0) -- (4,-2);
\node at (3.5,-1) {$\cdots$};
\end{tikzpicture}
\end{tabular}
\end{align*}
\caption{The simple diagrams of $\widehat{\P}_{n+1}^{L}(\bcirc)$.}
\label{Fig107--Fig109}
\end{figure}

\begin{remark}
It is well known that $d_1,\ldots ,d_{n-1}$ generate the basis diagrams of $\DTL(A_{n-1})$. Moreover, if $s_{x_1} \cdots s_{x_k}$ is a reduced expression for $w\in FC(A_{n-1})$, the isomorphism of Theorem~\ref{kauff} maps the monomial basis element $b_w$ to the diagram $d_w:=d_{x_1}\cdots d_{x_k}$.
\end{remark}

Finally, we are ready to define the diagram algebra that we are ultimately interested in.  

\begin{definition}\label{def:D_n}
Let $\DTL(D_{n})$ be the $\Z[\delta]$-subalgebra of $\widehat{\P}_{n+1}^{L}(\bcirc)$ generated as a unital algebra by $d_{\overline{1}}, d_{1}, d_{2}, \dots, d_{n-1}$ with multiplication inherited from $\widehat{\P}_{n+1}^{L}(\bcirc)$.
\end{definition}

Note that $\DTL(D_{n})\lneq \widehat{\P}_{n+1}^{L}(\bcirc)$ since the diagram in Figure~\ref{fig:apply}, for example, is in $\widehat{\P}_{n+1}^{L}(\bcirc)$ but not in $\DTL(D_{n})$.

\begin{proposition}\label{rem:D relations hold}
Each of the following relations is satisfied for $\DTL(D_{n})$.
\begin{enumerate}
\item $d_{i}^{2}=\delta d_{i}$ for all $i$;
\item $d_{i}d_{j}=d_{j}d_{i}$ if $s_i$ and $s_j$ are not connected in the Coxeter graph;
\item $d_{i}d_{j}d_{i}=d_{i}$ if $s_i$ and $s_j$ are connected in the Coxeter graph.
\end{enumerate}
\end{proposition}

\begin{proof}
We will first consider only the diagrams without decorations, namely, $d_1,\ldots ,d_n$. Then we will consider special cases involving decorations with $d_{\overline{1}}$.
\begin{enumerate}
\item We see that for $i\neq \overline{1}$

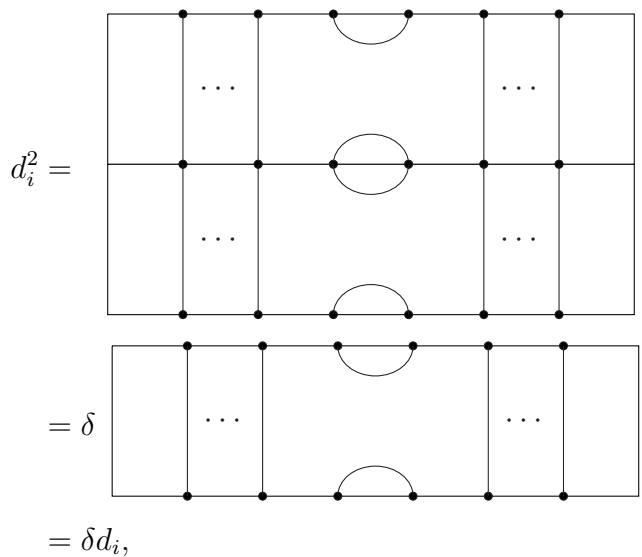
\begin{align*}
d_{i}^2 &= ~
\begin{tabular}[c]{l}
\begin{tikzpicture}
\sixbox{0};
\draw (3,0) arc (-180:0:0.5 and 0.4);
\draw (3,-2) arc (180:0:0.5 and 0.4);
\draw (1,0) -- (1,-2);
\draw (2,0) -- (2,-2);
\draw (5,0) -- (5,-2);
\draw (6,0) -- (6,-2);
\node at (1.5,-1) {$\cdots$};
\node at (5.5,-1) {$\cdots$};
\sixbox{-2};
\draw (3,-4) arc (180:0:0.5 and 0.4);
\draw (3,-2) arc (-180:0:0.5 and 0.4);
\draw (1,-4) -- (1,-2);
\draw (2,-4) -- (2,-2);
\draw (5,-4) -- (5,-2);
\draw (6,-4) -- (6,-2);
\node at (1.5,-3) {$\cdots$};
\node at (5.5,-3) {$\cdots$};
\end{tikzpicture}
\end{tabular}\\
&= \delta 
\begin{tabular}[c]{l}
\begin{tikzpicture}
\sixbox{0};
\draw (3,0) arc (-180:0:0.5 and 0.4);
\draw (3,-2) arc (180:0:0.5 and 0.4);
\draw (1,0) -- (1,-2);
\draw (2,0) -- (2,-2);
\draw (5,0) -- (5,-2);
\draw (6,0) -- (6,-2);
\node at (1.5,-1) {$\cdots$};
\node at (5.5,-1) {$\cdots$};
\end{tikzpicture}
\end{tabular}\\
&= \delta d_{i}, 
\end{align*}
since the loop may be replaced with the coefficient $\delta$. In the case of $d_{\overline{1}}^2$, we see that

\begin{align*}
d_{\overline{1}}^2 &= ~
\begin{tabular}[c]{l}
\begin{tikzpicture}
\sixbox{0};
\draw (1,0) arc (-180:0:0.5 and 0.4);
\draw (1,-2) arc (180:0:0.5 and 0.4);
\draw (3,0) -- (3,-2);
\draw (4,0) -- (4,-2);
\draw (5,0) -- (5,-2);
\draw (6,0) -- (6,-2);
\node at (5.5,-1) {$\cdots$};
\draw[fill=cyan, draw=white]{(1.5,-0.4) circle (2.8pt)};
\draw[fill=cyan, draw=white]{(1.5,-1.6) circle (2.8pt)};
\sixbox{-2};
\draw (1,-4) arc (180:0:0.5 and 0.4);
\draw (1,-2) arc (-180:0:0.5 and 0.4);
\draw (3,-4) -- (3,-2);
\draw (4,-4) -- (4,-2);
\draw (5,-4) -- (5,-2);
\draw (6,-4) -- (6,-2);
\node at (5.5,-3) {$\cdots$};
\draw[fill=cyan, draw=white]{(1.5,-2.4) circle (2.8pt)};
\draw[fill=cyan, draw=white]{(1.5,-3.6) circle (2.8pt)};
\end{tikzpicture}
\end{tabular}\\
&= \delta 
\begin{tabular}[c]{l}
\begin{tikzpicture}
\sixbox{0};
\draw (1,0) arc (-180:0:0.5 and 0.4);
\draw (1,-2) arc (180:0:0.5 and 0.4);
\draw (3,0) -- (3,-2);
\draw (4,0) -- (4,-2);
\draw (5,0) -- (5,-2);
\draw (6,0) -- (6,-2);
\node at (5.5,-1) {$\cdots$};
\draw[fill=cyan, draw=white]{(1.5,-0.4) circle (2.8pt)};
\draw[fill=cyan, draw=white]{(1.5,-1.6) circle (2.8pt)};
\end{tikzpicture}
\end{tabular}\\
&= \delta d_{\overline{1}}, 
\end{align*}
since two decorations on the same edge cancel each other, leaving an undecorated loop that is replaced with $\delta$.

\item Without loss of generality, assume $i<j$ and $i,j\in\{1,\ldots, n-1\}$. We see that

\begin{align*}
d_{i}d_{j} &= 
\begin{tabular}[c]{l}
\begin{tikzpicture}
\longbox{0};
\draw (3,0) arc (-180:0:0.5 and 0.4);
\draw (3,-2) arc (180:0:0.5 and 0.4);
\draw (1,0) -- (1,-2);
\draw (2,0) -- (2,-2);
\draw (7,0) -- (7,-2);
\draw (8,0) -- (8,-2);
\draw (5,0) -- (5,-2);
\draw (6,0) -- (6,-2);
\draw (10,0) -- (10,-2);
\draw (9,0) -- (9,-2);
\node at (1.5,-1) {$\cdots$};
\node at (5.5,-1) {$\cdots$};
\node at (9.5,-1) {$\cdots$};
\longbox{-2};
\draw (7,-4) arc (180:0:0.5 and 0.4);
\draw (7,-2) arc (-180:0:0.5 and 0.4);
\draw (1,-4) -- (1,-2);
\draw (2,-4) -- (2,-2);
\draw (3,-4) -- (3,-2);
\draw (4,-4) -- (4,-2);
\draw (5,-4) -- (5,-2);
\draw (6,-4) -- (6,-2);
\draw (10,-4) -- (10,-2);
\draw (9,-4) -- (9,-2);
\node at (1.5,-3) {$\cdots$};
\node at (5.5,-3) {$\cdots$};
\node at (9.5,-3) {$\cdots$};
\end{tikzpicture}
\end{tabular}\\
&=
\begin{tabular}[c]{l}
\begin{tikzpicture}
\longbox{0};
\draw (7,0) arc (-180:0:0.5 and 0.4);
\draw (7,-2) arc (180:0:0.5 and 0.4);
\draw (3,-2) arc (180:0:0.5 and 0.4);
\draw (3,0) arc (-180:0:0.5 and 0.4);
\draw (1,0) -- (1,-2);
\draw (2,0) -- (2,-2);
\draw (5,0) -- (5,-2);
\draw (6,0) -- (6,-2);
\draw (10,0) -- (10,-2);
\draw (9,0) -- (9,-2);
\node at (1.5,-1) {$\cdots$};
\node at (5.5,-1) {$\cdots$};
\node at (9.5,-1) {$\cdots$};
\end{tikzpicture}
\end{tabular}\\
&= 
\begin{tabular}[c]{l}
\begin{tikzpicture}
\longbox{0};
\draw (7,0) arc (-180:0:0.5 and 0.4);
\draw (7,-2) arc (180:0:0.5 and 0.4);
\draw (1,0) -- (1,-2);
\draw (2,0) -- (2,-2);
\draw (3,0) -- (3,-2);
\draw (4,0) -- (4,-2);
\draw (5,0) -- (5,-2);
\draw (6,0) -- (6,-2);
\draw (10,0) -- (10,-2);
\draw (9,0) -- (9,-2);
\node at (1.5,-1) {$\cdots$};
\node at (5.5,-1) {$\cdots$};
\node at (9.5,-1) {$\cdots$};
\longbox{-2};
\draw (3,-4) arc (180:0:0.5 and 0.4);
\draw (3,-2) arc (-180:0:0.5 and 0.4);
\draw (1,-4) -- (1,-2);
\draw (2,-4) -- (2,-2);
\draw (5,-4) -- (5,-2);
\draw (6,-4) -- (6,-2);
\draw (7,-4) -- (7,-2);
\draw (8,-4) -- (8,-2);
\draw (10,-4) -- (10,-2);
\draw (9,-4) -- (9,-2);
\node at (1.5,-3) {$\cdots$};
\node at (5.5,-3) {$\cdots$};
\node at (9.5,-3) {$\cdots$};
\end{tikzpicture}
\end{tabular}\\
&= ~d_{j}d_{i}. 
\end{align*}

In the case of $d_{\overline{1}}d_j$ when $j>2$, we see that

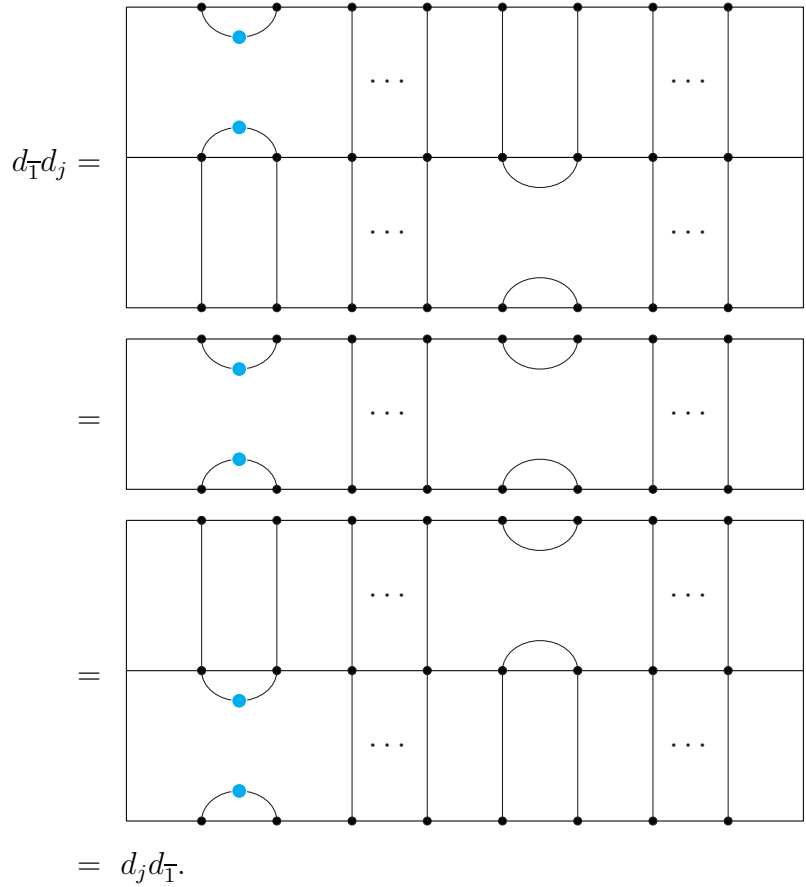
\begin{align*}
d_{\overline{1}}d_{j} &= 
\begin{tabular}[c]{l}
\begin{tikzpicture}
\ninebox{0};
\draw (1,0) arc (-180:0:0.5 and 0.4);
\draw (1,-2) arc (180:0:0.5 and 0.4);
\draw (7,0) -- (7,-2);
\draw (3,0) -- (3,-2);
\draw (4,0) -- (4,-2);
\draw (5,0) -- (5,-2);
\draw (6,0) -- (6,-2);
\draw (8,0) -- (8,-2);
\node at (3.5,-1) {$\cdots$};
\node at (7.5,-1) {$\cdots$};
\draw[fill=cyan, draw=white]{(1.5,-0.4) circle (2.8pt)};
\draw[fill=cyan, draw=white]{(1.5,-1.6) circle (2.8pt)};
\ninebox{-2};
\draw (5,-4) arc (180:0:0.5 and 0.4);
\draw (5,-2) arc (-180:0:0.5 and 0.4);
\draw (1,-4) -- (1,-2);
\draw (2,-4) -- (2,-2);
\draw (3,-4) -- (3,-2);
\draw (4,-4) -- (4,-2);
\draw (7,-4) -- (7,-2);
\draw (8,-4) -- (8,-2);
\node at (3.5,-3) {$\cdots$};
\node at (7.5,-3) {$\cdots$};
\end{tikzpicture}
\end{tabular}\\
&=
\begin{tabular}[c]{l}
\begin{tikzpicture}
\ninebox{0};
\draw (1,0) arc (-180:0:0.5 and 0.4);
\draw (1,-2) arc (180:0:0.5 and 0.4);
\draw (5,0) arc (-180:0:0.5 and 0.4);
\draw (5,-2) arc (180:0:0.5 and 0.4);
\draw (7,0) -- (7,-2);
\draw (3,0) -- (3,-2);
\draw (4,0) -- (4,-2);
\draw (8,0) -- (8,-2);
\node at (3.5,-1) {$\cdots$};
\node at (7.5,-1) {$\cdots$};
\draw[fill=cyan, draw=white]{(1.5,-0.4) circle (2.8pt)};
\draw[fill=cyan, draw=white]{(1.5,-1.6) circle (2.8pt)};
\end{tikzpicture}
\end{tabular}\\
&= 
\begin{tabular}[c]{l}
\begin{tikzpicture}
\ninebox{0};
\draw (5,0) arc (-180:0:0.5 and 0.4);
\draw (5,-2) arc (180:0:0.5 and 0.4);
\draw (7,0) -- (7,-2);
\draw (3,0) -- (3,-2);
\draw (4,0) -- (4,-2);
\draw (1,0) -- (1,-2);
\draw (2,0) -- (2,-2);
\draw (8,0) -- (8,-2);
\node at (3.5,-1) {$\cdots$};
\node at (7.5,-1) {$\cdots$};
\ninebox{-2};
\draw (1,-4) arc (180:0:0.5 and 0.4);
\draw (1,-2) arc (-180:0:0.5 and 0.4);
\draw (5,-4) -- (5,-2);
\draw (6,-4) -- (6,-2);
\draw (3,-4) -- (3,-2);
\draw (4,-4) -- (4,-2);
\draw (7,-4) -- (7,-2);
\draw (8,-4) -- (8,-2);
\node at (3.5,-3) {$\cdots$};
\node at (7.5,-3) {$\cdots$};
\draw[fill=cyan, draw=white]{(1.5,-2.4) circle (2.8pt)};
\draw[fill=cyan, draw=white]{(1.5,-3.6) circle (2.8pt)};
\end{tikzpicture}
\end{tabular}\\
&= ~d_{j}d_{\overline{1}}. 
\end{align*}

Since $s_{\overline{1}}$ is not connected to $s_1$, we will also consider the case of $d_{\overline{1}}d_1$. We see that

\begin{align*}
d_{\overline{1}}d_1 &=
\begin{tabular}[c]{l}
\begin{tikzpicture}
\sixbox{0};
\draw (1,0) arc (-180:0:0.5 and 0.4);
\draw (1,-2) arc (180:0:0.5 and 0.4);
\draw (3,0) -- (3,-2);
\draw (4,0) -- (4,-2);
\draw (5,0) -- (5,-2);
\draw (6,0) -- (6,-2);
\node at (5.5,-1) {$\cdots$};
\draw[fill=cyan, draw=white]{(1.5,-0.4) circle (2.8pt)};
\draw[fill=cyan, draw=white]{(1.5,-1.6) circle (2.8pt)};
\sixbox{-2};
\draw (1,-4) arc (180:0:0.5 and 0.4);
\draw (1,-2) arc (-180:0:0.5 and 0.4);
\draw (3,-4) -- (3,-2);
\draw (4,-4) -- (4,-2);
\draw (5,-4) -- (5,-2);
\draw (6,-4) -- (6,-2);
\node at (5.5,-3) {$\cdots$};
\end{tikzpicture}
\end{tabular}\\
&= 
\begin{tabular}[c]{l}
\begin{tikzpicture}
\sixbox{0};
\draw (1,0) arc (-180:0:0.5 and 0.4);
\draw (1,-2) arc (180:0:0.5 and 0.4);
\draw (3,0) -- (3,-2);
\draw (4,0) -- (4,-2);
\draw (5,0) -- (5,-2);
\draw (6,0) -- (6,-2);
\node at (5.5,-1) {$\cdots$};
\dlp{1.5}{-1};
\end{tikzpicture}
\end{tabular}\\
&=
\begin{tabular}[c]{l}
\begin{tikzpicture}
\sixbox{0};
\draw (1,0) arc (-180:0:0.5 and 0.4);
\draw (1,-2) arc (180:0:0.5 and 0.4);
\draw (3,0) -- (3,-2);
\draw (4,0) -- (4,-2);
\draw (5,0) -- (5,-2);
\draw (6,0) -- (6,-2);
\node at (5.5,-1) {$\cdots$};
\sixbox{-2};
\draw (1,-4) arc (180:0:0.5 and 0.4);
\draw (1,-2) arc (-180:0:0.5 and 0.4);
\draw (3,-4) -- (3,-2);
\draw (4,-4) -- (4,-2);
\draw (5,-4) -- (5,-2);
\draw (6,-4) -- (6,-2);
\node at (5.5,-3) {$\cdots$};
\draw[fill=cyan, draw=white]{(1.5,-2.4) circle (2.8pt)};
\draw[fill=cyan, draw=white]{(1.5,-3.6) circle (2.8pt)};
\end{tikzpicture}
\end{tabular}\\
&= d_1d_{\overline{1}}, 
\end{align*}

since any edge loses its decorations in the presence of a decorated loop.

\item Without loss of generality, assume $j=i+1$ and $i,j\in\{1,\ldots, n-1\}$. We see that

\begin{align*}
d_{i}d_{j}d_{i} &= 
\begin{tabular}[c]{l}
\begin{tikzpicture}
\eightbox{0};
\draw (3,0) arc (-180:0:0.5 and 0.4);
\draw (3,-2) arc (180:0:0.5 and 0.4);
\draw (1,0) -- (1,-2);
\draw (2,0) -- (2,-2);
\draw (7,0) -- (7,-2);
\draw (5,0) -- (5,-2);
\draw (6,0) -- (6,-2);
\node at (1.5,-1) {$\cdots$};
\node at (6.5,-1) {$\cdots$};
\eightbox{-2};
\draw (4,-4) arc (180:0:0.5 and 0.4);
\draw (4,-2) arc (-180:0:0.5 and 0.4);
\draw (1,-4) -- (1,-2);
\draw (2,-4) -- (2,-2);
\draw (3,-4) -- (3,-2);
\draw (6,-4) -- (6,-2);
\draw (7,-4) -- (7,-2);
\node at (1.5,-3) {$\cdots$};
\node at (6.5,-3) {$\cdots$};
\eightbox{-4};
\draw (3,-4) arc (-180:0:0.5 and 0.4);
\draw (3,-6) arc (180:0:0.5 and 0.4);
\draw (1,-4) -- (1,-6);
\draw (2,-4) -- (2,-6);
\draw (7,-4) -- (7,-6);
\draw (5,-4) -- (5,-6);
\draw (6,-4) -- (6,-6);
\node at (1.5,-5) {$\cdots$};
\node at (6.5,-5) {$\cdots$};
\end{tikzpicture}
\end{tabular}\\
&=
\begin{tabular}[c]{l}
\begin{tikzpicture}
\eightbox{0};
\draw (3,-2) arc (180:0:0.5 and 0.4);
\draw (3,0) arc (-180:0:0.5 and 0.4);
\draw (1,0) -- (1,-2);
\draw (2,0) -- (2,-2);
\draw (5,0) -- (5,-2);
\draw (6,0) -- (6,-2);
\draw (7,0) -- (7,-2);
\node at (1.5,-1) {$\cdots$};
\node at (6.5,-1) {$\cdots$};
\end{tikzpicture}
\end{tabular}\\
&= ~d_i
\end{align*}
Since $s_{\overline{1}}$ is only connected to $s_2$ in the Coxeter graph, we will consider two more cases. In the case $d_{\overline{1}}d_2d_{\overline{1}}$, we see that
\begin{align*}
d_{\overline{1}}d_2d_{\overline{1}} &=
\begin{tabular}[c]{l}
\begin{tikzpicture}[scale=0.93]
\sixbox{0};
\draw (1,0) arc (-180:0:0.5 and 0.4);
\draw (1,-2) arc (180:0:0.5 and 0.4);
\draw (3,0) -- (3,-2);
\draw (4,0) -- (4,-2);
\draw (5,0) -- (5,-2);
\draw (6,0) -- (6,-2);
\node at (5.5,-1) {$\cdots$};
\draw[fill=cyan, draw=white]{(1.5,-0.4) circle (2.8pt)};
\draw[fill=cyan, draw=white]{(1.5,-1.6) circle (2.8pt)};
\sixbox{-2};
\draw (2,-4) arc (180:0:0.5 and 0.4);
\draw (2,-2) arc (-180:0:0.5 and 0.4);
\draw (1,-4) -- (1,-2);
\draw (4,-4) -- (4,-2);
\draw (5,-4) -- (5,-2);
\draw (6,-4) -- (6,-2);
\node at (5.5,-3) {$\cdots$};
\sixbox{-4};
\draw (1,-4) arc (-180:0:0.5 and 0.4);
\draw (1,-6) arc (180:0:0.5 and 0.4);
\draw (3,-4) -- (3,-6);
\draw (4,-4) -- (4,-6);
\draw (5,-4) -- (5,-6);
\draw (6,-4) -- (6,-6);
\node at (5.5,-5) {$\cdots$};
\draw[fill=cyan, draw=white]{(1.5,-4.4) circle (2.8pt)};
\draw[fill=cyan, draw=white]{(1.5,-5.6) circle (2.8pt)};
\end{tikzpicture}
\end{tabular}\\
&= 
\begin{tabular}[c]{l}
\begin{tikzpicture}[scale=0.93]
\sixbox{0};
\draw (1,0) arc (-180:0:0.5 and 0.4);
\draw (1,-2) arc (180:0:0.5 and 0.4);
\draw (3,0) -- (3,-2);
\draw (4,0) -- (4,-2);
\draw (5,0) -- (5,-2);
\draw (6,0) -- (6,-2);
\node at (5.5,-1) {$\cdots$};
\draw[fill=cyan, draw=white]{(1.5,-0.4) circle (2.8pt)};
\draw[fill=cyan, draw=white]{(1.5,-1.6) circle (2.8pt)};
\end{tikzpicture}
\end{tabular}\\
&= d_{\overline{1}}, 
\end{align*}
since two decorations on the same edge cancel. Then, in the case of $d_2d_{\overline{1}}d_2$, we see that
\begin{align*}
d_2d_{\overline{1}}d_2 &=
\begin{tabular}[c]{l}
\begin{tikzpicture}[scale=0.93]
\sixbox{0};
\draw (2,0) arc (-180:0:0.5 and 0.4);
\draw (2,-2) arc (180:0:0.5 and 0.4);
\draw (1,0) -- (1,-2);
\draw (4,0) -- (4,-2);
\draw (5,0) -- (5,-2);
\draw (6,0) -- (6,-2);
\node at (5.5,-1) {$\cdots$};
\sixbox{-2};
\draw (1,-4) arc (180:0:0.5 and 0.4);
\draw (1,-2) arc (-180:0:0.5 and 0.4);
\draw (3,-4) -- (3,-2);
\draw (4,-4) -- (4,-2);
\draw (5,-4) -- (5,-2);
\draw (6,-4) -- (6,-2);
\node at (5.5,-3) {$\cdots$};
\draw[fill=cyan, draw=white]{(1.5,-2.4) circle (2.8pt)};
\draw[fill=cyan, draw=white]{(1.5,-3.6) circle (2.8pt)};
\sixbox{-4};
\draw (2,-4) arc (-180:0:0.5 and 0.4);
\draw (2,-6) arc (180:0:0.5 and 0.4);
\draw (1,-4) -- (1,-6);
\draw (4,-4) -- (4,-6);
\draw (5,-4) -- (5,-6);
\draw (6,-4) -- (6,-6);
\node at (5.5,-5) {$\cdots$};
\end{tikzpicture}
\end{tabular}\\
&= 
\begin{tabular}[c]{l}
\begin{tikzpicture}[scale=0.93]
\sixbox{0};
\draw (2,0) arc (-180:0:0.5 and 0.4);
\draw (2,-2) arc (180:0:0.5 and 0.4);
\draw (1,0) -- (1,-2);
\draw (4,0) -- (4,-2);
\draw (5,0) -- (5,-2);
\draw (6,0) -- (6,-2);
\node at (5.5,-1) {$\cdots$};
\end{tikzpicture}
\end{tabular}\\
&= d_{2}, 
\end{align*}
since two decorations on the same edge cancel.
\end{enumerate}
\end{proof}

The next proposition follows quickly from Proposition~\ref{rem:D relations hold} since $\DTL(D_{n})$ satisfies the relations given in Theorem~\ref{def:TL(D)}.

\begin{proposition}\label{prop:surjective homomorphism}
The map $\theta: \TL(D_{n}) \to \DTL(D_{n})$ determined by $\theta(b_{i})=d_{i}$ is a well-defined surjective $\Z[\delta]$-algebra homomorphism. \qed
\end{proposition}

In order to show that $\theta$ is an isomorphism, we need to first define $D$-admissible diagrams.

\end{section}


\begin{section}{$D$-admissible diagrams of type I and type II}\label{sec:admissible}


It turns out that the set of $D$-admissible diagrams form a basis for $\DTL(D_n)$.  Our definition of $D$-admissible comes from Theorem 4.2 in~\cite{Green1998}.

\begin{definition}\label{def:admissible}
Let $d$ be an irreducible (i.e., no relations to apply) L-decorated diagram.  Then we say that $d$ is \emph{$D$-admissible of type I} or \emph{$D$-admissible of type II} depending on which of the two mutually exclusive conditions below it satisfies. 

\begin{enumerate}[leftmargin=0.5in]
\item[(I)] \label{type1} The diagram contains one loop which is decorated, and no other loops or decorations.
\item[(II)] \label{type2} The diagram contains no loops and the total number of decorations is even.
\end{enumerate}
\end{definition}

\begin{figure}[!ht]
\centering
\begin{subfigure}[b]{0.4\textwidth}
\begin{tikzpicture}
\fivebox{0};
\draw (1,0) arc (-180:0:0.5 and 0.4);
\draw (4,0) arc (-180:0:0.5 and 0.4);
\draw (3,0) -- (3,-2);
\draw (1,-2)  arc (180:0:0.5 and 0.4) ;
\draw (4,-2)  arc (180:0:0.5 and 0.4) ;
\dlp{1}{-1};
\end{tikzpicture}
\caption{Type I}
\label{fig:type1diag}
\end{subfigure}
\quad
\begin{subfigure}[b]{0.4\textwidth}
\begin{tikzpicture}
\fivebox{0};
\draw (1,0) arc (-180:0:0.5 and 0.4);
\draw (4,0) arc (-180:0:0.5 and 0.4);
\draw (3,0) -- (5,-2)
	node[fill=cyan,  pos=0.5, shape=circle, inner sep=1.8pt, minimum size=2pt]{};
\draw (1,-2)  arc (180:0:0.5 and 0.4) ;
\draw (3,-2)  arc (180:0:0.5 and 0.4) ;
\draw[fill=cyan, draw=white]{(1.5,-1.6) circle (2.8pt)};
\end{tikzpicture}
\caption{Type II}
\label{fig:type2diag}
\end{subfigure}
\caption{$D$-admissible diagrams.}
\label{fig:diag}
\end{figure}
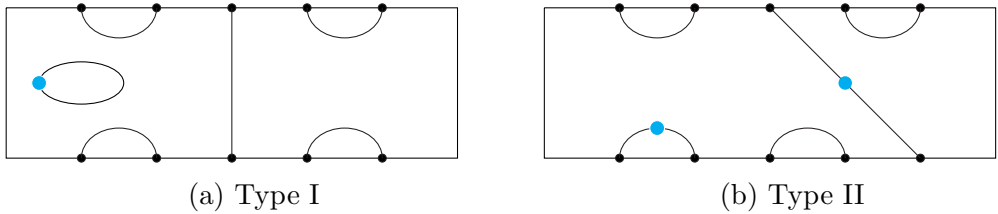

\begin{example}
Figure~\ref{fig:diag} shows an example of a type I diagram and a type II diagram.
\end{example}

\begin{proposition}[Green~\cite{Green1998a}]\label{isomorphic}
In type $D_{n}$, the number of $D$-admissible diagrams of type~I is $C(n)-1$, and the number of type~II is $\frac{1}{2}{2n\choose n}$. Therefore, the total number of $D-$admissible diagrams is $\(\frac{n+3}{2}\)C(n)-1$. \qed
\end{proposition}

\begin{proposition}[Green~\cite{Green1998a}]
The $D$-admissible diagrams form a basis of $\DTL(D_{n})$ and thus $\dim\left(\DTL(D_{n})\right)=\(\frac{n+3}{2}\)C(n)-1$. \qed
\end{proposition}

\begin{theorem}[Green~\cite{Green1998a}]
As $\Z[\delta]$-algebras, the diagram algebra, $\DTL(D_{n})$, is isomorphic to $\TL(D_{n})$ under $\theta$ as defined in Proposition~\ref{prop:surjective homomorphism}. Moreover, the monomial basis for $\TL(D_{n})$ is in natural bijection with the $D$-admissible diagrams. \qed
\end{theorem}
 
\end{section}


\chapter{A cellular quotient}


In this chapter, we construct a diagrammatic representation of a quotient of $\TL(D_n)$ that will be used to show that this particular quotient is cellular.


\section{Pair-free Temperley-Lieb algebra}\label{sec:pairfree}


We will define the \emph{pair-free Temperley--Lieb algebra of type $D_{n}$}, denoted $\PFTL(D_{n})$, to be the quotient of $\TL(D_n)$ with the additional relation $b_1b_{\overline{1}}=0$. 
Since the monomial basis forms a basis for $\TL(D_n)$ and the relation $b_1b_{\overline{1}}=0$ eliminates the monomials indexed by the type I heaps but has no impact on the monomials indexed by the type II heaps, the following proposition holds.

\begin{proposition}
Let $\b_w$ be the image of $b_w$ in the quotient $\PFTL(D_n)$. Then
\[
\{\b_w: H(w)\text{ is of type II}\}
\]
is a basis for $\PFTL(D_n)$. \qed
\end{proposition}

Note that if $w\in\FC(D_n)$ and no reduced expression of $w$ has $s_{\overline{1}}s_1$ as a subword, the heap of $w$ is of type II. In this case, we can safely identify $\b_w$ with $b_w$. We can represent $\PFTL(D_{n})$ in terms of generators and relations in a similar fashion to that of $\TL(D_n)$.

\begin{remark}\label{def:pfTL(D)}
The algebra $\PFTL(D_{n})$ ($n\ge4$)  is the unital $\Z[\delta]$-algebra generated by $\b_{\overline{1}},\b_{1},\b_{2},\ldots,\b_{n-1}$ with defining relations
\begin{enumerate}[leftmargin=0.6in]
\item $\b_{i}^{2}=\delta \b_{i}$ for all $i$, where $\delta$ is an indeterminate;
\item $\b_{i}\b_{j} = \b_{j}\b_{i}$ if $s_i$ and $s_j$ are not connected in the Coxeter graph of type $D_n$;
\item $\b_{i}\b_{j}\b_{i} = \b_{i}$ if $s_i$ and $s_j$ are connected in the Coxeter graph of type $D_n$;
\item $\b_{1}\b_{\overline{1}}=0.$
\end{enumerate}
\end{remark}


\begin{section}{Loop-free diagram algebra}\label{sec:loopfree}


We will now construct a diagram algebra that turns out to be a diagrammatic representation of $\PFTL(D_n)$.

\begin{definition}\label{def:D_nII}
Let $\LFD(D_{n})$ be the $\Z[\delta]$-algebra equal to the quotient of $\DTL(D_n)$ with the additional relation given in Figure~\ref{decloop}.
\end{definition}

\begin{figure}[!ht]
\centering
\begin{align*}
\begin{tabular}[c]{l}
\begin{tikzpicture}
\dlp{0}{-1};
\end{tikzpicture}
\end{tabular}
&=~0
\end{align*}
\caption{Additional defining relation of $\LFD(D_n)$.}
\label{decloop}
\end{figure}

Let $\d_w$ denote the image of $d_w$ in the quotient. Since the relation in Figure~\ref{decloop} has no effect on the type II diagrams, we can safely identify $\d_{w}$ with $d_{w}$ when $d_w$ is of type II and will use the same diagram to represent $\d_w$. Since $d_{\overline{1}},d_1,\ldots,d_n$ generate the type I and type II diagrams and each $d_i$ is a type II diagram, it follows that $\LFD(D_n)$ is generated by $\d_{\overline{1}},\d_1,\d_2,\ldots,\d_{n-1}$.

\begin{proposition}\label{rem:DII relations hold}
Each of the following relations are satisfied for $\LFD(D_{n})$.
\begin{enumerate}[leftmargin=0.6in]
\item $\d_{i}^{2}=\delta \d_{i}$ for all $i$;
\item $\d_{i}\d_{j}=\d_{j}\d_{i}$ if $s_i$ and $s_j$ are not connected in the Coxeter graph of type $D_n$;
\item $\d_{i}\d_{j}\d_{i}=\d_{i}$ if $s_i$ and $s_j$ are connected in the Coxeter graph of type $D_n$;
\item $\d_{1}\d_{\overline{1}}=0.$
\end{enumerate}
\end{proposition}

\begin{proof}
Since the first three relations hold in $\TL(D_n)$, the only relation left to check is $\d_{1}\d_{\overline{1}}=0$. We see that
\begin{align*}
\d_1\d_{\overline{1}} &=
\begin{tabular}[c]{l}
\begin{tikzpicture}[scale=.8]
\sixbox{0};
\draw (1,0) arc (-180:0:0.5 and 0.4);
\draw (1,-2) arc (180:0:0.5 and 0.4);
\draw (3,0) -- (3,-2);
\draw (4,0) -- (4,-2);
\draw (5,0) -- (5,-2);
\draw (6,0) -- (6,-2);
\node at (5.5,-1) {$\cdots$};
\sixbox{-2};
\draw (1,-4) arc (180:0:0.5 and 0.4);
\draw (1,-2) arc (-180:0:0.5 and 0.4);
\draw (3,-4) -- (3,-2);
\draw (4,-4) -- (4,-2);
\draw (5,-4) -- (5,-2);
\draw (6,-4) -- (6,-2);
\node at (5.5,-3) {$\cdots$};
\draw[fill=cyan, draw=white]{(1.5,-2.4) circle (2.8pt)};
\draw[fill=cyan, draw=white]{(1.5,-3.6) circle (2.8pt)};
\end{tikzpicture}
\end{tabular}\\
&= 
\begin{tabular}[c]{l}
\begin{tikzpicture}[scale=.8]
\sixbox{0};
\draw (1,0) arc (-180:0:0.5 and 0.4);
\draw (1,-2) arc (180:0:0.5 and 0.4);
\draw (3,0) -- (3,-2);
\draw (4,0) -- (4,-2);
\draw (5,0) -- (5,-2);
\draw (6,0) -- (6,-2);
\node at (5.5,-1) {$\cdots$};
\dlp{1.5}{-1};
\end{tikzpicture}
\end{tabular}\\
&=0.
\end{align*}
\end{proof}

The next proposition follows quickly from Proposition~\ref{rem:DII relations hold} since $\LFD(D_{n})$ satisfies the relations given in Remark~\ref{def:pfTL(D)}.

\begin{proposition}\label{prop:surjective homo}
The map $\phi: \PFTL(D_{n}) \to \LFD(D_{n})$ determined by $\phi(\b_{i})=\d_{i}$ is a well-defined surjective $\Z[\delta]$-algebra homomorphism. \qed
\end{proposition}

Since the $D$-admissible diagrams form a basis for $\DTL(D_{n})$ and the relation in (\ref{decloop}) eliminates the type I $D$-admissible diagrams but has no impact on the type II $D$-admissible diagrams, the following proposition holds.

\begin{proposition}
The images of the type II $D$-admissible diagrams form a basis for $\LFD(D_n)$. \qed
\end{proposition}

If $d$ is a $D$-admissible diagram, then we say that a non-propagating edge joining $i$ to $i+1$ (respectively, $i'$ to $(i+1)'$) is \emph{simple} if it is  identical to the edge joining $i$ to $i+1$ (respectively $i'$ to $(i+1)'$) in the simple diagram $d_i$. That is, an edge is simple if it joins adjacent vertices in the north face (respectively, south face) and is undecorated, except when one of the vertices is $1$ or $1'$, in which case it may be decorated by only a single decoration~$\bcirc$.

Let $\w=s_{x_1}\cdots s_{x_k}$ be a reduced expression for $w\in\FC(D_n)$. Then $d=d_{x_1}\cdots d_{x_k}$. For each $d_{x_i}$, fix a concrete representation that has straight propagating edges and no unneccessary ``wiggling" of the simple non-propagating edges. Now, consider the concrete diagram that results from stacking the concrete simple diagrams $d_{x_1},\ldots,d_{x_k}$, rescaling vertically to recover the standard $k$-box, but not deforming any of the simple edges or applying any relations among the decorations. We will refer to this concrete diagram as the \emph{concrete simple representation of} $d_{\w}$. 

Since $w$ is fully commutative and vertical equivalence respects commutation, given two different reduced expressions $\w_1$ and $\w_2$ for $w$, the concrete simple representations $d_{\w_1}$ and $d_{\w_2}$ will be vertically equivalent (see Remark~\ref{vertequiv}). We define the vertical equivalence class of concrete simple representations to be the \emph{simple representation of }$d_w$. The simple representation of $d_w$ is designed to replicate the structure of the corresponding heap.

\begin{example}\label{simplerep}
Let $\w=s_{\overline{1}}s_{3}s_{2}s_{1}$ be a reduced expression for $w\in\FC(D_4)$. The concrete simple representation of $d_{w}$ is shown in Figure~\ref{fig:simplerep} where the vertical dashed lines in the diagram indicate that the two non-propagating edges are part of the same generator.
\end{example}

\begin{figure}[!ht]
\centering
\begin{tabular}[c]{l}
\begin{tikzpicture}[scale=1]
\fivebox{2};
\draw (1,2) arc (-180:0:0.5 and 0.4);
\draw[fill=cyan, draw=white]{(1.5,1.6) circle (2.8pt)};
\draw (1,0) arc (180:0:0.5 and 0.4);
\draw[fill=cyan, draw=white]{(1.5,0.4) circle (2.8pt)};
\draw (3,2) -- (3,0);
\draw (4,2) -- (4,0);
\draw (5,2) -- (5,0);
\fivebox{0};
\draw (1,0)  -- (1,-2);
\draw (3,0)  arc (-180:0:0.5 and 0.4) ;
\draw (3,-2)  arc (180:0:0.5 and 0.4) ;
\draw (2,-2)  -- (2,0);
\draw (5,0) -- (5,-2);
\fivebox{-2};
\draw (1,-2)  -- (1,-4);
\draw (2,-2)  arc (-180:0:0.5 and 0.4) ;
\draw (2,-4)  arc (180:0:0.5 and 0.4) ;
\draw (4,-4)  -- (4,-2);
\draw (5,-4) -- (5,-2);
\fivebox{-4};
\draw (3,-6)  -- (3,-4);
\draw (1,-4)  arc (-180:0:0.5 and 0.4) ;
\draw (1,-6)  arc (180:0:0.5 and 0.4) ;
\draw (4,-4)  -- (4,-6);
\draw (5,-4) -- (5,-6);
\end{tikzpicture}
\end{tabular}
~$=$~
\begin{tabular}[c]{l}
\begin{tikzpicture}[scale=1]
\topbox{2};
\draw (1,2) arc (-180:0:0.5 and 0.4);
\draw[fill=cyan, draw=white]{(1.5,1.6) circle (2.8pt)};
\draw (1,0) arc (180:0:0.5 and 0.4);
\draw[fill=cyan, draw=white]{(1.5,0.4) circle (2.8pt)};
\draw (3,2) arc (-180:0:0.5 and 0.4);
\draw (5,2) -- (5,0);
\draw[dashed] (1.5,1.5) -- (1.5,0.5);
\draw[dashed] (3.5,1.5) -- (3.5,0.5);
\middlebox{0};
\draw (1,0)  -- (1,-2);
\draw (3,0)  arc (180:0:0.5 and 0.4) ;
\draw (2,0)  arc(-180:0:0.5 and 0.4);
\draw (2,-2)  arc(180:0:0.5 and 0.4);
\draw (4,0) -- (4,-2);
\draw (5,0) -- (5,-2);
\draw[dashed] (2.5,-0.5) -- (2.5,-1.5);
\bottombox{-2};
\draw (3,-2)  -- (3,-4);
\draw (4,-4)  -- (4,-2) ;
\draw (1,-2)  arc (-180:0:0.5 and 0.4) ;
\draw (1,-4)  arc (180:0:0.5 and 0.4) ;
\draw (5,-2) -- (5,-4);
\draw[dashed] (1.5,-2.5) -- (1.5,-3.5);
\end{tikzpicture}
\end{tabular}
\caption{Example of a simple representation.}\label{fig:simplerep}
\end{figure}
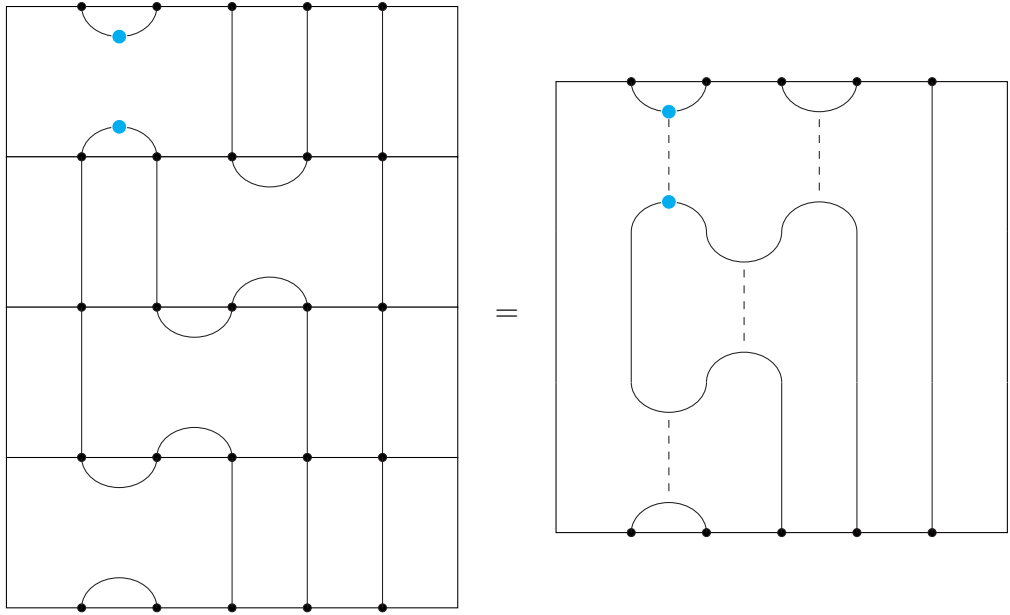 
 
\begin{lemma}\label{configs}
Let $\w$ be a reduced expression for $w\in \FC(D_n)$. Then the simple representation of $d_{\w}$ cannot have the configurations shown in Figure~\ref{notFC}.
\end{lemma}

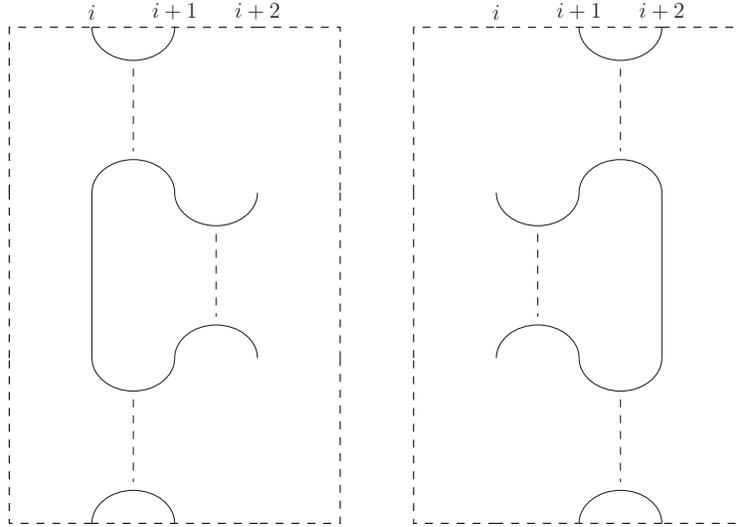
\begin{figure}[!ht]
\begin{center}
\begin{tikzpicture}[scale=1.1]
\tttbox{2};
\draw (1,2) arc (-180:0:0.5 and 0.4);
\draw (1,0) arc (180:0:0.5 and 0.4);
\draw[dashed] (1.5,1.5) -- (1.5,0.5);
\mmmidbox{0};
\draw (2,0)  arc(-180:0:0.5 and 0.4);
\draw (2,-2)  arc(180:0:0.5 and 0.4);
\draw (1,0) -- (1,-2);
\draw[dashed] (2.5,-0.5) -- (2.5,-1.5);
\bbbotbox{-2};
\draw (1,-2)  arc (-180:0:0.5 and 0.4) ;
\draw (1,-4)  arc (180:0:0.5 and 0.4) ;
\draw[dashed] (1.5,-2.5) -- (1.5,-3.5);
\end{tikzpicture}
\quad\quad
\begin{tikzpicture}[scale=1.1]
\tttbox{2};
\draw (2,2) arc (-180:0:0.5 and 0.4);
\draw (2,0) arc (180:0:0.5 and 0.4);
\draw[dashed] (2.5,1.5) -- (2.5,0.5);
\mmmidbox{0};
\draw (1,0)  arc(-180:0:0.5 and 0.4);
\draw (1,-2)  arc(180:0:0.5 and 0.4);
\draw (3,0) -- (3,-2);
\draw[dashed] (1.5,-0.5) -- (1.5,-1.5);
\bbbotbox{-2};
\draw (2,-2)  arc (-180:0:0.5 and 0.4) ;
\draw (2,-4)  arc (180:0:0.5 and 0.4) ;
\draw[dashed] (2.5,-2.5) -- (2.5,-3.5);
\end{tikzpicture}
\end{center}
\caption{Impermissible configurations for a simple representation.}\label{notFC}
\end{figure}

\begin{proof}
Let $\w$ be a reduced expression for $w\in\FC(D_n)$. If $\w$ has either configuration in Figure~\ref{notFC}, then $\w$ has $s_{i+1}s_{i}s_{i+1}$ as a subword. Hence, $w$ is not fully commutative, which is a contradiction. 
\end{proof}

\begin{theorem}\label{index}
The type I and type II diagrams are indexed by the type I and type II heaps, respectively.
\end{theorem}

\begin{proof}
Let $\w=s_{x_{1}}\cdots s_{x_{n}}$ be a reduced expression for $w\in\FC(D_n)$. Consider the diagram $d_{w}=d_{x_{1}}\cdots d_{x_{n}}$. If $s_{1}s_{\overline{1}}$ is a subword of some reduced expression for $w$, then it is obvious that $d_{w}$  is a type I diagram. Now assume $d_{w}$ is a type I diagram. Clearly, $s_{\overline{1}}\in$ supp$(w)$ since $d_{\overline{1}}$ is the only simple diagram that contains decorations. It is also clear that $d_{w}$ contains only one loop, and hence, if $s_{1}s_{\overline{1}}$ is a subword of $\w$, then it only appears once. 
For $d_{w}$ to be a type I diagram, $d_{w}$ must contain one decorated loop. Consider the occurence of $s_{\overline{1}}$ involved in the loop. Without loss of generality, assume this occurence of $s_{\overline{1}}$ appears on the ``top" of the loop in the simple representation for $d_{\w}$. Since the configurations in Figure~\ref{notFC} cannot happen, there is no way for the loop edge to wander through the simple representation unless the portion of the diagram given in Figure~\ref{1bar1} appears in the simple representation $d(w)$.
\end{proof}

\begin{figure}[h]
\centering
\begin{tikzpicture}[scale=1]
\tbox{2};
\draw (1,2) arc (-180:0:0.5 and 0.4);
\draw[fill=cyan, draw=white]{(1.5,1.6) circle (2.8pt)};
\draw (1,0) arc (180:0:0.5 and 0.4);
\draw[fill=cyan, draw=white]{(1.5,0.4) circle (2.8pt)};
\draw[dashed] (1.5,1.5) -- (1.5,0.5);
\node[above,scale=0.8] at (1,2){$1$};
\node[above,scale=0.8] at (2,2){$2$};
\botbox{0};
\draw (1,0)  arc (-180:0:0.5 and 0.4) ;
\draw (1,-2)  arc (180:0:0.5 and 0.4) ;
\draw[dashed] (1.5,-0.5) -- (1.5,-1.5);
\end{tikzpicture}
\caption{Portion of the simple representation from Theorem~\ref{index}.}
\label{1bar1}
\end{figure}
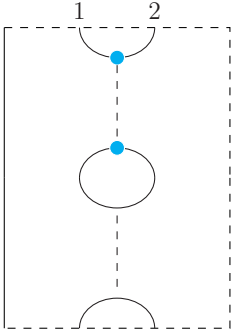

The following theorem follows quickly from Proposition~\ref{prop:surjective homo} and Theorem~\ref{index}.

\begin{theorem}
The diagram algebra, $\LFD(D_{n})$, is isomorphic to $\PFTL(D_{n})$ under $\phi$ as defined in Proposition~\ref{prop:surjective homo}. Moreover, the image of the monomial basis for $\PFTL(D_{n})$ is in natural bijection with the image of the set of type II $D$-admissible diagrams. \qed
\end{theorem}

\end{section}


\section{Cellular algebras} \label{sec:cellular}


Cellular algebras were introduced by Graham and Lehrer~\cite{Graham1996a}, and are a class of finite dimensional associative algebras defined in terms of a ``cell datum" and three axioms. The axioms allow one to define a set of modules for the algebra known as cell modules, and one of the main strengths of the theory is that it is relatively straightforward to construct and to classify the irreducible modules for a cellular algebra in terms of quotients of the cell modules. 

Let $d$ be a $D$-admissible diagram of type II for $\LFD(D_n)$.  Remove all of the propagating edges from $d$, then take the upper half and call it $\overline{d}$. Invert the lower half of $d$ in a horizontal line and call this $\underline{d}$. We call $\overline{d}$ and $\underline{d}$ \emph{half-diagrams}. Then $d$ can be reconstituted from the ordered pair $(\overline{d},\underline{d})$, written as $d=\overline{d}\circ \underline{d}$, by inserting the appropriate propagating edges. Note that if $d$ has any progagating edges, then whether the leftmost propagating edge is decorated is uniquely determined since we know that the total number of decorations is even. If $h$ is a half-diagram, then we define $\a(\underline{d})$ and $\mathbf{p}(\underline{d})$ in the obvious way.

\begin{example}
Consider the two half-diagrams 
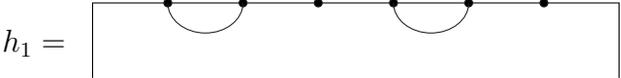
\begin{center}
$h_1=$
\begin{tabular}[c]{l}
\begin{tikzpicture}[scale=1]
\dprimebox{0};
\draw (1,0) arc (-180:0:0.5 and 0.4);
\draw (4,0) arc (-180:0:0.5 and 0.4);
\end{tikzpicture}
\end{tabular}
\end{center}
and
\begin{center}
$h_2=$
\begin{tabular}[c]{l}
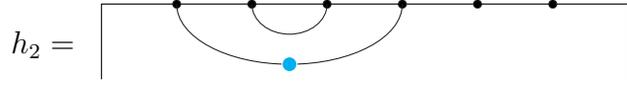
\begin{tikzpicture}[scale=1]
\dprimebox{0};
\draw (2,0)  arc (-180:0:0.5 and 0.4) ;
\draw (1,0)  arc (-180:0:1.5 and 0.8) ;
\draw[fill=cyan, draw=white]{(2.5,-0.8) circle (2.8pt)};
\end{tikzpicture},
\end{tabular}
\end{center}
then 
\begin{center}
$h_1\circ h_2=$
\begin{tabular}[c]{l}
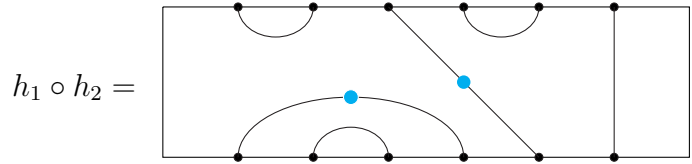
\begin{tikzpicture}[scale=1]
\sixbox{0};
\draw (1,0)  arc (-180:0:0.5 and 0.4) ;
\draw (4,0)  arc (-180:0:0.5 and 0.4) ;
\draw (2,-2)  arc (180:0:0.5 and 0.4) ;
\draw (1,-2)  arc (180:0:1.5 and 0.8) ;
\draw (3,0) -- (5,-2)
	node[fill=cyan,  pos=0.5, shape=circle, inner sep=1.8pt, minimum size=2pt]{};
\draw[fill=cyan, draw=white]{(2.5,-1.2) circle (2.8pt)};
\draw (6,0) -- (6,-2); 
\end{tikzpicture}
\end{tabular}
\end{center}
\end{example}

We define $h$ to be a sub-half-diagram of $h'$, as shown in Figure~\ref{subhalf}, if all non-propagating edges of $h'$ are non-propagating edges of $h$. In this case, we write $h\leq h'$.

\begin{figure}[!ht]
\centering
\begin{tabular}[c]{l}
\begin{tikzpicture}[scale=0.7]
\dprimebox{0};
\draw (2,0)  arc (-180:0:0.5 and 0.4) ;
\draw (1,0)  arc (-180:0:1.5 and 0.8) ;
\draw[fill=cyan, draw=white]{(2.5,-0.8) circle (3.6pt)};
\end{tikzpicture}
\end{tabular}
$\leq$
\begin{tabular}[c]{l}
\begin{tikzpicture}[scale=0.7]
\dprimebox{0};
\draw (2,0)  arc (-180:0:0.5 and 0.4) ;
\draw (1,0)  arc (-180:0:1.5 and 0.8) ;
\draw[fill=cyan, draw=white]{(2.5,-0.8) circle (3.6pt)};
\draw (5,0) arc (-180:0:0.5 and 0.4);
\end{tikzpicture}
\end{tabular}
\caption{Example of a subhalf-diagram.}\label{subhalf}
\end{figure}
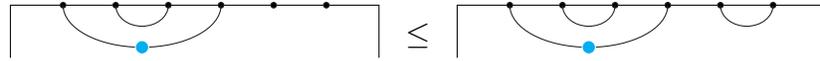  
 
The following definition of cellular algebra comes from~\cite{Graham1996a}.

\begin{definition}
Let $R$ be a commutative ring with identity. A \emph{cellular algebra} over $R$ is an associative unital algebra, $A$, together with a cell datum $(\Lambda,M,C,*)$ where
\begin{enumerate}
\item $\Lambda$ is a poset. For each $\lambda\in\Lambda,M(\lambda)$ is a finite set such that 
\[
C:\coprod_{\lambda\in\Lambda}(M(\lambda)\times M(\lambda))\rightarrow A
\]
is injective with image equal to an $R$-basis of $A$.
\item If $\lambda\in\Lambda$ and $S,T\in M(\lambda)$, we write $C(S,T)=C_{S,T}^{\lambda}\in A$. Then $*$ is an $R$-linear involutory anti-automorphism of $A$ such that $(C_{S,T}^{\lambda})^*=C_{T,S}^{\lambda}$.
\item If $\lambda\in\Lambda$ and $S,T\in M(\lambda)$, then for all $a\in A$ we have
\[
aC_{S,T}^{\lambda}\equiv \sum_{S'\in M(\lambda)}r_a(S',S)C_{S',T}^{\lambda}\mod A(<\lambda),
\]
where $r_a(S',S)\in R$ is independent of $T$ and $A(<\lambda)$ is the $R$-submodule of $A$ generated by the set 
\[
\{C_{S'',T''}^{\mu}:\mu <\lambda,S''\in M(\mu),T''\in M(\mu)\}.
\]
\end{enumerate}
\end{definition}

\begin{example}
Let $S_{n}$ be the symmetric group on $n$ letters. Then the group algebra $\Z S_n$ is cellular over $\Z$. In this case, the poset $\Lambda$ is the set of partitions of $n$, ordered by dominance (meaning that if $\lambda\trianglerighteq\mu$, then $\lambda\leq\mu$). The set $M(\lambda)$ is the set of standard tableaux of shape $\lambda$, namely the ways of writing the numbers $1,2,\ldots,n$ once each into a Young diagram of shape $\lambda$ such that the entries increase along rows and down columns. The element $C^{\lambda}_{S,T}$ is the Kazhdan--Lusztig basis element $C'_w$ such that $w\in S_n$ corresponds via the Robinson--Schensted correspondence to the ordered pair of standard tableaux $(S,T)$. The map $*$ sends $C'_w$ to $C'_{w^{-1}}$. For details on standard tableaux, Young diagrams, and the Robinson--Schensted correspondence we refer the reader to~\cite[Chapter 2]{Sagan2001}.
\end{example}

The Hecke algebra $\H(A_n)$ was shown to be cellular by Graham and Lehrer in~\cite[Example 1.2]{Graham1996a}, and the underlying idea was already implicit in~\cite{Kazhdan1979}. The example of the symmetric group above is obtained simply by specializing $q$ to 1, as was observed by Graham and Lehrer in their treatment of the Brauer algebra~\cite{Graham1996a}. 

We will now construct the cell datum $(\Lambda, M, C, *)$ for $\LFD(D_n)$. Let $\Lambda$ be the set of symbols $\{1,3,5,\ldots,n\}$ when $n$ is odd and $\{0^{+},0^{-},2,4,\ldots,n\}$ when $n$ is even. We put a partial order $<$ on these symbols by declaring that $i<j$ if $\left|i\right|<\left|j\right|$, where $\left|i\right|=i$ if $i$ is a natural number, and $\left|0^{+}\right|=\left|0^{-}\right|=0$. The Hasse diagrams for the posets $\Lambda$ are shown in Figure~\ref{poset}.

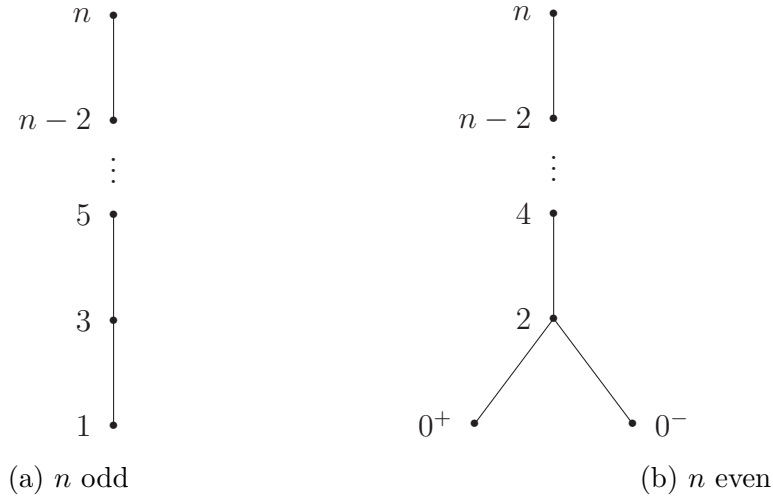
\begin{figure}[!ht]
\centering
\begin{subfigure}[b]{0.48\textwidth}
\centering
\begin{tikzpicture}[scale=0.7]
\node[scale=0.6, label=left:$n$] at (1,5.8) {$\bullet$};
\node[scale=0.6, label=left:$n-2$] at (1,3.8) {$\bullet$};
\node[scale=0.6, label=left:$5$] at (1,2) {$\bullet$};
\node[scale=0.6, label=left:$3$] at (1,0) {$\bullet$};
\node[scale=0.6, label=left:$1$] at (1,-2) {$\bullet$};
\draw (1,5.8)--(1,3.8);
\node at (1,3){$\vdots$};
\draw (1,2)--(1,0)--(1,-2);
\end{tikzpicture}
\caption{$n$ odd}
\end{subfigure}
\quad
\begin{subfigure}[b]{0.48\textwidth}
\begin{tikzpicture}[scale=0.7]
\node[scale=0.6, label=left:$n$] at (1,5.8) {$\bullet$};
\node[scale=0.6, label=left:$n-2$] at (1,3.8) {$\bullet$};
\node[scale=0.6, label=left:$4$] at (1,2) {$\bullet$};
\node[scale=0.6, label=left:$2$] at (1,0) {$\bullet$};
\node[scale=0.6, label=left:$0^{+}$] at (-0.5,-2) {$\bullet$};
\node[scale=0.6, label=right:$0^{-}$] at (2.5,-2) {$\bullet$};
\draw (1,5.8)--(1,3.8);
\node at (1,3){$\vdots$};
\draw (1,2)--(1,0)--(-0.5,-2);
\draw (1,0)--(2.5,-2); 
\end{tikzpicture}
\caption{$n$ even}
\end{subfigure}
\caption{Hasse diagrams for $\Lambda$.}
\label{poset}
\end{figure}

If $\lambda\in\Lambda$, the set $M(\lambda)$ has elements parametrised by the half-diagrams $h$ arising from $D$-admissible diagrams of type II with $\mathbf{p}\left(h\right)=\left|\lambda\right|$. If $\mathbf{p}(d)=0$, then the diagram $d$ has to be reconstructed from two half-diagrams with the same parity. Hence, a half-diagram $h$ with no propagating edges will have the symbol $0^{+}$ if $h$ has an even number of decorations and $0^{-}$ if $h$ has an odd number of decorations.
The anti-automorphism $*$ corresponds to top-bottom inversion of a $D$-admissible diagram of type II.
The map $C$ takes elements $h_1$ and $h_2$ from $M(\lambda)$ and produces the element $C(h_1,h_2)$ which is defined to be $h_1\circ h_2.$

Note that the identity element appears in the image of $C$.

\begin{lemma}\label{axiomthree}
Let $\lambda\in\Lambda$ and $h_1,h_2\in M(\lambda)$. If $d=h_1\circ h_2$, then for all simple diagrams $d_i$, we have
\[
d_id=\epsilon d'
\]
for some $d'$ where $\epsilon\in\{0,1,\delta\}$ and $h_2\leq\underline{d'}$.
\end{lemma}

\begin{proof}
Since the multiplication of diagrams $d_id$ preserves the non-propagating edges in the north face of $d_i$ and the south face of $d$, $\underline{d'}$ must have at least the same non-propagating edges as $h_2$. So,  $h_2\leq\underline{d'}$. 
If $h_1$ has a non-propagating edge from node $i$ to node $i+1$ decorated with $x\in\{\emptyset,\bcirc\}$ (where $\emptyset$ denotes that the edge is undecorated), then
\[
\epsilon=\begin{cases} 0, &\mbox{ if }x=\bcirc \mbox{ and }i\neq\overline{1}\\ \delta, &\mbox{ otherwise.}\end{cases}
\]
If $h_1$ does not have a non-propagating edge from node $i$ to $i+1$, then $\epsilon=1$.
\end{proof}

\begin{remark}
The multiple $\epsilon$ does not depend on $h_2$ and $\mathbf{p}(d')\leq\mathbf{p}(d)$.
\end{remark}

Let $\LFD(D_n)\left(<\lambda\right)$ be the submodule generated by diagrams with $\mathbf{p}$-values less than $|\lambda|$. Then the following lemma holds.

\begin{lemma}\label{lessnonprop}
Let $d,d'\in\LFD(D_n)$. If $\mathbf{p}(d')<\mathbf{p}(d)$, then $d'\in \LFD(D_n)\left(<\lambda\right)$, where $|\lambda|=\mathbf{p}(d)$. \qed
\end{lemma}

\begin{theorem}\label{loopfreecellular}
The algebra $\LFD(D_n)$ over the ring $\Z[\delta]$ has a cell datum $(\Lambda,M,C,*)$, where the sets are given as above.
\end{theorem}

\begin{proof}
The proof is largely straightforward. The only nontrivial part is the verification of axiom 3. Since $\LFD(D_n)$ is generated by the simple diagrams, it is enough to multiply $d$ by a simple diagram and hence axiom 3 follows quickly from Lemma~\ref{axiomthree} and Lemma~\ref{lessnonprop}.
\end{proof}

Since $\LFD(D_n)\cong\PFTL(D_n)$ as $\Z[\delta]$-algebras, the following corollary is immediate.

\begin{corollary}\label{pairfreecellular}
The algebra $\PFTL(D_n)$ over the ring $\Z[\delta]$ is cellular. \qed
\end{corollary}


\section{Future work}


We have shown that $\PFTL(D_n)$ is cellular by explicitly constructing a cell datum for $\LFD(D_n)$. However, it remains to describe the corresponding cell datum for $\PFTL(D_n)$, which may or may not be enlightening.  As an application, we could utilize the structure of the algebra to quickly compute $\mu$-values (see Section~\ref{hecke}) for pairs of elements that index the basis of $\PFTL(D_n)$. 

We have shown that a quotient of $\TL(D_n)$ is cellular, but we are uncertain if $\TL(D_n)$ itself is cellular. This is likely known, but we are unable to find a reference.

Lastly, it is known that there is a connection between $\PFTL(D_n)$ and the so-called blob algebra~\cite{Martin1994}. Future work could include making this connection more explicit.


\bibliographystyle{plain}
\bibliography{References}

\begin{thebibliography}{10}

\bibitem{Billey2007}
S.C. Billey and B.C. Jones.
\newblock {Embedded factor patterns for Deodhar elements in Kazhdan--Lusztig
  theory}.
\newblock {\em Ann. Comb.}, 11(3--4):285--333, 2007.

\bibitem{Ernst2010a}
D.C. Ernst.
\newblock {Non-cancellable elements in type affine $C$ Coxeter groups}.
\newblock {\em Int. Electron. J. Algebra}, 8:191--218, 2010.

\bibitem{Ernst2012}
D.C. Ernst.
\newblock {Diagram calculus for a type affine $C$ Temperley--Lieb algebra, I}.
\newblock {\em J. Pure Appl. Alg. {\rm (to appear)}}, 2012.

\bibitem{Fan1997}
C.K. Fan.
\newblock {Structure of a Hecke algebra quotient}.
\newblock {\em J. Amer. Math. Soc.}, 10:139--167, 1997.

\bibitem{Geck2000}
M.~Geck and G.~Pfeiffer.
\newblock {\em {Characters of finite Coxeter groups and Iwahori--Hecke
  algebras}}.
\newblock 2000.

\bibitem{Graham1995}
J.J. Graham.
\newblock {\em {Modular representations of Hecke algebras and related
  algebras}}.
\newblock PhD thesis, 1995.

\bibitem{Graham1996a}
J.J. Graham and G.I. Lehrer.
\newblock {Cellular algebras}, 1996.

\bibitem{Green1998}
R.M. Green.
\newblock {Cellular algebras arising from Hecke algebras of type $H_n$}.
\newblock {\em Math. Zeit.}, 229:365--383, 1998.

\bibitem{Green1998a}
R.M. Green.
\newblock {Generalized Temperley--Lieb algebras and decorated tangles}.
\newblock {\em J. Knot Th. Ram.}, 7:155--171, 1998.

\bibitem{Green2006a}
R.M. Green.
\newblock {Star reducible Coxeter groups}.
\newblock {\em Glasgow Math. J.}, 48:583--609, 2006.

\bibitem{Green1999}
R.M. Green and J~Losonczy.
\newblock {Canonical bases for Hecke algebra quotients}.
\newblock {\em Math. Res. Lett.}, 6:213--222, 1999.

\bibitem{Humphreys.J:A}
J.E. Humphreys.
\newblock {\em {Reflection Groups and Coxeter Groups}}.
\newblock 1990.

\bibitem{Kauffman1987}
L.H. Kauffman.
\newblock {State models and the Jones polynomial}.
\newblock {\em Topology}, 26:395--407, 1987.

\bibitem{Kazhdan1979}
D.~Kazhdan and G.~Lusztig.
\newblock {Representations of Coxeter groups and Hecke algebras}.
\newblock {\em Invent. Math.}, 53:165--184, 1979.

\bibitem{Martin1994}
P.P. Martin and H.~Saleur.
\newblock {The blob algebra and the periodic Temperley--Lieb algebra}.
\newblock {\em Lett. Math. Phys.}, 30 (3):189--206, 1994.

\bibitem{Penrose1971}
R.~Penrose.
\newblock {Angular momentum: An approach to combinatorial space-time}.
\newblock In {\em Proceedings}, pages 151--180, 1971.

\bibitem{Sagan2001}
B.E. Sagan.
\newblock {\em {The Symmetric Group: Representations, Combinatorial Algorithms,
  and Symmetric Functions}}.
\newblock Graduate Texts in Mathematics. Springer, 2nd edition, 2001.

\bibitem{Stembridge1996}
J.R. Stembridge.
\newblock {On the fully commutative elements of Coxeter groups}.
\newblock {\em J. Algebraic Combin.}, 5:353--385, 1996.

\end{thebibliography}


\end{document}